\documentclass[11pt,letterpaper]{amsart} 

\usepackage{amssymb,amsmath,amsfonts,amsthm,mathtools}
\usepackage{latexsym,mathrsfs,pb-diagram}
\usepackage[usenames,dvipsnames,svgnames,table]{xcolor} 
\usepackage[pdftex]{graphicx}
\usepackage{pgf,tikz,tikz-cd} 

\usepackage{enumitem} 

\usepackage[plainpages=false,colorlinks,pdfpagelabels]{hyperref} 
\hypersetup{
  urlcolor=blue,
  citecolor=green,
  linkcolor=blue
}

\numberwithin{equation}{section}

\theoremstyle{definition}
\newtheorem{Def}{Definition}[section]

\theoremstyle{remark}

\newtheorem{Rem}[Def]{Remark}

\theoremstyle{plain}
\newtheorem{Prop}[Def]{Proposition}
\newtheorem{Cor}[Def]{Corollary}
\newtheorem{Thm}[Def]{Theorem}
\newtheorem{Lem}[Def]{Lemma}

\newcommand{\dfn}{\doteq}
\newcommand{\st}{ \ ; \ }
\newcommand{\lra}{\longrightarrow}
\newcommand{\sset}{\subset}
\newcommand{\bb}{\bullet}
\DeclareMathOperator{\pr}{pr}

\newcommand{\Z}{\mathbb{Z}}
\newcommand{\N}{\mathbb{N}}
\newcommand{\R}{\mathbb{R}}
\newcommand{\Q}{\mathbb{Q}}
\newcommand{\C}{\mathbb{C}}

\newcommand{\TR}[5]{
  \begin{array}{c c c c c}
    {#1} & : & {#3} & \longrightarrow & {#5} \\
         &   & {#2} & \longmapsto     & {#4}
  \end{array}
}

\newcommand{\transp}[1]{\prescript{\mathrm{t} \!}{}{{#1}}}
\DeclareMathOperator{\Span}{\mathrm{span}}
\DeclareMathOperator{\ran}{\mathrm{ran}}
\DeclareMathOperator{\Id}{\mathrm{Id}}

\DeclareMathOperator{\Aut}{\mathrm{Aut}}
\DeclareMathOperator{\Deck}{\mathrm{Deck}}

\newcommand{\supp}{\mathrm{supp}}
\newcommand{\del}{\partial}
\newcommand{\dd}{\mathrm{d}}
\let\Re\relax
\let\Im\relax
\DeclareMathOperator{\Re}{\mathsf{Re}}
\DeclareMathOperator{\Im}{\mathsf{Im}}
\newcommand{\dR}{\mathrm{dR}}

\newcommand{\D}{\mathscr{D}}
\newcommand{\E}{\mathscr{E}}
\newcommand{\cinfty}{\mathscr{C}^\infty}
\newcommand{\Ol}{\mathcal{O}}
\newcommand{\msf}{\mathsf}

\newcommand{\VV}{\mathcal{V}}
\newcommand{\WW}{\mathcal{W}}
\newcommand{\psum}{\sideset{}{'}\sum}
\newcommand{\NN}{\mathfrak{N}}
\newcommand{\szeta}{\zeta \mkern-13mu {-}}
\newcommand{\per}{\mathcal{P}}

\newcommand{\til}{\widetilde}
\newcommand{\iprod}{\mathbin{\lrcorner}}

\author[Ara\'{u}jo]{Gabriel Ara\'{u}jo}
\address{Universidade de S{\~a}o Paulo, ICMC-USP, S{\~a}o Carlos, SP, Brazil}
\email{\texttt{gccsa@icmc.usp.br}}

\author[Dattori da Silva]{Paulo Leandro Dattori da Silva}
\address{Universidade de S{\~ao} Paulo, ICMC-USP, S{\~a}o Carlos, SP, Brazil}
\email{dattori@icmc.usp.br}

\author[de Lessa Victor]{Bruno de Lessa Victor}
\address{Universidade Federal de Santa Catarina, Brazil}
\email{\texttt{bruno.lessa@ufsc.br}}

\author[Novelli]{Vin\'icius Novelli}
\address{University of Vienna, Vienna, Austria}
\email{vinicius.novelli.da.silva@univie.ac.at}

\title[Global properties of real, non-singular $1$-forms]{Global properties of the differential complex associated to closed, nonsingular $1$-forms on compact manifolds}

\begin{document}

\begin{abstract}
Given a closed, real, non-singular $1$-form on a compact manifold $\Omega$, global properties of the associated differential complex are studied. We completely characterize global solvability in the first and last levels of the complex. Furthermore, in the particular case where the $1$-form is rational, we prove global solvability for every degree and give a complete description of the cohomology spaces. Finally, a complete characterization for global hypoellipticity is obtained, building on the work of A.~Meziani (Comm.~PDE.,~2002). In all cases, it is shown that the conditions depend exclusively on the arithmetic nature of the form's periods.
\end{abstract}
		
\maketitle

\section{Introduction}
\label{sec:intro}

Let $\Omega$ be an $(n+1)$-dimensional smooth manifold\footnote{We assume it is connected, Hausdorff, paracompact and orientable.} and $\zeta \in \cinfty(\Omega; T^* \Omega)$ be a closed, non-singular, real-valued $1$-form. To such a form, one associates the subbundle $\VV \sset \C T \Omega$ whose sections are the complex vector fields annihilated by $\zeta$. Closedness of the latter entails involutivity of $\VV$. In such a situation, we obtain a differential complex~\cite{bch_iis, treves_has}, whose definition we recall.

Let $F^q(\Omega)$ denote the space of smooth, complex-valued $q$-forms on $\Omega$. For $1 \leq q \leq n$, let $\NN^q_{\VV}(\Omega)$ denote the subspace of all $f \in F^q(\Omega)$ such that $f(Y_1, \ldots, Y_q) = 0$ whenever $Y_1, \ldots, Y_q$ are smooth sections of $\VV$. Set $\NN^0_{\VV}(\Omega) \dfn \{0\}$, and define 
\begin{equation*}
  \Lambda^q_\VV (\Omega) \dfn \frac{F^q(\Omega)}{\NN^q_{\VV}(\Omega)},
  \quad 0 \leq q \leq n.
\end{equation*}
We endow $F^{q}(\Omega)$ with the standard Fr{\'e}chet $\cinfty$ topology. Then, $\NN^q_{\VV}(\Omega)$ is a closed subspace of $F^q(\Omega)$, thus $\Lambda^q_\VV(\Omega)$ is a Fr\'echet space for the quotient topology. It is also a $\cinfty(\Omega)$-module. This construction can be localized on open sets $U \sset \Omega$: for every $q$ there is a smooth vector bundle $\Lambda^{q}_\VV \to \Omega$ such that $\Lambda^{q}_\VV(U)$ is its space of smooth sections over $U$.

Since $\VV$ is involutive, we have $\dd \NN^q_{\VV} (\Omega) \sset \NN^{q+1}_{\VV}(\Omega)$ for $0 \leq q \leq n - 1$. Therefore, the exterior derivative $\dd$ induces continuous linear operators
\begin{equation}
   \label{eq:dprime}
   \dd'_q: \Lambda^q_\VV (\Omega) \lra \Lambda^{q+1}_\VV (\Omega)
\end{equation}
that form a differential complex. Using local coordinates, it is plain that these are first-order linear differential operators with smooth coefficients. The cohomology spaces of this complex are denoted by
\begin{equation*}
  H^q_{\mathcal{V}}(\Omega) \dfn
  \frac{ \ker \left\{ \dd'_q: \Lambda^q_\VV(\Omega) \lra \Lambda^{q+1}_\VV (\Omega) \right\} }
  { \ran \left\{ \dd'_{q-1} : \Lambda^{q-1}_\VV (\Omega) \lra \Lambda^q_\VV(\Omega) \right\} },
  \quad 1 \leq q \leq n,
\end{equation*}
whereas $H^0_{\VV}(\Omega) \dfn \ker \left\{ \dd'_0: \cinfty(\Omega) \to \Lambda^1_\VV (\Omega) \right\}$. By Frobenius' Theorem, $\VV$ defines a (non-singular smooth) foliation $\mathcal{F}$ of $\Omega$ by smooth hypersurfaces, and $\dd'$ can be thought of as the exterior derivative ``along the leaves'' of $\mathcal{F}$. Accordingly, these are often called \emph{tangential foliated} (or \emph{leafwise}) cohomology spaces (see, for instance,~\cite{Asaoka2014, Moore1988} and references therein).

From now on, we assume that $\Omega$ is compact. Concerning the complex~\eqref{eq:dprime}, we are interested in global questions of analytic nature. The first one is \emph{global solvability} of $\dd'_q$: in this work, this property means that its range is a \emph{closed} subspace of $\Lambda^{q+1}_{\VV}(\Omega)$. As is well-known, this property is equivalent to smooth solvability of the equation
\begin{equation*}
  \dd'_{q} \pmb{u} = \pmb{f},
\end{equation*}
for right-hand sides $\pmb{f} \in \Lambda^{q+1}_{\VV}(\Omega)$ satisfying appropriate compatibility conditions\footnote{Namely, the condition is that $\pmb{f}$ belongs to the annihilator of the kernel of the transpose operator $\transp{\dd}'_q$, as the Hahn-Banach theorem easily implies. See Section~\ref{sec:irr_degn} for more details.}. In that case, we shall say $\zeta$ is \emph{globally solvable} in degree $q+1$.

When $\Omega$ is of the form $M \times S^1$, for $M$ a compact $n$-dimensional manifold, an interesting class is given by $\zeta \dfn \omega - \dd \theta$, where $\omega \in F^1(M)$ is a closed, real-valued $1$-form on $M$ and $\dd\theta$ is the standard volume element in $S^1$: structures of this kind are called \emph{real tube-type structures}, and their global solvability has been studied by several authors (see, e.g.,~\cite{araujo_jahnke_ragog_ferra24, bcp96, dm16}). Characterizations are known in the first and last levels of the complex, and rely mainly on Fourier-analytic techniques, exploiting crucially the product structure with $S^1$.

In this paper, the absence of a product structure (or any circular symmetry) requires the introduction of new techniques. In common with previous works, our characterizations of global solvability rely on certain arithmetic properties of $\zeta$ (see below for precise definitions). In the so-called \emph{rational} case (Definition~\ref{def:rat_irr_integ}), we prove solvability in every level of the complex without further hypotheses. In the \emph{irrational} case, on the other hand, we obtain complete characterizations for solvability in the first and last levels of the complex. In particular, we obtain \emph{smooth} global solvability in \emph{top degree}. It is worth mentioning that the only results  which were previously known in that direction, even for tube-type situations, were the particular where the manifold is the torus $\mathbb{T}^{m + n}$~\cite{dm16}  and the weaker result where \emph{distributional} solutions are obtained for smooth right-hand sides~\cite{bcp96}.

In order to describe our results, we need to define a crucial object: the \emph{group of periods} of $\zeta$ is the (additive) subgroup of $\R$ given by
\begin{equation*}
  \per(\zeta) \dfn \left\{
    \int_\gamma \zeta \st [\gamma] \in H_1(\Omega; \Z)
  \right\}.
\end{equation*}
\begin{Def}
  \label{def:rat_irr_integ}
  A closed $1$-form $\zeta$ is \emph{integral} if $\per(\zeta) = \Z$. More generally, $\zeta$ is \emph{rational} if $\per(\zeta)$ is infinite cyclic; and \emph{irrational} otherwise\footnote{In which case $\per(\zeta)$ is everywhere dense~\cite[Lemma~9.3.13]{conlon_manifolds}. It can never be zero, since in that case $\zeta$ would be exact (by de Rham's Theorem), which cannot be since $\zeta$ is by assumption non-singular, whereas any globally defined smooth function must have a critical point by compactness.}.
\end{Def}

Our first result is a complete description of $H^{q}_{\VV}(\Omega)$ when $\zeta$ is rational.
\begin{Thm}
  \label{thm:rational_mainthm}
  Let $\Omega$ be a compact $(n+1)$-dimensional manifold and $\zeta \in F^1(\Omega)$ be a closed, non-singular, real-valued $1$-form. Assume that $\zeta$ is rational. Then, there is a compact hypersurface $L \subset \Omega$ such that, for every $0 \leq q \leq n$, there are a smooth vector bundle $\mathbb{H}^q_{\dR} (L)$ over the circle, whose fibers are naturally isomorphic to $H^{q}_{\dR}(L)$, and a topological isomorphism
  \begin{equation*}
    H^{q}_{\VV}(\Omega) \simeq \cinfty(S^1; \mathbb{H}^q_{\dR} (L))
  \end{equation*}
-- the space of smooth sections of $\mathbb{H}^q_{\dR} (L) \to S^1$, endowed with the standard Fr{\'e}chet topology. In particular, $\zeta$ is globally solvable in every degree.
\end{Thm}

The proof of Theorem~\ref{thm:rational_mainthm} is presented in Section~\ref{sec:rational}. The main idea is the following: when $\zeta$ is rational, the foliation $\mathcal{F}$ defined by $\zeta$ -- of which $L$ is a leaf -- is actually given by a locally trivial fibration, so the analysis reduces to a gluing of two partial de Rham complexes. Even though the solvability statement can be deduced from~\cite[Theorem~2.2]{cn24}, we present a self-contained argument which leads to the description of the vector bundle $\mathbb{H}^q_{\dR}(L)$. 

In the irrational case, we tackle solvability in levels $1$ and $n$. These statements depend on the ``degree of irrationality'' of the form $\zeta$, encoded as Diophantine approximations.
\begin{Def}
  \label{def:liouville}
  An irrational $1$-form $\zeta$ is said to be \emph{Liouville} if there are a sequence of integral $1$-forms $\{ \zeta_\nu \}_{\nu \in \N}$ and a sequence of integers $\{ q_\nu \}_{\nu \in \N}$, with $q_\nu \geq 2$, such that
  \begin{equation*}
    \left\{ q_\nu^\nu (\zeta - q_\nu^{-1} \zeta_\nu) \st \nu \in \N \right\}
    \quad \text{is a bounded set in $F^1 (\Omega)$}.
  \end{equation*}
\end{Def}
Our main result reads as follows:
\begin{Thm}
  \label{thm:degs1n_mainthm}
  Let $\Omega$ be a compact $(n+1)$-dimensional manifold and $\zeta \in F^1(\Omega)$ be a closed, non-singular, real-valued $1$-form. The following are equivalent:
  \begin{enumerate}
  \item \label{thm:degs1n_mainthm-it1} $\zeta$ is globally solvable in degree $1$.
  \item \label{thm:degs1n_mainthm-it2} $\zeta$ is globally solvable in degree $n$.
  \item \label{thm:degs1n_mainthm-it3} $\zeta$ is either rational or an irrational non-Liouville form.
  \end{enumerate}
\end{Thm}

Its proof is given in Sections~\ref{sec:irr_deg1} and~\ref{sec:irr_degn}. Let us describe the main ideas involved. When $\zeta$ is irrational the leaves of $\mathcal{F}$ are dense, which complicates the geometry. We first observe~\cite[Section~1.4]{godbillon} that there is a minimal covering space $\Pi: \widehat{\Omega}\to \Omega$ such that the pullback $\Pi^*\zeta$ is exact (the application of minimal coverings to such questions was introduced in the paper \cite{BKNZ}). One can show that $\widehat{\Omega} \simeq \R \times L$, where $L$ is a fixed (non-compact) leaf of $\mathcal{F}$. Moreover, the group of periods $\per(\zeta)$ acts in $\R \times L$ via translation in the first coordinate and the flow of a transversal vector field\footnote{In our techniques, the dynamics of a vector field $X$ \emph{transversal} to $\zeta$ plays a key role. This means that $X$ is globally defined on $\Omega$ and satisfies $\zeta(X) = 1$ everywhere.} in the second. 

Let $\mathscr{A}$ be the sheaf of smooth solutions of $\dd'_0$ in $\Omega$ and $\hat{\mathscr{A}}$ the sheaf of smooth solutions of $\dd_{L}$ (the partial de Rham operator in the leaf direction) in $\R \times L$. Since $\per(\zeta)$ acts on $\R \times L$, the sheaf $\hat{\mathscr{A}}$ is a so-called \emph{$\per(\zeta)$-sheaf}~\cite[Chapitre~5]{gro57}), and $\hat{\mathscr{A}}^{\per(\zeta)}$ -- the \emph{sheaf of invariants} of $\Pi_* \hat{\mathscr{A}}$ -- is isomorphic to $\mathscr{A}$ (see Proposition~\ref{prop:pullback_map}). Then, the Cartan-Leray spectral sequence degenerates and yields an isomorphism between $H^q(\Omega;\mathscr{A}) \simeq H^q_\VV(\Omega)$ and a direct sum of \emph{group cohomology spaces} associated with $\hat{\mathscr{A}}$. This technique has been used previously by Kacimi-Alaoui \cite{Alaoui2025} in the general setting of developable foliations and their connections with deformation theory. 

However, in our context we are interested in topological isomorphisms (and not merely algebraic ones which come from the abstract theory), which are much more difficult to exhibit, thereby requiring us to work directly with the double complex that yields the Cartan-Leray spectral sequence (see Theorem~\ref{thm:abstract_eq}).
 
The problem is then reduced to determining when the differential of group cohomology has closed range, i.e., when does the vector-valued operator
\begin{equation*}
  f \longmapsto \left( f(\cdot + a_j) - f \right)_{j = 0,1, \ldots, r}
\end{equation*}
has closed range (acting on $\cinfty$ functions on $\R$), where $\{ a_0, a_1, \ldots, a_r\}$ are the generators of $\per(\zeta)$ (see Corollaries \ref{cor:abstract_eq} and \ref{cor:reduction_to_generators}). This is accomplished, via the Homomorphism Theorem for Fr{\'e}chet spaces~\cite[p.~18]{{kothe_tvs2}}, by studying the transpose operator and solving a convolution-type equation on compactly supported distributions. Here, a classical result due to L.~Ehrenpreis (Theorem~\ref{thm:ehrenpreis}) is decisive. We remark that this technique was also applied recently in \cite{christensen2024} for problems related to the wave equation.

As for $\eqref{thm:degs1n_mainthm-it2} \Longleftrightarrow \eqref{thm:degs1n_mainthm-it3}$, we proceed by duality. It is enough to prove that $\dd'_0$ has strongly closed range when acting on distributions, rather than smooth functions. The proof proceeds in a similar fashion, but the necessity of dealing with distributions and currents in a non-compact setting requires a much more delicate manipulation; their functional analysis is also more intricate.

The second property that concerns us is global hypoellipticity of $\dd'_0$. That is, the property that
\begin{equation}
  \label{eq:GH}
  \tag{GH}
  \text{any $u \in \D'(\Omega)$ such that $\dd'_0 u$ is smooth is itself smooth.}
\end{equation}
As $\dd'_0$ is a differential operator, \eqref{eq:GH} entails smooth global solvability in degree $1$~\cite[Theorem~3.5]{afr22}. Again, in the real tube-type case, this property is well understood~\cite{araujo_lessa_datt_22, bcm93}, but in our current situation the usual Fourier-analytic techniques break down, and it becomes necessary to introduce different methods to approach this problem.

This was first addressed in~\cite{meziani02}, where a general necessary condition for~\eqref{eq:GH} was established~\cite[Theorem~3.1]{meziani02}. There, however, the sufficiency of the said condition was shown under an additional hypothesis of periodicity of the suspension diffeomorphism that defines the fibration $\Omega \to S^1$~\cite[Theorem~4.1]{meziani02}, which allows one to essentially return to the tube-type case.

In our final result we solve that issue, by proving the sufficiency of the condition for~\eqref{eq:GH} without any extraneous hypothesis. That is:
\begin{Thm}
  \label{thm:hypo_mainthm}
  Let $\Omega$ be a compact $(n+1)$-dimensional manifold and $\zeta \in F^1(\Omega)$ a closed, non-singular, real-valued $1$-form. The following are equivalent:
  \begin{enumerate}
  \item \label{thm:hypo_mainthm-it1} $\dd'_0$ is~\eqref{eq:GH}.
  \item \label{thm:hypo_mainthm-it2} $\zeta$ is irrational non-Liouville.
  \end{enumerate}
\end{Thm}
As we mentioned, the implication $\eqref{thm:hypo_mainthm-it1} \Longrightarrow \eqref{thm:hypo_mainthm-it2}$ is proven in \cite[Theorem~3.1]{meziani02} by explicitly producing non-smooth distributions in the kernel of $\dd'_0$. The converse $\eqref{thm:hypo_mainthm-it2} \Longrightarrow \eqref{thm:hypo_mainthm-it1}$, carried out in Section~\ref{sec:hypo}, uses the same formalism developed to prove solvability, and shows that this property is equivalent to global hypoellipticity of the group cohomology differential in the first level (see Theorem~\ref{thm:group_coh_hypo}). Once we are in this situation, we can apply the fundamental solution obtained by Ehrenpreis's technique to conclude the desired hypoellipticity.

\section{The geometry of integral forms}
\label{sec:rational}

We devote this section to the proof of Theorem~\ref{thm:rational_mainthm}. Very little knowledge of foliation theory is required, and although many of the constructions below are well-known, it will cost us little to present them in a (mostly) self-contained manner; it will also help us to set the notation straight. We refer the reader to~\cite[Section~9.3]{conlon_manifolds} whenever necessary.

Assume that $\zeta \in F^1(\Omega)$ is a rational $1$-form; this means (Definition~\ref{def:rat_irr_integ}) that $\per(\zeta) = c \Z$ for some non-zero $c \in \R$. Since the definition of $\VV$ is invariant by multiplication of $\zeta$ by a non-vanishing constant, we may assume without loss of generality that $c = 1$, that is, that $\zeta$ is integral to start with. For simplicity, we adopt throughout the normalization
\begin{equation*}
  \szeta \dfn 2\pi \zeta.
\end{equation*}
Notice that the annihilator bundle of $\szeta$ is exactly $\VV$, whereas $\per(\szeta) = 2 \pi \Z$. If we denote by $\Pi: \widetilde{\Omega} \to \Omega$ the universal covering space of $\Omega$, and take $\psi \in \cinfty(\widetilde{\Omega}; \R)$ such that $\dd \psi = \Pi^* \szeta$, then~\cite[Lemma~2.3]{bcm93}
\begin{equation*}
  \forall p, p' \in \widetilde{\Omega}, \ \Pi(p) = \Pi(p')
  \Longrightarrow
  \psi(p) - \psi(p') \in 2 \pi \Z,
\end{equation*}
which is precisely the property required for the existence of a $G \in \cinfty(\Omega)$ such that $e^{i \psi} = \Pi^* G$. Hence,
\begin{equation*}
  \Pi^* (\dd G) = \dd \Pi^* G = \dd e^{i \psi} = i e^{i \psi} \dd \psi = i e^{i \psi} \Pi^* \szeta = \Pi^* (i G \szeta)
  \quad \text{on $\widetilde{\Omega}$},
\end{equation*}
and since $\Pi$ is a submersion we conclude that
\begin{equation*}
  \dd G = i G \szeta \quad \text{on $\Omega$}.
\end{equation*}

Notice, moreover, that $|G(x)| = 1$ for every $x \in \Omega$. Hence, there exists a smooth map $g: \Omega \to S^1$ -- which we call the \emph{fibration associated with $\zeta$} -- such that $G = \iota \circ g$, where $\iota: S^1 \hookrightarrow \C$ is the inclusion map. It follows that
\begin{equation}
  \label{eq:gdthetaiszeta}
  g^* (\dd \theta) = \szeta.
\end{equation}
Indeed, first we pick a point $z_0 \in S^1$ and take an angle coordinate $\theta: S^1 \setminus \{ z_0 \} \to \R$. As such, $\theta (e^{i t}) = t + 2 \pi N_0$ for some $N_0 \in \Z$ independent of $t$. Let $\Omega_0 \dfn g^{-1}(S^1 \setminus \{ z_0 \})$. On $\Pi^{-1}(\Omega_0)$ we have
\begin{equation*}
  \theta(g(\Pi(p))) = \theta(G(\Pi(p))) = \theta(e^{i \psi(p)}) = \psi(p) + 2 \pi N_0;
\end{equation*}
hence,
\begin{equation*}
  \Pi^* g^* (\dd \theta) = \dd (\theta \circ g \circ \Pi) = \dd \psi = \Pi^* \szeta.
\end{equation*}
Since $\Pi$ is a submersion we conclude that~\eqref{eq:gdthetaiszeta} holds on $\Omega_0$, and since $z_0$ is arbitrary that identity then holds everywhere. Incidentally, this also proves that $g$ must be surjective: if its image missed a point in $S^1$ we would be able to define an angle coordinate $\theta$ all over it, and $\theta \circ g \in \cinfty(\Omega; \R)$ would be a global primitive of $2 \pi \zeta$, which cannot happen.

It is the case that $g$ is a submersion, since $g^*: T^*_{g(x)} S^1 \to T^*_x \Omega$ is injective for every $x \in \Omega$. Indeed, any $\xi \in T^*_{g(x)} S^1$ can be written as $\xi = \lambda \dd \theta|_{g(x)}$ for some $\lambda \in \R$. Hence, by~\eqref{eq:gdthetaiszeta}:
\begin{equation*}
  g^* \xi = 0
  \Longrightarrow
  \lambda \szeta|_x = \lambda g^* (\dd \theta|_{g(x)}) = g^* \xi = 0
  \Longrightarrow
  \lambda = 0,
\end{equation*}
since $\szeta$ is non-singular. It also follows from~\eqref{eq:gdthetaiszeta} that
\begin{equation*}
  \VV_x = \ker \left\{ g_*: T_x \Omega \lra T_{g(x)} S^1 \right\},
  \quad \forall x \in \Omega.
\end{equation*}
Indeed, given $v \in T_x \Omega$ we have
\begin{equation*}
  \langle \zeta|_x, v \rangle = 0
  \Longleftrightarrow
  \langle \dd \theta|_{g(x)}, g_* v \rangle = 0
\end{equation*}
and since $\dd \theta$ is a global frame for $T S^1$ we have that $\dd \theta|_{g(x)}$ is injective as a linear functional on $T_{g(x)} S^1$. We conclude that, when $\zeta$ is integral, the leaves of our foliation $\mathcal{F}$ -- which are maximal integral submanifolds of $\VV$ -- are contained in the level sets of $g$.

Next, pick $X$ a vector field transversal to $\szeta$ on $\Omega$. Then
\begin{equation*}
  \langle \dd \theta|_{g(x)}, g_* (X|_x) \rangle = \langle \szeta|_x, X|_x \rangle = 1
  \Longrightarrow
  g_* (X|_x) = \left. \frac{\del}{\del \theta} \right|_{g(x)},
  \quad \forall x \in \Omega;
\end{equation*}
i.e.,~$X$ and $\del/\del \theta$ are $g$-related. Moreover, if $\{ \Phi_t \}_{t \in \R}$ is the flow of $X$ then
\begin{equation}
  \label{eq:g_equivariant}
  g (\Phi_t(x)) = e^{i t} g(x),
  \quad \forall (t, x) \in \R \times \Omega.
\end{equation}
Indeed, if we fix $x \in \Omega$ then $t \in \R \mapsto e^{i t} g(x) \in S^1$ is clearly the maximal integral curve of $\del / \del \theta$ on $S^1$ passing through $g(x)$ at $t = 0$; on the other hand, for any $t_0 \in \R$ one has
\begin{equation*}
  \left. \frac{\dd}{\dd t} \right|_{t = t_0} g (\Phi_t(x))
  =
  g_* \left( \left. \frac{\dd}{\dd t} \right|_{t = t_0} \Phi_t(x) \right)
  =
  g_* (X|_{\Phi_{t_0}(x)})
  =
  \left. \frac{\del}{\del \theta} \right|_{g(\Phi_{t_0}(x))}.
\end{equation*}
Hence, the same holds true for the curve $t \in \R \mapsto g (\Phi_t(x)) \in S^1$. Relationship~\eqref{eq:g_equivariant} then follows from the uniqueness of solutions to the Cauchy problem in ODEs.

Actually, one can prove~\cite[Lemma~9.3.9]{conlon_manifolds} that the leaves of $\mathcal{F}$ are exactly level sets
\begin{equation*}
  L_z \dfn g^{-1}(z), \quad z \in S^1,
\end{equation*}
and it follows from~\eqref{eq:g_equivariant} that
\begin{equation*}
  \Phi_t (L_z) \sset L_{e^{it} z},
  \quad \forall t \in \R, \ \forall z \in S^1.
\end{equation*}
Thanks to time reversibility we have
\begin{equation}
  \label{eq:navigating_leaves}
  \Phi_t: L_z \lra L_{e^{it} z} \quad \text{diffeomorphically}.
\end{equation}
Moreover,
\begin{equation*}
  \left\{ t \in \R \st \Phi_t(L_z) = L_z \right\} = 2 \pi \Z = \per(\szeta),
  \quad \forall z \in S^1.
\end{equation*}
We also get a nice description of the (smooth) null solutions of $\VV$. Given $f \in \cinfty(\Omega)$, the following\footnote{This is already the zeroth instance of Theorem~\ref{thm:rational_mainthm}: leaves, being connected, have a vanishing $H^0$; hence, generate a trivial rank zero vector bundle over the circle; the isomorphism, in this regard, is explicitly pullback by $g$.} are equivalent:
\begin{equation}
  \label{eq: equivalence kernel d'}
  \parbox{7cm}{
    \begin{itemize}
    \item $\dd'_0 f = 0$;
    \item $f$ is constant along the leaves of $\mathcal{F}$;
    \item $f = \rho \circ g$ for some $\rho \in \cinfty(S^1)$.
    \end{itemize}
  }
\end{equation}

\subsection{The structure of locally trivial fiber bundle}

Still under the assumption of integrality of $\zeta$, we show next that $g$ is a locally trivial fibration, with typical fiber $L_1$, taking the care of making explicit the transition functions. We split $S^1$ into two overlapping open sets
\begin{equation*}
  U_z \dfn S^1 \setminus \{ z \}, \quad z = \pm 1,
\end{equation*}
over which we fix angle coordinates $\theta_z: U_z \to \R$. For each $z = \pm 1$, notice that
\begin{equation*}
  w \in U_z \Longrightarrow w \neq z \Longrightarrow w^{-1} \neq z^{-1} = z \Longrightarrow w^{-1} \in U_z
\end{equation*}
(that is, both $U_1$ and $U_{-1}$ are symmetric under inversion); hence,
\begin{equation*}
  e^{i \theta_z( w^{-1})} w = w^{-1} w = 1, \quad \forall w \in U_z.
\end{equation*}
Now consider
\begin{equation*}
\Omega_z \dfn g^{-1}(U_z).
\end{equation*}
Given $x \in \Omega_z$, it follows from \eqref{eq:navigating_leaves} that
\begin{equation*}
  x \in L_{g(x)}, \ g(x) \in U_z \Longrightarrow \Phi_{\theta_z( g(x)^{-1})} (x) \in L_{e^{i \theta_z( g(x)^{-1})} g(x)} = L_1.
\end{equation*}

We conclude that, for each $z = \pm 1$, the map
\begin{equation*}
  \TR{\phi_z}{x}{\Omega_z}{\left( g(x), \Phi_{\theta_z( g(x)^{-1})} (x) \right)}{U_z \times L_1}
\end{equation*}
is well-defined, and, moreover, a diffeomorphism.

  Indeed, smoothness of $\phi_z$ is clear. If $(w, y) \in \ran (\phi_z)$ then there exists $x \in U_z$ such that
  \begin{align*}
    \left\{
    \begin{array}{r c l}
      w & = & g(x) \\
      y & = & \Phi_{\theta_z( g(x)^{-1})} (x)
    \end{array}
    \right. \ \Longrightarrow \ x = \Phi_{- \theta_z( w^{-1})} (y),
  \end{align*}
  by~\eqref{eq:navigating_leaves}. This motivates the definition of
  \begin{equation*}
    \TR{\Xi_{z}}{(w, y)}{U_z \times L_1}{ \Phi_{- \theta_z( w^{-1})} (y)}{\Omega_z}
  \end{equation*} 
  which, we claim, is well-defined and the inverse of $\phi_z$; its smoothness is immediate. Indeed:
  \begin{itemize}
  \item $\Xi_{z}$ is well-defined: it follows from~\eqref{eq:g_equivariant} that
    \begin{equation*}
      g \left( \Phi_{- \theta_z( w^{-1})} (y) \right) = e^{-i \theta_z( w^{-1})} g(y) = w \cdot 1 = w \in U_{z}.
    \end{equation*}
  \item $\Xi_{z}$ is the right-inverse of $\phi_{z}$: for any $(w, y) \in U_z \times L_1$ we have
    \begin{align*}
      \left( \phi_{z} \circ \Xi_{z} \right) (w, y)
      &= \phi_{z} \left( \Phi_{- \theta_z( w^{-1})} (y) \right) \\ 
      &= \left( w,  \Phi_{\theta_z( w^{-1})} \left( \Phi_{- \theta_z( w^{-1})} (y) \right) \right) \\ 
      &= (w, y).
    \end{align*}  
  \item $\Xi_{z}$ is the left-inverse of $\phi_{z}$: for any $x \in \Omega_z$ we have 
    \begin{align*}
      \left( \Xi_{z} \circ \phi_{z}  \right) (x)
      &= \Xi_{z} \left( g(x), \Phi_{\theta_z( g(x)^{-1})} (x) \right) \\
      &= \Phi_{- \theta_z( g(x)^{-1})} \left( \Phi_{\theta_z( g(x)^{-1})} (x) \right) \\
      &= x.
    \end{align*}
  \end{itemize}

Moreover, note that for $w \in U_1 \cap U_{-1}$ and $y \in L_1$ we have
\begin{align*}
  \phi_{-z} (\phi_z^{-1} (w, y)) &= \phi_{-z} \left( \Phi_{-\theta_z( w^{-1})} (y) \right) \\
  &= \left( w, \Phi_{\theta_{-z} (w^{-1})} (\Phi_{-\theta_z( w^{-1})} (y) ) \right) \\
  &= \left( w, \Phi_{[\theta_{-z} (w^{-1}) - \theta_z( w^{-1})]} (y) \right).
\end{align*}
But since $\theta_{-z}, \theta_z$ are angle functions, there exists a non-zero $N \in \Z$ such that
\begin{equation*}
  \theta_{-z} - \theta_z = 2 \pi N
  \quad \text{on $U_1 \cap U_{-1}$};
\end{equation*}
hence,
\begin{equation} \label{eq: transition functions}
  ( \phi_{-z} \circ \phi_z^{-1}) (w, y) = (w, \Phi_{2 \pi N} (y)),
  \quad \forall w \in U_1 \cap U_{-1}, \ \forall y \in L_1.
\end{equation}

\subsection{Description of the cohomology}

Using the constructions in the previous section, we shall now study $H^{q}_{\VV}(\Omega)$, including its topology.

For $z = \pm 1$, take $Y$ a smooth section of $\mathcal{V}$ over $\Omega_z = g^{-1}(U_z)$ and let $(\phi_z)_* Y$ be its pushforward to $U_z \times L_1$. If $\pr_1:S^1 \times L_1 \to S^1$ denotes projection onto the first factor, then $\pr_1 \circ \phi_z = g$ on $\Omega_z$. Hence, by~\eqref{eq:gdthetaiszeta},
\begin{equation*}
  \phi_z^* (\pr_1^* (\dd \theta)) = g^* (\dd \theta) = \szeta
  \quad \text{on $\Omega_z$}.
\end{equation*}
Since $\szeta$ annihilates $L$, so $\pr_1^* (\dd \theta)$ must annihilate $(\phi_z)_* Y$. And, conversely, if $Y$ is any smooth vector field in $\Omega_z$ such that $(\phi_z)_* Y$ is in the kernel of $\pr_1^* ( \dd \theta)$, then $Y$ must be a section of $\mathcal{V}$.

The involutive subbundle $\WW \sset T (S^1 \times L_1)$ whose sections are annihilated by $\pr_1^* (\dd \theta)$ is precisely $\pr_2^* (TL_1)$ -- the tangent bundle of $L_1$ embedded into $T (S^1 \times L_1)$, where $\pr_2: S^1 \times L_1 \to L_1$ stands for the remaining projection. What we actually proved above is that the map 
\begin{equation*}
  \text{$\phi_z: \Omega_z \to U_z \times L_1$ is \emph{compatible} with the pair of structures $\VV$ and $\WW$},
\end{equation*}
restricted to the open sets $\Omega_z$ and $U_z \times L_1$, respectively. As such, and by functoriality alone, $\phi_z$ induces (topological linear) isomorphisms
\begin{equation}
  \label{eq:iso1}
  \phi_z^*: H_{\WW}^q( U_z \times L_1 ) \lra H^q_{\VV} (\Omega_z),
\end{equation}
which we will describe concretely and in detail below.

Since $\phi_z$ is a diffeomorphism, $\phi_z^*: F^q(U_z \times L_1) \to F^q(\Omega_z)$ is an isomorphism and, given $f \in F^q(U_z \times L_1)$,
\begin{equation*}
  f \in \NN^q_{\WW}(U_z \times L_1) \Longleftrightarrow \phi_z^* f \in \NN^q_{\VV} (\Omega_z),
\end{equation*}
clearly, as $(\phi_z)_*$ maps sections of $\VV$ onto sections of $\WW$ bijectively. Isomorphisms are, thus, induced on the quotients:
\begin{equation*}
  \phi_z^*: \Lambda_{\WW}^q( U_z \times L_1 ) \lra \Lambda^q_{\VV} (\Omega_z).
\end{equation*}
These are chain maps: pullbacks commute with the exterior derivatives on the level of forms; hence, their induced maps on the quotients commute with the induced differentials. Therefore, they further descend to cohomology as~\eqref{eq:iso1}, surely isomorphisms as that was already the case above.

The advantage is that the differential complex associated to $\WW = \pr_2^* (TL_1)$ is much easier to describe. For $U \sset S^1$ an open set, one can show that
\begin{equation*}
  \NN^q_{\WW}(U \times L_1) = \left\{ \pr_1^* (\dd \theta) \wedge \eta \st \eta \in \cinfty \left( U \times L_1; \pr_2^* (\wedge^{q - 1} T^* L_1) \right) \right\}.
\end{equation*}
  Indeed, by definition $\NN^q_{\WW}(U \times L_1)$ denotes the space of elements of $f \in F^{q}(U \times L_1)$ such that
  \begin{equation*}
    f(Y_{1}, \ldots, Y_{q}) = 0, \quad \text{for any $Y_{1}, \ldots Y_{q}$ smooth sections of $\WW$}.
  \end{equation*} 
  Since each section $Y$ of $\WW$ is given locally by
  \begin{equation*}
    Y = \sum_{j=1}^{n} a_{j}(w, t) \frac{\del}{\del t_{j}},
  \end{equation*}
  where $(t_1, \ldots, t_n)$ are local coordinates on $L_1$, the result follows.

Hence, the composition
\begin{equation*}
  \cinfty \left( U \times L_1 ; \pr_2^* (\wedge^{q} T^* L_1) \right)
  \hookrightarrow
  F^q (U \times L_1) \lra \Lambda_\WW^q (U \times L_1)
\end{equation*}
(where the first arrow is inclusion, and the second one is projection onto the quotient) is an isomorphism of Fr\'echet spaces. Moreover, let
\begin{equation*}
  \dd_{L_1}: \cinfty \left( U \times L_1 ; \pr_2^* (\wedge^q T^*L_1) \right) \lra \cinfty \left (U \times L_1 ; \pr_2^* (\wedge^{q+1} T^*L_1) \right)
\end{equation*}
be the partial exterior derivative in the direction of $L_1$. Locally, if 
\begin{equation*}
  \text{$(V; t_1, \ldots, t_n)$ is a coordinate system in $L_1$,}
\end{equation*}
we can write an $f \in \cinfty (U \times L_1 ; \pr_2^* (\wedge^{q} T^*L_1))$ as\footnote{Here, $J = (j_1, \ldots, j_q)$ is an \emph{ordered} multi-index, and $\dd t_J = \dd t_{j_1} \wedge \ldots \wedge \dd t_{j_q}$}.
\begin{equation*}
  f = \psum_{|J| = q} f_J (w, t) \dd t_J,
  \quad f_J \in \cinfty(U \times V),
\end{equation*}
in which case
\begin{equation*}
  \dd_{L_1} f = \psum_{|J| = q} \sum_{j = 1}^n \frac{\del f_J (w, t)}{\del t_j} \dd t_j \wedge \dd t_J.
\end{equation*}
Hence, $\dd f = \dd_{L_1} f$ modulo multiples of $\pr_1^* (\dd \theta)$ -- notice that 
\begin{equation*}
  \text{$\pr_1^* (\dd \theta), \dd t_1, \ldots, \dd t_n$ is a frame for $T^* (S^1 \times V)$,}
\end{equation*}
just as $(\theta_z, t_1, \ldots, t_n)$ is a smooth chart on $U_z \times V$ --; hence, we have a commutative diagram
\begin{equation*}
  \begin{tikzcd}
    \cinfty \left( U \times L_1 ; \pr_2^* (\wedge^{q} T^* L_1) \right) \arrow{r}{\simeq} \arrow[swap]{d}{\dd_{L_1}}
    &
    \Lambda_\WW^q (U \times L_1) \arrow{d}{\dd'_q}  \\
    \cinfty \left( U \times L_1 ; \pr_2^* (\wedge^{q + 1} T^* L_1) \right) \arrow[swap]{r}{\simeq}
    &
    \Lambda_\WW^{q + 1} (U \times L_1)
  \end{tikzcd} .
\end{equation*}
An important result about $\dd_{L_1}$ is contained in the following general result:
\begin{Prop}
  \label{prop:partial_derham_solvable}
  Let $M,N$ be smooth manifolds of dimensions $m,n$, respectively. Consider $\WW \dfn \pr^*_2\left(TN \right) \subset T(M\times N)$, which defines an involutive structure over $M \times N$. Then, if $\dd_{N}:\Lambda^q_{\WW}(M \times N) \to \Lambda^{q+1}_{\WW}(M \times N)$ denotes the associated differential complex, the cohomology space is topologically isomorphic to 
  \begin{equation*}
    \cinfty \left(M; H^{q}_{\dR}(N) \right) = \cinfty(M) \widehat{\otimes} H^{q}_{\dR}(N),
    \quad 0 \leq q \leq n,
  \end{equation*}
  where $\widehat{\otimes}$ denotes the completed tensor product in the category of Fr{\'e}chet-nuclear spaces. In particular, $\dd_N$ is globally solvable in every degree.
\end{Prop}
\begin{Rem}
  This result is well-known, and can be obtained by applying the K\"unneth formula for foliated de Rham cohomology~\cite[Theorem~2.1]{Bertelson2011} with the point foliation $\mathcal{F} \dfn \{p\}_{p \in M}$ and the foliation of a single leaf $\mathcal{G} \dfn \{N\}$ (see also~\cite{Grothendieck1955, Kaup1967}). Since our context is much simpler than that, it is interesting to provide a self-contained proof, which we present below.
\end{Rem}
\begin{proof}
  Recall that the natural topology on the space of smooth sections of a vector bundle is Fr{\'e}chet-nuclear\footnote{If $E \to M$ is a vector bundle of rank $r$, the space of smooth sections $\cinfty (M; E)$ is topologically isomorphic to a closed subspace of $\prod_{\alpha} \cinfty(U_\alpha; \C^r)$, where $\{ U_\alpha \}$ is a countable trivializing open covering of $M$, thus Fr{\'e}chet and nuclear~\cite[Proposition~50.1, Theorem~51.5]{treves_tvs}.}. Moreover, it follows from de Rham's Theorem that the exterior derivative has closed range in every degree, so all the de Rham cohomology spaces are Fr{\'e}chet-nuclear~\cite[Proposition~50.1]{treves_tvs}. Consider the natural continuous bilinear map
  \begin{equation}
    \label{eq:tensor_product_1}
    \begin{array}{r c l}
      \cinfty(M) \times F^q(N) & \lra & \cinfty \left( M \times N; \pr^*_2( \wedge^q T^*N) \right) \\
      (f,\omega) & \longmapsto & \pr_2^* f \cdot \pr_2^* \omega
    \end{array}
  \end{equation}
  which extends to a linear map 
  \begin{equation}
    \label{eq:tensor_product_2}
    \cinfty(M) \otimes F^q(N) \lra \cinfty \left( M \times N; \pr^*_{2} ( \wedge^q T^*N) \right),
  \end{equation}
  whose range is given by all the $\sigma \in F^q(M \times N)$ such that 
  \begin{equation*}
    \sigma|_p \in \Span \left\{ ( \pr_2^*\eta_1 )|_p, \ldots, (\pr_2^* \eta_k)|_p \right\},
    \quad \forall p \in M \times N,
  \end{equation*}
  for some $\C$-linearly independent $\eta_1, \ldots, \eta_k \in F^q(N)$. In particular, the map~\eqref{eq:tensor_product_2} is injective with dense range, hence~\eqref{eq:tensor_product_1} induces an isomorphism
  \begin{equation*}
    \cinfty(M) \widehat{\otimes} F^q(N) \lra \cinfty \left(M \times N; \pr^*_2(\wedge^q T^*N) \right).
  \end{equation*}
  Moreover, the diagram 
  \begin{equation*}
    \begin{tikzcd}
      \cinfty(M) \widehat{\otimes} F^q(N) \arrow{r} \arrow{d}{\Id \widehat{\otimes} \dd} & \cinfty \left(M \times N; \pr^*_2( \wedge^q T^* N) \right) \arrow{d}{\dd_N} \\
      \cinfty(M) \widehat{\otimes} F^{q+1}(N) \arrow{r} & \cinfty \left( M \times N; \pr^*_2(\wedge^{q+1} T^* N) \right)
    \end{tikzcd}
  \end{equation*}
  is commutative. From de Rham's Theorem, one has a short exact sequence 
  \begin{equation*}
    0 \lra \ker \dd \lra F^{q}(N) \lra \ran \dd \lra 0
  \end{equation*}
  of Fr{\'e}chet-nuclear spaces, and since the completed tensor product preserves short exact sequences~\cite[Proposition~(5.17), Chapter~IX]{demailly}, we conclude that so is
  \begin{equation*}
    0 \lra \cinfty(M) \widehat{\otimes} \ker \dd \lra \cinfty(M) \widehat{\otimes} F^q(N) \lra \cinfty(M) \widehat{\otimes} \ran \dd \lra  0.
  \end{equation*}
  In particular, the range of $\dd_N$ is closed and, moreover, the cohomology is isomorphic to 
  \begin{equation*}
    \frac{\cinfty(M) \widehat{\otimes} \ker \dd}{\cinfty(M) \widehat{\otimes} \ran \dd},
  \end{equation*}
  which is in turn isomorphic to $\cinfty(M) \widehat{\otimes} H^q_{\dR}(N)$, since one has the obvious short exact sequence $0\to \ran \dd \to \ker \dd \to H^q_{\dR}(N) \to 0$ of Fr{\'e}chet-nuclear spaces, and the exact functor property implies the result.
\end{proof}

It follows from Proposition~\ref{prop:partial_derham_solvable} (with $M = U_z$ and $N = L_1$) that $\dd'_q: \Lambda^q_\VV(\Omega_z)\to \Lambda^{q+1}_\VV(\Omega_z)$ has closed range and~\eqref{eq:iso1} induces a topological isomorphism
\begin{equation*}
  (\phi_z)_*: H^q_{\VV} ( \Omega_z ) \lra \cinfty \left( U_z ; H^q_{\dR}(L_1) \right).
\end{equation*}
Consider the continuous map
\begin{equation}
  \label{eq:main_iso}
  \mathsf{T}: H^q_{\VV}(\Omega) \lra \cinfty (U_{-1} ; H^q_{\dR} (L_1) ) \times \cinfty (U_1 ; H^q_{\dR}(L_1) )
\end{equation}
defined by
\begin{equation*}
  \mathsf{T}([\pmb{f}]) \dfn \left( (\phi_{-1})_* \left[ \pmb{f}|_{\Omega_{-1}} \right], (\phi_1)_* \left[ \pmb{f}|_{\Omega_1} \right] \right).
\end{equation*}
In general, one should not expect injectivity from a map like this (the Mayer-Vietoris long exact sequence). However, in our situation, one has the structure of a locally fine sheaf (over the base $S^1$): assume that $\pmb{f} \in \Lambda^q_\VV(\Omega)$ is $\dd'_q$-closed and such that for each $z=\pm 1$ we have $\pmb{f}|_{\Omega_z} = \dd'_{q - 1} \pmb{v}_z$ for some $\pmb{v}_z \in \Lambda^{q-1}_\VV(\Omega_z)$. Let $\{\rho_z\}_{z=\pm 1}$ be a partition of unity on $S^1$ subordinated to the cover $\{U_z\}_{z=\pm 1}$. Then
\begin{equation*}
  \pmb{v} \dfn (\rho_{-1} \circ g) \pmb{v}_{-1} + (\rho_{1} \circ g) \pmb{v}_1
\end{equation*}
is a well-defined element of $\Lambda_\VV^{q - 1} (\Omega)$, and satisfies $\dd' \pmb{v} = \pmb{f}$ thanks to~\eqref{eq: equivalence kernel d'}.

  Indeed:
  \begin{equation*}
    \dd' \pmb{v}
    = \left( \dd' (\rho_{-1} \circ g) \right) \wedge \pmb{v}_{-1} + \left( \dd'(\rho_{1} \circ g) \right) \wedge \pmb{v}_{1}
    + (\rho_{-1} \circ g) \pmb{f} + (\rho_{1} \circ g) \pmb{f} 
    = \pmb{f},
  \end{equation*}
  since $\dd' (\rho_{-1} \circ g) = \dd' (\rho_{1} \circ g) = 0$.

Thus we have just proved that~\eqref{eq:main_iso} is indeed injective; we now describe its range. Consider the map
\begin{equation*}
  \tau: \cinfty \left( U_{-1} \cap U_1 ; H^q_{\dR}(L_1) \right) \lra \cinfty \left( U_{-1} \cap U_1 ; H^q_{\dR}(L_1) \right)
\end{equation*}
given by (see~\eqref{eq: transition functions})
\begin{equation*}
  \tau \dfn (\phi_1 \circ \phi_{-1}^{-1})_*: \sigma^\flat(\cdot) \longmapsto  \left( \Phi_{2\pi N} \right)_* \sigma^\flat(\cdot),
\end{equation*}
where $\Phi_{2\pi}$ is Poincar{\'e}'s first return map on the leaf $L_1$ for the transversal vector field $X$. One can then form the (closed) subspace
\begin{equation*}
  \mathcal{R} \sset \cinfty (U_{-1} ; H^q_{\dR} (L_1) ) \times \cinfty (U_1 ; H^q_{\dR}(L_1) )
\end{equation*}
of compatible pairs $(\sigma^\flat_{-1}, \sigma^\flat_1)$ satisfying
\begin{equation} \label{eq: compatible pairs}
  \tau \left( \sigma^\flat_{-1}|_{U_{-1} \cap U_1} \right) = \sigma^\flat_1|_{U_{-1} \cap U_1}.
\end{equation}

Notice that $\mathcal{R}$ is none other than the space of smooth sections of a certain vector bundle $\mathbb{H}^q_{\dR} (L_1) \to S^1$ with typical fiber $H^q_{\dR} (L_1)$. Indeed, $\{ U_z \}_{z = \pm 1}$ is an open covering of $S^1$, and, for $z, z' = \pm 1$, we define $\varphi_{z,z'}: U_z \cap U_{z'} \to \mathrm{GL}(H^q_{\dR} (L_1))$ as:
\begin{itemize}
\item $\varphi_{z,z} (w) \dfn  \Id_{H^q_{\dR} (L_1)}$;
\item $\varphi_{1, -1} (w) \dfn \left( \Phi_{2\pi N} \right)_*$;
\item $\varphi_{-1, 1} (w) \dfn \left( \Phi_{-2\pi N} \right)_*$.
\end{itemize}
It is clear that these are cocycles; hence, by the general theory of vector bundles~\cite[Section~3.4]{conlon_manifolds} they are the transition functions of a well-defined vector bundle $\mathsf{p}: \mathbb{H}^q_{\dR} (L_1) \to S^1$. 

Indeed, 
  \begin{itemize}
  \item $\varphi_{1, -1} \circ \varphi_{-1, 1} = \Id_{H^q_{\dR} (L_1)} = \varphi_{1, 1};$ 
  \item $\varphi_{-1, 1} \circ \varphi_{1, - 1} = \Id_{H^q_{\dR} (L_1)} = \varphi_{-1,-1};$ 
  \item the other cases are immediate.
  \end{itemize}

As such, they completely describe the space of smooth sections of $\mathbb{H}^q_{\dR} (L_1)$, which are smooth maps $\sigma: S^{1} \to \mathbb{H}^q_{\dR} (L_1)$ such that $\mathsf{p} \circ \sigma = \Id_{S^{1}}$. The local trivializations
\begin{equation*}
  \chi_z : \mathsf{p}^{-1}(U_z) \lra U_z \times H^q_{\dR}(L_{1}),
  \quad z = \pm 1,
\end{equation*}
satisfy
\begin{equation*}
  \chi_{z} \circ \chi_{z'}^{-1} (w, [f]) = \left( w, \varphi_{z,z'}(w) [f] \right),
  \quad z, z' = \pm 1.
\end{equation*}
Due to the section property, the image of $\sigma_{z} \dfn \sigma|_{U_z}$ is contained in $\mathsf{p}^{-1}(U_{z})$, hence satisfies
\begin{equation*}
  \chi_z \sigma_z(w) = (w, \sigma^\flat_z(w)),
  \quad w \in U_z,
\end{equation*}
for a uniquely determined smooth map $\sigma^\flat_z: U_z \to H^q_\dR(L_1)$. It follows that, for $w \in U_{-1} \cap U_1$, we have
\begin{align*}
  \left( w, \varphi_{z,z'}(w) \sigma^\flat_{z'} (w) \right)
  &= \chi_{z} \circ \chi_{z'}^{-1} (w, \sigma^\flat_{z'} (w)) \\
  &= \chi_z \sigma_z'(w) \\
  &=\chi_z \sigma_z(w) \\
  &= \left( w, \sigma^\flat_z(w) \right),
\end{align*}
that is:
\begin{equation*}
  \varphi_{z,z'} (w) \sigma^\flat_{z'} (w) = \sigma^\flat_{z} (w),
  \quad \forall z, z' = \pm 1, \ w \in U_{-1} \cap U_1.
\end{equation*}
Note that in fact a section is completely characterized by the property above, summarized as
\begin{equation*}
  \varphi_{1,-1} (w) \sigma^\flat_{-1}(w) =  \sigma^\flat_{1}(w),
  \quad \forall w \in U_{-1} \cap U_{1},
\end{equation*}
which is precisely~\eqref{eq: compatible pairs}.   

One has $\ran \mathsf{T} = \mathcal{R}$ and, moreover, we can explicitly write down its (clearly continuous) inverse $\mathsf{T}^{-1}: \mathcal{R} \to H^q_\VV(\Omega)$, by the formula
\begin{equation*}
  \mathsf{T}^{-1} \left( \sigma^\flat_{-1}, \sigma^\flat_1 \right) = (\rho_{-1} \circ g) (\phi_{-1}^*\sigma^\flat_{-1}) + (\rho_{1} \circ g) (\phi_1^*\sigma^\flat_1).
\end{equation*}
Indeed:
\begin{align*}
  \mathsf{T}^{-1} \mathsf{T}([\pmb{f}])
  &= \mathsf{T}^{-1} \left( (\phi_{-1})_* \left[ \pmb{f}|_{\Omega_{-1}} \right], (\phi_1)_* \left[ \pmb{f}|_{\Omega_1} \right] \right) \\
  &= \left[ (\rho_{-1} \circ g) \pmb{f} +  (\rho_{1} \circ g) \pmb{f}  \right]  \\
  &= [\pmb{f}];
\end{align*}
Now, for the other composition, take $\left( \sigma_{-1}^\flat, \sigma_1^\flat \right) \in \mathcal{R}$ and let
\begin{equation*}
  [\pmb{f}] \dfn (\rho_{-1} \circ g)(\phi_{-1}^* \sigma_{-1}^\flat) + (\rho_1 \circ g)(\phi^*_1 \sigma_1^\flat).
\end{equation*}
To prove that $\mathsf{T}([\pmb{f}]) = \left( \sigma_{-1}^\flat, \sigma_1^\flat \right)$, we have, for $z = \pm 1$, to restrict this cohomology class to $\Omega_z$ and then compute the respective pushforward. That is, we claim that
\begin{equation*}
  (\phi_z)_* [\pmb{f}|_{\Omega_z}] = \sigma_z^\flat,
  \quad z = \pm 1,
\end{equation*}
which is equivalent to showing that
\begin{equation}
  \label{eq:equalityofrestrictions}
  [\pmb{f}|_{\Omega_z}] = \phi_z^* \sigma_z^\flat,
  \quad z = \pm 1.
\end{equation}
We prove the latter equality locally about a point $x \in \Omega_z$. Assume first that $x \in \Omega_{-1} \cap \Omega_1$. Then, in a neighborhood of $x$, we have the equality $\phi_1^* \sigma_1^\flat = \phi_{-1}^* \sigma_{-1}^\flat$, since the pair $\left( \sigma_{-1}^\flat, \sigma_1^\flat \right)$ belongs to $\mathcal{R}$. Therefore, \eqref{eq:equalityofrestrictions} holds in that neighborhood since $( \rho_{-1} \circ g) + (\rho_1 \circ g) = 1$. Suppose now that $x \in \Omega_1 \setminus \Omega_{-1}$. Since $\supp (\rho_{-1} \circ g) \subset \Omega_{1}$, we have $\rho_{1} \circ g = 1$ in a neighborhood of $x$, and we have that $[ \pmb{f} ]$ equals $\phi_1^* \sigma_1^\flat$ there (by its very definition). Hence, $[\pmb{f}|_{\Omega_1}] = \phi_1^* \sigma_1^\flat$, and a similar argument entails the same for $z = -1$.

We have thus realized an isomorphism
\begin{equation*}
  \mathsf{T}: H^q_{\VV}(\Omega) \lra \cinfty (S^1; \mathbb{H}^q_{\dR} (L_1)).
\end{equation*}
and the proof of Theorem~\ref{thm:rational_mainthm} is complete.

\section{The irrational case in degree one}
\label{sec:irr_deg1}

In this and the next section, we work on Theorem~\ref{thm:degs1n_mainthm}, addressing specifically the first degree. In the irrational case, the foliation induced by $\zeta$ has dense leaves and the geometry is a lot more involved. Let us restate the desired theorem; recall the definition of $\dd'_q$ in~\eqref{eq:dprime}.

\begin{Thm}
  \label{thm:first_step_solv}
  Let $\zeta \in F^1(\Omega)$ be a closed, non-singular real-valued $1$-form on the compact manifold $\Omega$. Then, the following are equivalent:
  \begin{enumerate}
  \item $\dd'_0:\cinfty(\Omega)\to \Lambda^1_{\VV}(\Omega)$ is globally solvable.
  \item $\zeta$ is either rational or an irrational non-Liouville form.
  \end{enumerate}
\end{Thm}
\begin{Rem}
  This proves the equivalence $\eqref{thm:degs1n_mainthm-it1} \Longleftrightarrow \eqref{thm:degs1n_mainthm-it3}$ in Theorem~\ref{thm:degs1n_mainthm}.
\end{Rem}

Due to the results from the previous section, we can assume $\zeta$ irrational. Moreover, rescaling\footnote{Again, such a procedure does not change $\VV$ nor its complex.} $\zeta$ allows one to assume that its group of periods $\per(\zeta)$ is generated by $a_0 \dfn 1, a_1,\ldots,a_r$, all positive real numbers and linearly independent over $\Q$.

At this point, it is fruitful to reinterpret $\per(\zeta)$ in terms of the fundamental group of $\Omega$, that is:
\begin{equation*}
  \per(\zeta) = \left\{ \int_{\gamma} \zeta \st [\gamma] \in \pi_1(\Omega; x_0) \right\}.
\end{equation*}
As we shall see, this viewpoint triggers a few geometric constructions that will be primordial in the proof of the theorem, and which we present next.

\subsection{Coverings and groups}

Fix a point $x_0 \in \Omega$ and consider 
\begin{equation*}
  H(x_0) \dfn \left\{ [\gamma]\in \pi_1(\Omega; x_0) \st \int_{\gamma} \zeta = 0 \right\},
\end{equation*}
which is a normal subgroup of the fundamental group of $\Omega$ at $x_0$ (one can always choose a smooth representative in any homotopy class $[\gamma] \in \pi_1(\Omega; x_0)$). It is clear that $H(x_0)$ contains the commutator subgroup $[\pi_1(\Omega; x_0), \pi_1(\Omega; x_0)]$. In particular, if $L$ is the leaf of $\mathcal{F}$ through $x_0$ and $i: L \to \Omega$ is the inclusion, then $i_* \pi_1(L; x_0) \subset H(x_0)$. The surjective map $\pi_1(\Omega; x_0) \to \per(\zeta)$ given by $[\gamma] \mapsto \int_\gamma \zeta$ induces a group isomorphism
\begin{equation*}
  \per(\zeta) \simeq \pi_1(\Omega; x_0) / H (x_0).
\end{equation*}

We can then consider the Galois covering
\begin{equation*}
  \Pi: \widehat{\Omega} \lra \Omega
\end{equation*}
associated with the subgroup $H(x_0) \subset \pi_1(\Omega; x_0)$. From the general theory (see, for instance, \cite[Chapter 21]{Lee2012} or \cite[Theorem 5.9 and Exercise 5.2]{Forster1981}), we know that $\widehat{\Omega}$ is a connected $(n+1)$-dimensional smooth manifold, that
\begin{equation*}
  \Pi_* \left( \pi_1 (\widehat{\Omega}; p) \right) = H(\Pi(p)),
  \quad \forall p \in \widehat{\Omega},
\end{equation*}
and we can define a primitive $\psi: \widehat{\Omega} \to \R$ for $\Pi^* \zeta$ as follows: fixing a point $p_0 \in \Pi^{-1}(x_0)$, let
\begin{equation*}
  \psi (p) \dfn \int_{\Pi \circ \widehat{\gamma}} \zeta = \int_{\widehat{\gamma}} \Pi^*\zeta,
  \quad p \in \widehat{\Omega},
\end{equation*}
where $\widehat{\gamma}:[0,1] \to \widehat{\Omega}$ is any smooth path satisfying $\widehat{\gamma}(0) = p_0$ and $\widehat{\gamma} (1) = p$. It is clear that $\psi$ is well-defined, smooth, and, moreover, $\psi(p_0) = 0$ and $\dd \psi = \Pi^*\zeta$. Since $\zeta$ is non-singular, $\psi$ is a submersion, and the pullback foliation $\widehat{\mathcal{F}} \dfn \Pi^*\mathcal{F}$ is the same as the one induced by $\psi$. Indeed, for $p \in \widehat{\Omega}$:
\begin{equation*}
  \ker (\dd \psi)|_{p} = \ker \left(\zeta|_{\Pi(p)} \circ \Pi_*|_{p} \right) = (\Pi_*|_{p})^{-1} \left( \ker \zeta|_{\Pi(p)} \right) = (\Pi_*|_{p})^{-1} \VV|_{\Pi(p)},
\end{equation*}
the last term being precisely the tangent space to the leaf of $\widehat{\mathcal{F}}$ through $p$.

Observe that the leaves $\widehat{L}$ of $\widehat{\mathcal{F}}$ are precisely the connected components of the inverse images $\Pi^{-1}(L)$, for $L$ a leaf of $\mathcal{F}$. Note also that $\Pi|_{\widehat{L}}: \widehat{L} \to L$ is a local diffeomorphism (since $\widehat{L}$ and $L$ are weakly embedded submanifolds of the same dimension); and, moreover, is injective. Indeed, let $p, p' \in \widehat{L}$ be such that $\Pi(p) = \Pi(p') = x \in L$, and $\widehat{\gamma}: [0,1]\to \widehat{L}$ be a curve satisfying $\widehat{\gamma}(0) = p$ and $\widehat{\gamma}(1) = p'$. Then, $\gamma \dfn \Pi \circ \widehat{\gamma}$ is a closed curve based on $x$, which is contained in $L$, therefore
\begin{equation*}
  [\gamma] \in H (x) = \Pi_* \left( \pi_1 (\widehat{\Omega}; p) \right),
\end{equation*}
i.e.,~there is a loop $\widehat{\eta} :[0,1] \to \widehat{\Omega}$ based on $p$ that satisfies $[\Pi \circ \widehat{\eta}] = [\gamma]$. By the uniqueness property of (homotopy classes of) liftings, $\widehat{\gamma}$ must be closed, so $p = p'$. Therefore, $\Pi|_{\widehat{L}}: \widehat{L} \to L$ is a diffeomorphism.

Next, we will straighten this picture, and linearize both $\widehat{\Omega}$ and $\widehat{\mathcal{F}}$. If $X$ is a vector field on $\Omega$ transversal to $\zeta$, let $\widehat{X}$ be the smooth vector field on $\widehat{\Omega}$ which is $\Pi$-related to $X$, meaning that
\begin{equation*}
  \Pi_* (\widehat{X}|_{p}) = X|_{\Pi(p)},
  \quad \forall p \in \widehat{\Omega},
\end{equation*}
which is well-defined and unique since $\Pi$ is a local diffeomorphism. This vector field is clearly transversal to $\Pi^*\zeta $ and complete. Indeed,
\begin{equation*}
  \langle (\Pi^*\zeta)|_{p},  \widehat{X}|_{p} \rangle =
  \langle \zeta|_{\Pi(p)},  \Pi_* (\widehat{X}|_{p}) \rangle =
  \langle \zeta|_{\Pi(p)},  X|_{\Pi(p)} \rangle =
  1.
\end{equation*}
Moreover, by uniqueness of integral curves, the flows of $X$ and $\widehat{X}$ are $\Pi$-conjugated, hence the completeness of the latter. This means that, if $\{ \Phi_t \}_{t \in \R}$ and $\{ \widehat{\Phi}_t \}_{t \in \R}$ denote the flows of $X$ and $\widehat{X}$, respectively, then
\begin{equation*}
  \Pi \left( \widehat{\Phi}_t (p) \right) = \Phi_t (\Pi(p)),
  \quad \forall (t, p) \in \R \times \widehat{\Omega}.
\end{equation*}
Fixing $p \in \widehat{\Omega}$ we have
\begin{equation*}
  \left. \frac{\dd}{\dd t} \right|_{t=t_0} \psi(\widehat{\Phi}_t(p))
  = \left \langle \dd \psi|_{\widehat{\Phi}_{t_0}(p)}, \widehat{X}|_{\widehat{\Phi}_{t_0}(p)} \right \rangle
  = \left \langle (\Pi^* \zeta)|_{\widehat{\Phi}_{t_0}(p)}, \widehat{X}|_{\widehat{\Phi}_{t_0}(p)} \right \rangle
  = 1
\end{equation*}
from which we conclude that 
\begin{equation}
  \label{eq:psionphihat}
  \psi(\widehat{\Phi}_t(p)) = t + \psi(p),
  \quad \forall (t,p) \in \R \times \widehat{\Omega}.
\end{equation}

In order to proceed, we fix a leaf $\widehat{L}$ of $\widehat{\mathcal{F}}$, and consider the map
\begin{equation*}
  \TR{\widehat{\Phi}}{(t,p)}{\R \times \widehat{L}}{\widehat{\Phi}_t(p)}{\widehat{\Omega}}
\end{equation*}
which is easily seen to be a diffeomorphism: it is a local diffeomorphism (as the flow of a non-singular vector field), onto -- for any two leaves of $\widehat{\mathcal{F}}$ are diffeomorphic via some $\widehat{\Phi}_t$~\cite[Lemma~9.3.9]{conlon_manifolds} --, and, moreover, injective: were $\widehat{\Phi}_t(p) = \widehat{\Phi}_{t'}(p')$ for $p,p' \in \widehat{L}$, \eqref{eq:psionphihat} would entail $t + \psi(p) = t' + \psi(p')$, hence $t = t'$ (as $\psi$ is constant along the leaves), so $p = p'$. Another consequence of~\eqref{eq:psionphihat} is that
\begin{equation*}
  \widehat{\Phi}^* (\dd \psi) = \dd (\psi \circ \widehat{\Phi}) = \dd t
\end{equation*}
since $\psi$ is constant along $\widehat{L}$: in particular, $\widehat{\Phi}^* \widehat{\mathcal{F}} = (\Pi \circ \widehat{\Phi})^* \mathcal{F}$ is the product foliation on $\R \times \widehat{L}$ -- i.e.,~the one with leaves given by $\{t\}\times \widehat{L}$. 

The linearization is now done: we obtained a covering map $g:\R \times \widehat{L} \to \Omega$ given by
\begin{equation*}
  g(t,p) \dfn \Pi(\widehat{\Phi}(t,p)) = \Phi_t(\Pi(p))
\end{equation*}
such that $g^* \zeta = \dd t$, i.e.,~$g^* \mathcal{F}$ is the product foliation. We consider the associated involutive structure, denoted by $\widehat{\VV}$ -- which is just the tangent bundle of $\widehat{L}$ embedded into $T(\R \times \widehat{L})$ -- , and the associated spaces $\NN^q_{\widehat{\VV}}(\R \times \widehat{L})$ and $\Lambda^q_{\widehat{\VV}} (\R \times \widehat{L})$ defined via the main construction in the introduction: locally, classes $\pmb{f} \in \Lambda^q_{\widehat{\VV}} (\R \times \widehat{L})$ are uniquely represented by $q$-forms of the kind
\begin{equation*}
    \psum_{|J|=q} f_J(t,x) \dd x_J
\end{equation*}
where $(V; x_1,\ldots,x_n)$ is a system of coordinates for $\widehat{L}$ and $f_J \in \cinfty(\R \times V)$. The exterior derivative $\dd$ on $\R \times \widehat{L}$ admits a decomposition $\dd_{\R} + \dd_{\widehat{L}}$, and the induced differential on the complex $\Lambda^q_{\widehat{\VV}} (\R \times \widehat{L})$ becomes $\dd_{\widehat{L}}$.

By the previous considerations, the map $g: \R \times \widehat{L} \to \Omega$ is compatible with the pair of structures $\widehat{\VV}, \VV$, and thus by functoriality gives rise to continuous linear maps
\begin{equation}
  \label{eq:gstarquotients}
  g^*: \Lambda_\VV^q(\Omega) \lra \Lambda_{\widehat{\VV}}^q(\R \times \widehat{L}),
  \quad 0 \leq q \leq n,
\end{equation}
which form, moreover, a chain map. Indeed, one must check that the pullback map
\begin{equation*}
  g^*: F^{q}(\Omega) \lra F^{q}(\R \times \widehat{L})
\end{equation*}
induces a map in the quotient spaces. Let $\omega \in \NN^q_{\VV}(\Omega)$, that is, for each $x \in \Omega$ we have
\begin{equation*}
  \omega(v_1, \ldots, v_q) = 0,
  \quad \forall v_1, \ldots, v_q \in \VV_x.
\end{equation*}
Now take $\hat{v}_1, \ldots, \hat{v}_q \in \widehat{\VV}_{(t,x)} \simeq \{t\} \times T_x \widehat{L} $ for some $(t,x) \in \R \times \widehat{L}$. Then
\begin{equation*}
  (g^*\omega) (\hat{v}_1, \ldots, \hat{v}_q) =
  \omega (g_* \hat{v}_1, \ldots, g_* \hat{v}_q) =
  0
\end{equation*}
since $g_*: T_{(t,x)} (\R \times \widehat{L}) \to T_{g(t,x)} \Omega$ maps $\widehat{\VV}_{(t,x)}$ to $\VV_{g(t,x)}$ -- the meaning of compatibility between $\widehat{\VV}$ and $\VV$. Therefore, the mappings~\eqref{eq:gstarquotients} are well-defined; chainness being a consequence of the fact that exterior derivatives commute with pullbacks. Notice, moreover, that each map in~\eqref{eq:gstarquotients} is one-to-one (but not onto: we will compute their ranges in Proposition~\ref{prop:pullback_map} below); this is so because $g$ is a covering map (and in particular a local diffeomorphism, hence $g_*: \widehat{\VV}_{(t,x)} \to \VV_{g(t,x)}$ is an isomorphism).

\subsection{$\per(\zeta)$-invariant forms}
\label{sec:invariant-forms}

Let $\Deck(g)$ denote the group of \emph{deck transformations} of the covering $g$, which are diffeomorphisms $T:\R \times \widehat{L}\to \R \times \widehat{L}$ that satisfy $g = g \circ T$. From standard covering space theory, we know that $\Deck(g) \simeq \per(\zeta)$, and, moreover, the action of $\Deck(g)$ on $\R \times \widehat{L}$ is free and proper. In particular, we have, as smooth manifolds, $\R \times \widehat{L}/\Deck(g) \simeq \Omega$ (see \cite[Proposition 7.23 and Theorem 21.13]{Lee2012})

In the following lemma, we compute the deck transformations of $g$.
\begin{Lem}
  \label{deck_lemma}
  Deck transformations of $g$ are of the form
  \begin{equation*}
    T_\alpha (t,p)  \dfn  \left( t + \alpha, (\Pi_{\widehat{L}})^{-1} \circ \Phi_{-\alpha} \circ ( \Pi|_{\widehat{L}})(p) \right),
    \quad (t,p ) \in \R \times \widehat{L},
  \end{equation*}
  for some $\alpha \in \per(\zeta)$. Moreover, $\alpha \in \per(\zeta) \mapsto T_\alpha \in \Deck(g)$ is a group isomorphism.
\end{Lem}
\begin{proof}
  It is clear that each $T_\alpha$ is a deck transformation. Indeed, since $\alpha \in \per(\zeta)$, $\Phi_{-\alpha}$ is a diffeomorphism from $L \dfn \Pi (\widehat{L})$ onto itself~\cite[Lemma~9.3.9]{conlon_manifolds}, which makes the map well-defined. Furthermore
  \begin{align*}
    (g \circ T_\alpha) (t, p)
    &= g \left( t + \alpha, (\Pi_{\widehat{L}})^{-1} \circ \Phi_{-\alpha} \circ ( \Pi|_{\widehat{L}})(p) \right) \\
    &= \Phi_{t + \alpha} \left( \Phi_{-\alpha}  (\Pi (p)) \right) \\
    &=  \Phi_t  (\Pi (p)) \\
    &= g(t,p).
  \end{align*}
  Conversely, given $T \in \Deck(g)$ we write $T(s,p) = (t(s,p), v(s,p)) \in \R \times \widehat{L}$. Then
  \begin{equation}
    \label{eq:deck1}
    \Phi_s (\Pi(p)) = g(s,p) = g( T(s,p) ) = \Phi_{t(s,p)} \left( \Pi (v(s,p)) \right). 
  \end{equation}
  Note that $p, v(s,p) \in \widehat{L}$ imply $\Pi(p), \Pi(v(s,p)) \in L$. Hence~\cite[Lemma~9.3.9]{conlon_manifolds}: 
  \begin{equation*}
    t(s,p) - s \in \per(\zeta),
    \quad \forall (s,p) \in \R \times \widehat{L}.
  \end{equation*}
  By continuity and connectedness, and because $\per(\zeta)$ is countable (as a finitely generated group) we conclude that there exists $\alpha \in \per(\zeta)$ such that 
  \begin{equation*}
    t(s,p) - s = \alpha,
    \quad \forall (s,p) \in \R \times \widehat{L}.
  \end{equation*}
  Therefore~\eqref{eq:deck1}, by the flow property,
  \begin{equation*}
    \Pi (v(s,p)) = \Phi_{t(s,p)}^{-1} \Phi_{s} \left( \Pi (p) \right) = \Phi_{s - t(s,p)} \left( \Pi (p) \right) = \Phi_{-\alpha} \left( \Pi (p)  \right).
  \end{equation*}
  Since $v(s,p) \in \widehat{L}$, we can apply $(\Pi|_{\widehat{L}})^{-1}$ to both sides and obtain
  \begin{equation*}
    v(s,p) =  (\Pi|_{\widehat{L}})^{-1} \circ \Phi_{-\alpha} \circ (\Pi|_{\widehat{L}}) (p),
  \end{equation*}
  i.e.,~$T = T_\alpha$ as desired. The final claim follows from a simple computation.
\end{proof}

Now we consider the action of $\Deck(g)$ on the differential complex associated with $\widehat{\VV}$. A very simple calculation (using Lemma~\ref{deck_lemma}) entails that, for every $\alpha \in \per(\zeta)$, we have 
\begin{equation*}
  T_\alpha^* \left\{ \NN^q_{\widehat{\VV}}(\R \times \widehat{L}) \right\} = \NN^q_{\widehat{\VV}}(\R \times \widehat{L}),
  \quad 0 \leq q \leq n.
\end{equation*}
In particular, $T^*_\alpha$ induces isomorphisms of Fr{\'e}chet spaces
\begin{equation*}
  T_\alpha^* : \Lambda^q_{\widehat{\VV}} (\R \times \widehat{L}) \lra \Lambda^q_{\widehat{\VV}} (\R \times \widehat{L}),
  \quad 0 \leq q \leq n.
\end{equation*}
We can then consider the following (closed) subspace of $\Lambda^q_{\widehat{\VV}} (\R \times \widehat{L})$:
\begin{equation*}
  \Lambda^q_{\widehat{\VV}} (\R \times \widehat{L})^{\per(\zeta)}
  \dfn
  \left\{ \pmb{f} \in \Lambda^q_{\widehat{\VV}} (\R \times \widehat{L}) \st T^*_\alpha \pmb{f} = \pmb{f}, \quad \forall \alpha  \in \per(\zeta) \right\}.
\end{equation*}
Since the exterior derivative commutes with pullbacks, we have an induced subcomplex
\begin{equation*}
  \begin{tikzcd}
    \Lambda^0_{\widehat{\VV}} (\R \times \widehat{L})^{\per(\zeta)} \arrow[r, "\dd_{\widehat{L}}"] &
    \Lambda^1_{\widehat{\VV}} (\R \times \widehat{L})^{\per(\zeta)} \arrow[r, "\dd_{\widehat{L}}"] &
    \cdots                                                     \arrow[r, "\dd_{\widehat{L}}"] &
    \Lambda^n_{\widehat{\VV}} (\R \times \widehat{L})^{\per(\zeta)}
  \end{tikzcd} .
\end{equation*}
We can then prove the following result:
\begin{Prop}
  \label{prop:pullback_map}
  The map~\eqref{eq:gstarquotients} induces an isomorphism of complexes
  \begin{equation*}
    \begin{tikzcd}
      \Lambda^q_{\widehat{\VV}} (\R \times \widehat{L})^{\per(\zeta)} \arrow[r, "\dd_{\widehat{L}}"] & \Lambda^{q + 1}_{\widehat{\VV}} (\R \times \widehat{L})^{\per(\zeta)} \\
      \Lambda^q_\VV(\Omega) \arrow[u, "g^*"] \arrow[r,"\dd'_q "] & \Lambda^{q + 1}_\VV(\Omega) \arrow[u, "g^*"] 
    \end{tikzcd}
  \end{equation*}
  where $0 \leq q \leq n - 1$.
\end{Prop}
\begin{proof}
  It is clear that $g^*$ maps $\Lambda_\VV^q(\Omega)$ to $\Lambda^q_{\widehat{\VV}} (\R \times \widehat{L})^{\per(\zeta)}$, and since the latter is closed in $\Lambda^q_{\widehat{\VV}} (\R \times \widehat{L})$ we have a well-defined continuous linear map
  \begin{equation*}
    g^*: \Lambda_\VV^q(\Omega) \lra \Lambda^q_{\widehat{\VV}} (\R \times \widehat{L})^{\per(\zeta)}
  \end{equation*}
  defined by taking the pullback of a representative and computing its class. Indeed, for $\pmb{f} \in \Lambda_\VV^q(\Omega)$ we have, for all $\alpha \in \per(\zeta)$,
    \begin{equation*}
      T^*_\alpha (g^* \pmb{f}) = (g \circ T_{\alpha})^* \pmb{f} = g^* \pmb{f},
    \end{equation*}
    due to the fact that $T_{\alpha} \in \Deck(g)$. As already pointed out, this mapping is injective; we now check surjectivity.

  For each $x \in \Omega$, take $U_{x}\subset \Omega$ a distinguished neighborhood of $x$, meaning that 
  \begin{equation*}
    g^{-1}(U_x) = \bigsqcup_{\alpha \in \per(\zeta)} \til{U}_{x, \alpha},
  \end{equation*}
  where $\til{U}_{x, \alpha} \subset \R \times \widehat{L}$ is open, $g|_{\til{U}_{x,\alpha}} :\til{U}_{x, \alpha} \to U_x$ is a diffeomorphism and $T_\alpha ( \til{U}_{x,\beta} ) = \til{U}_{x, \alpha+\beta}$ for every $\alpha,\beta \in \per(\zeta)$. Given $\pmb{\omega} \in \Lambda^{q}_{\widehat{\VV}} (\R \times \widehat{L})^{\per(\zeta)}$, it follows from the $\per(\zeta)$-invariance that
  \begin{equation}
    \label{equivariance}
    \pmb{\omega}|_{\til{U}_{x, \alpha}} = T_{\beta-\alpha}^* \left( \pmb{\omega} |_{\til{U}_{x, \beta}} \right),
    \quad \forall \alpha, \beta \in \per(\zeta), \ \forall x \in \Omega.
  \end{equation}
  We then define $\pmb{\eta}_{U_x} \in \Lambda^q_\VV(U_{x})$ by
  \begin{equation*}
    \pmb{\eta}_{U_{x}} \dfn \left( g|_{\til{U}_{x,\alpha}}^{-1} \right)^* (\pmb{\omega}|_{\til{U}_{x, \alpha}}),
  \end{equation*}
  which is independent of $\alpha \in \per(\zeta)$ thanks to~\eqref{equivariance}. Indeed:
    \begin{align*}
      \left( g|_{\til{U}_{x,\alpha}}^{-1} \right)^* (\pmb{\omega}|_{\til{U}_{x, \alpha}})
      &= \left( g|_{\til{U}_{x,\alpha}}^{-1} \right)^* T_{\beta - \alpha}^* (\pmb{\omega}|_{\til{U}_{x, \beta}}) \\
      &= \left( T_{\beta - \alpha} \circ g|_{\til{U}_{x,\alpha}}^{-1} \right)^* (\pmb{\omega}|_{\til{U}_{x, \beta}}) \\
      &= \left( \left( g|_{\til{U}_{x,\alpha}} \circ T_{\beta - \alpha}^{-1} \right)^{-1} \right)^* (\pmb{\omega}|_{\til{U}_{x, \beta}}) \\
      &= \left( \left( g|_{\til{U}_{x,\alpha}} \circ T_{\alpha - \beta} \right)^{-1} \right)^* (\pmb{\omega}|_{\til{U}_{x, \beta}}) \\
      &= \left( g|_{\til{U}_{x,\beta}}^{-1} \right)^* (\pmb{\omega}|_{\til{U}_{x, \beta}})
    \end{align*}
    since $T_{\alpha - \beta}$ is a deck transformation of $g$ and $T_{\alpha - \beta} ( \til{U}_{x, \beta} )  = \til{U}_{x, \alpha}$. We obtain in this fashion of an open cover $\{U_{x}\}_{x \in \Omega}$ of $\Omega$, with locally defined $\pmb{\eta}$. However, using~\eqref{equivariance} again, it is easy to check that these definitions coincide in the intersection of open subsets of this cover, thus obtaining a global object $\pmb{\eta}\in \Lambda^q_\VV(\Omega)$, which clearly satisfies $g^*\pmb{\eta} = \pmb{\omega}$. From the Open Mapping Theorem, $g^*$ is a topological isomorphism. 
\end{proof}

We have, therefore, embedded our differential complex of interest as a subcomplex of $(\Lambda^{q}_{\widehat{\VV}}(\R \times \widehat{L}) ,\dd_{\widehat{L}})$.

\subsection{Group cohomology}

We recall here the relevant definitions that will concern us (see, e.g.,~\cite[Chapter~6]{weibel}).
\begin{Def}
  Let $G$ be a group. A topological vector space $E$ is a \emph{$G$-module} if there is a group homomorphism
  \begin{equation*}
    T:G \lra \Aut(E).
  \end{equation*}
  Explicitly, we have a family $\{T_g\}_{g\in G}$ of topological automorphisms of $E$ that satisfy the group property
  \begin{equation*}
    T_{g\cdot g'} = T_g \circ T_{g'},
    \quad g,g' \in G.
  \end{equation*}
\end{Def}
\begin{Rem}
  We regard $G$ as an abstract group (or endowed with the discrete topology). We also use the notation $g \cdot e$ to mean $T_g(e)$, for $g \in G$ and $e \in E$, where $E$ is a $G$-module.
\end{Rem}
Let $G$ be a countable abelian group and $E$ a $G$-module. Then, we define, for $q\geq 1$, the vector space
\begin{equation*}
  \msf{C}^{q}(G; E) \dfn \left\{ f:G^{q} \lra E \right\} = \prod_{G^{q}}E,
\end{equation*}
endowed with the product topology. We set $\msf{C}^{0}(G; E) \dfn E$. Define the map $\delta = \delta^{q}: \msf{C}^{q}(G; E) \to \msf{C}^{q+1}(G; E)$, for $q \geq 1$, by
\begin{align*}
  ( \delta \msf{f} )(g_1, \ldots, g_{q+1})
  &\dfn g_1 \cdot \msf{f}(g_2, \ldots,g_{q+1}) + (-1)^{q+1} \msf{f}(g_1,\ldots,g_q) \\
  &+ \sum_{k=1}^{q}(-1)^k \msf{f} (g_1, \ldots, g_{k-1}, g_k \cdot g_{k+1}, g_{k+2}, \ldots, g_{q+1} ), \\
\end{align*}
whereas, for $q=0$, we set $\delta^0: \msf{C}^0(G; E) \to \msf{C}^1(G; E)$ as
\begin{equation*}
  (\delta e)(g) \dfn g \cdot e - e.
\end{equation*}
For instance, if $\msf{f}:G \to E$ is an element of $\msf{C}^1(G; E)$, then 
\begin{equation*}
  (\delta \msf{f})(g_1,g_2) = g_1 \cdot \msf{f}(g_2) - \msf{f}(g_1\cdot g_2) + \msf{f}(g_1).
\end{equation*}

One can check that $\delta^{q+1} \circ \delta^{q} = 0$ for every $q \geq 0$, so it defines a cochain complex. It is also clear that $\delta^q$ is a continuous linear map. We define the \emph{group cohomology spaces} of the $G$-module $E$ as
\begin{equation*}
  \msf{H}^{q} (G; E) \dfn \frac{\ker \delta^{q}}{\ran \delta^{q-1}},
  \quad q \geq 1,
\end{equation*}
whereas $\msf{H}^{0}(G; E) \dfn \ker \delta^0$. We endow them with their natural quotient and subspace topologies. Regarding them, the following elementary property will be very important in what follows~\cite[Corollary~6.2.7]{weibel}.
\begin{Prop}
  \label{prop:groupcohom_general}
  If $G$ is a free abelian group, then $\msf{H}^{q}(G;E)=0$ for every $q>1$ and every $G$-module $E$.
\end{Prop}

\subsection{The double complex}

As we have seen, the spaces $\Lambda^q_{\widehat{\VV}} (\R \times \widehat{L})$ are naturally $\per(\zeta)$-modules. We can, therefore, consider the double complex of Fr{\'e}chet spaces
\begin{equation*}
  \msf{K}^{p,q} \dfn \msf{C}^{q}\left( \per(\zeta); \Lambda^p_{\widehat{\VV}}(\R \times \widehat{L}) \right),
  \quad p, q \geq 0,
\end{equation*}
whereas $\msf{K}^{p,q} \dfn 0$ if either $p$ or $q$ is negative. The differentials are, of course,
\begin{equation*}
  \dd^{p,q}_{\widehat{L}}:\msf{K}^{p,q} \lra \msf{K}^{p+1,q},
  \quad
  \delta^{p,q}:\msf{K}^{p,q} \lra \msf{K}^{p,q+1}
\end{equation*}
defined in the obvious manner; they commute. In fact, for $\msf{f} \in \msf{K}^{p,q}$:
  \begin{equation*}
    \left( \dd^{p,q}_{\widehat{L}} \msf{f} \right)(\alpha_1, \ldots, \alpha_q) \dfn \dd_{\widehat{L}} \left( \msf{f}(\alpha_1, \ldots, \alpha_q) \right),
    \quad \alpha_1, \ldots, \alpha_q \in \per(\zeta);
  \end{equation*}
  hence,
  {\scriptsize
    \begin{align*}
      \left[ \delta^{p+1,q} \left( \dd^{p,q}_{\widehat{L}} \msf{f} \right) \right] (\alpha_{1}, \ldots, \alpha_{q + 1})
      &= \alpha_1 \cdot \left( \dd^{p,q}_{\widehat{L}} \msf{f} \right) (\alpha_2,\ldots,\alpha_{q+1}) + (-1)^{q+1} \left( \dd^{p,q}_{\widehat{L}} \msf{f} \right) (\alpha_1, \ldots, \alpha_q) \\
      &+ \sum_{k=1}^{q}(-1)^k \left( \dd^{p,q}_{\widehat{L}} \msf{f} \right) (\alpha_1, \ldots, \alpha_{k-1}, \alpha_k + \alpha_{k+1}, \alpha_{k+2}, \ldots, \alpha_{q+1}) \\
      &= T_{\alpha_1}^* \left[ \dd_{\widehat{L}} \left( \msf{f} (\alpha_2,\ldots,\alpha_{q+1}) \right) \right] + (-1)^{q+1} \dd_{\widehat{L}} \left( \msf{f} (\alpha_1, \ldots, \alpha_q) \right) \\
      &+ \sum_{k=1}^{q}(-1)^k  \dd_{\widehat{L}} \left( \msf{f} (\alpha_1, \ldots, \alpha_{k-1}, \alpha_k + \alpha_{k+1}, \alpha_{k+2}, \ldots, \alpha_{q+1}) \right) \\
      &=  \dd_{\widehat{L}} \left[ T_{\alpha_1}^* \left( \msf{f} (\alpha_2,\ldots,\alpha_{q+1}) \right) \right] + (-1)^{q+1} \dd_{\widehat{L}} \left( \msf{f} (\alpha_1, \ldots, \alpha_q) \right) \\
      &+ \sum_{k=1}^{q}(-1)^k  \dd_{\widehat{L}} \left( \msf{f} (\alpha_1, \ldots, \alpha_{k-1}, \alpha_k + \alpha_{k+1}, \alpha_{k+2}, \ldots, \alpha_{q+1}) \right) \\
      &=  \dd_{\widehat{L}} \Big\{ \alpha_1 \cdot \msf{f} (\alpha_2,\ldots,\alpha_{q+1}) + (-1)^{q+1} \msf{f} (\alpha_1, \ldots, \alpha_q)  \\
      &+ \sum_{k=1}^{q}(-1)^k  \msf{f} (\alpha_1, \ldots, \alpha_{k-1}, \alpha_k + \alpha_{k+1}, \alpha_{k+2}, \ldots, \alpha_{q+1}) \Big\} \\
      &= \dd_{\widehat{L}} \left\{ \left( \delta^{p,q} \msf{f} \right) (\alpha_1, \ldots, \alpha_{q + 1}) \right\} \\
      &= \left[ \dd^{p,q + 1}_{\widehat{L}} \left( \delta^{p,q} \msf{f} \right) \right] (\alpha_1, \ldots, \alpha_{q + 1}).
    \end{align*}
    }

When it is clear what spaces are being acted on, we shall often omit the indices $(p,q)$ to alleviate notation.
\begin{Rem}
  \label{rem:invariantfunctions}
  Notice that $\msf{K}^{p,q}$ has extra structure, that of a $\cinfty(\R \times \widehat{L})$-module; the differentials are not linear w.r.t.~this structure, though. Yet, $\delta^{p,q}$ is clearly linear w.r.t.~the subring of \emph{$\per(\zeta)$-invariant functions on $\R \times \widehat{L}$}, which is
  \begin{equation*}
    \cinfty(\R \times \widehat{L})^{\per(\zeta)} = \ker \delta^{0,0} = g^* \cinfty(\Omega).
  \end{equation*}
\end{Rem}
\begin{Lem}
  \label{lem:delta_acyclic}
  For every $p\geq 0$, the complex
  \begin{equation*}
    \begin{tikzcd}
      \msf{K}^{p, 0} \arrow[r, "\delta^{p, 0}"] &
      \msf{K}^{p, 1} \arrow[r, "\delta^{p, 1}"] &
      \msf{K}^{p, 2} \arrow[r, "\delta^{p, 2}"] &
      \cdots
    \end{tikzcd}
  \end{equation*}
  is exact. Explicitly, for every $q\geq 0$,
  \begin{equation*}
    \ker \delta^{p,q+1}=\ran \delta^{p,q};
  \end{equation*}
  or, equivalently, $\msf{H}^{q} \left(\per(\zeta) ; \Lambda^p_{\widehat{\VV}} (\R \times \widehat{L}) \right) = 0$ for all $q\geq 1$.
\end{Lem}
\begin{Rem} This technique is based on the proof of \cite[Proposition 1]{Godement1951} and on \cite[Corollaire 3]{gro57}.
	
\end{Rem}
\begin{proof}
  Fix $p\geq 0$. By Proposition~\ref{prop:groupcohom_general}, the only thing left to prove is that  that $\msf{H}^{1} \left(\per(\zeta) ; \Lambda^p_{\widehat{\VV}} (\R \times \widehat{L}) \right) = 0$. We claim that\footnote{The general statement is actually true for every $q \geq 0$, but we only need it for $q=1$.} the correspondence
  \begin{equation*}
    \Omega \supset U \longmapsto \msf{H}^{1} \left( \per(\zeta); \Lambda^p_{\widehat{\VV}} (g^{-1} (U)) \right)
  \end{equation*}
  defines a sheaf over $\Omega$. Indeed, notice first that given any open $U \sset \Omega$ the set $g^{-1}(U)$ is invariant by the action of $\per(\zeta)$ since each $T_\alpha$ is a deck transformation of $g$. It thus makes sense to regard $\Lambda^p_{\widehat{\VV}} (g^{-1} (U))$ as a $\per(\zeta)$-module as well, rerunning the relevant arguments in Section~\ref{sec:invariant-forms} (see Lemma~\ref{deck_lemma}). In particular, its group cohomology spaces are well-defined. Now, if $V\subset U \subset \Omega$ are open subsets of $\Omega$, then we have the restriction map of complexes
  \begin{equation*}
    \msf{C}^q \left( \per(\zeta); \Lambda^p_{\widehat{\VV}} (g^{-1} (U)) \right) \lra \msf{C}^q \left( \per(\zeta); \Lambda^p_{\widehat{\VV}} (g^{-1} (V)) \right),
    \quad q \geq 0,
  \end{equation*}
  given by (coordinate-wise) restriction of the differential forms. This map clearly commutes with $\delta$, so it induces a map
  \begin{equation*}
    \msf{H}^q \left( \per(\zeta); \Lambda^p_{\widehat{\VV}} (g^{-1} (U)) \right) \lra \msf{H}^q \left( \per(\zeta); \Lambda^p_{\widehat{\VV}} (g^{-1} (V)) \right)  \end{equation*}
  that satisfies the axioms of a presheaf for all $q\geq 0$. To check that such presheaf is complete when $q = 1$, we take $\{U_j\}$ an open covering of an open set $U\subset \Omega$ and fix $\{\rho_j\}$ a smooth partition of unity subordinated to it. If
  \begin{equation*}
    \text{$[\msf{f}] \in \msf{H}^1 \left( \per(\zeta); \Lambda^p_{\widehat{\VV}} (g^{-1} (U)) \right)$ is such that $[\msf{f}]|_{U_j}=0$ for every $j$},
  \end{equation*}
  then there exist $\msf{u}_j \in \msf{C}^{0} \left( \per(\zeta); \Lambda^p_{\widehat{\VV}} (g^{-1} (U_j)) \right)$ such that $\msf{f}|_{U_j} = \delta^{p,0}\msf{u}_j$ for every $j$. Hence (see Remark~\ref{rem:invariantfunctions}):
  \begin{equation*}
    \delta^{p,0} \left(\sum_{j}(\rho_j \circ g) \msf{u}_j \right) = \sum_{j} (\rho_j \circ g) \ \delta^{p,0} \msf{u}_j = \msf{f}
  \end{equation*}
  i.e.,~$[\msf{f}] = 0$. Similarly, let $[\msf{f}_j] \in \msf{H}^1 \left(\per(\zeta); \Lambda^p_{\widehat{\VV}} (g^{-1} (U_j)) \right)$ be such that $[\msf{f}_j]|_{U_j\cap U_k} = [\msf{f}_k]|_{U_j\cap U_k}$; i.e., there are $\msf{u}_{jk}\in \msf{C}^{0} \left( \per(\zeta);  \Lambda^p_{\widehat{\VV}} ( g^{-1}(U_j\cap U_k)) \right)$ such that 
  \begin{equation*}
    \msf{f}_j - \msf{f}_k = \delta^{p,0} \msf{u}_{jk}
    \quad \text{on $U_j \cap U_k$}.
  \end{equation*}
  Setting $\msf{f} \dfn \sum_j (\rho_j \circ g) \msf{f}_j\in \msf{C}^1 \left( \per(\zeta); \Lambda^p_{\widehat{\VV}} ( g^{-1}(U)) \right)$, we have $\delta^{p,1} \msf{f} = 0$ and 
  \begin{equation*}
    \msf{f}_j - \msf{f}|_{U_j}
    = \sum_{k ; U_j \cap U_k \neq \emptyset} (\rho_k \circ g) (\msf{f}_j - \msf{f}_k)
    = \delta^{p,0} \left( \sum_{k ; U_j \cap U_k \neq \emptyset} (\rho_k \circ g) \msf{u}_{jk} \right),
  \end{equation*}
  i.e., $[\msf{f}_j] = [\msf{f}|_{U_j}] = [\msf{f}]|_{U_j}$. This proves completeness.
  
  We claim that the stalks of this sheaf are trivial. Indeed, let $x_0\in \Omega$ and take $U \subset \Omega$ a distinguished neighborhood of $x_0$: this means that $g^{-1}(U) = \bigsqcup_{\alpha \in \per(\zeta)} \til{U}_\alpha$, where $g|_{\til{U}_\alpha}: \til{U}_\alpha \to U$ are diffeomorphisms and $T_{\alpha}(\til{U}_\beta) = \til{U}_{\alpha+\beta}$ for all $\alpha,\beta \in \per(\zeta)$. Let $\msf{f} \in \msf{C}^{1} \left( \per(\zeta) ; \Lambda^p_{\widehat{\VV}} (g^{-1}(U)) \right)$ be such that $\delta^{p,1} \msf{f} = 0$. Explicitly, this equation is
  \begin{equation}
    \label{eq:delta_rel}
    T^*_\alpha (\msf{f}(\beta)) - \msf{f}(\alpha + \beta) + \msf{f} (\alpha) = 0,
    \quad \forall \alpha,\beta \in \per(\zeta).
  \end{equation}
  
  Given $\alpha \in \per(\zeta)$, we have $\msf{f}(\alpha) \in \Lambda_{\widehat{\VV}}^p(g^{-1}(U)) = \prod_{\beta \in \per(\zeta)} \Lambda_{\widehat{\VV}}^{p}(\til{U}_\beta)$. Let $\pi_\beta: \Lambda_{\widehat{\VV}}^p(g^{-1}(U)) \to \Lambda_{\widehat{\VV}}^{p}(\til{U}_\beta)$ denote the projection induced by the restriction map and write $\pmb{f}_{\alpha,\beta} \dfn \pi_\beta (\msf{f}(\alpha))$. Note that, given $\alpha,\beta,\gamma \in \per(\zeta)$,
  \begin{equation*}
    T^*_\alpha (\pmb{f}_{\beta,\gamma+\alpha})
    = T^*_\alpha \left( \pi_{\gamma + \alpha} (\msf{f}(\beta)) \right)
    = \pi_{\gamma}\left( T^*_\alpha (\msf{f}(\beta)) \right),
  \end{equation*}
  since $T^*_{\alpha}: \Lambda_{\widehat{\VV}}^p (\til{U}_{\kappa}) \to \Lambda_{\widehat{\VV}}^p(\til{U}_{\kappa-\alpha})$ for every $\kappa \in \per(\zeta)$. Therefore, applying $\pi_{\gamma}$ to~\eqref{eq:delta_rel} yields
  \begin{equation}
    \label{eq:group_rel_coord}
    T^*_{\alpha} (\pmb{f}_{\beta,\gamma+\alpha}) - \pmb{f}_{\alpha+\beta,\gamma} + \pmb{f}_{\alpha,\gamma}=0,
    \quad \alpha,\beta,\gamma \in \per(\zeta).
  \end{equation}
  We would like to find $\pmb{u} \in \Lambda_{\widehat{\VV}}^p(g^{-1}(U))$ satisfying
  \begin{equation*}
    \msf{f}(\alpha) = T^*_\alpha \pmb{u} - \pmb{u},
    \quad \alpha \in \per(\zeta);
  \end{equation*}
  in components $\pmb{u} = (\pmb{u}_\alpha)_{\alpha \in \per(\zeta)}$, this condition reduces to
  \begin{equation*}
    \pmb{f}_{\alpha,\beta} = T^*_\alpha (\pmb{u}_{\alpha+\beta}) - \pmb{u}_{\beta},
    \quad \alpha,\beta \in \per(\zeta).
  \end{equation*}
  With the choice $\pmb{u}_\alpha \dfn -\pmb{f}_{-\alpha,\alpha}$, one verifies that the previous equation holds (by replacing $\beta$ with $-(\alpha+\beta)$ and $\gamma$ with $\beta$ in~\eqref{eq:group_rel_coord}).
\end{proof}

\begin{Lem}
  For every $p,q \geq 0$, the map 
  \begin{equation*}
    \dd_{\widehat{L}}^{p,q}: \msf{K}^{p,q} \lra \msf{K}^{p+1,q}
  \end{equation*}
  has closed range.
\end{Lem}
\begin{proof}
  Let $\{\msf{u}_\nu\}_{\nu \in \N} \sset \msf{K}^{p,q}$ be such that $\dd^{p,q}_{\widehat{L}} \msf{u}_\nu \to \msf{f}$ in $\msf{K}^{p+1,q}$, meaning that for every $(\alpha_1, \ldots, \alpha_q) \in \per(\zeta)^{q}$ we have, in the topology of $\Lambda_{\widehat{\VV}}^{p+1}(\R \times \widehat{L})$:
  \begin{equation*}
    \dd_{\widehat{L}} (\msf{u}_\nu (\alpha_1, \ldots, \alpha_q))
    = \left( \dd_{\widehat{L}}^{p,q} \msf{u}_\nu \right) (\alpha_1, \ldots, \alpha_q)
    \lra \msf{f} (\alpha_1, \ldots, \alpha_q).
  \end{equation*}
  By Proposition~\ref{prop:partial_derham_solvable}, there are $\pmb{u}_{\alpha_1, \ldots, \alpha_q} \in \Lambda_{\widehat{\VV}}^{p}(\R \times \widehat{L})$ such that $\dd_{\widehat{L}} \pmb{u}_{\alpha_1, \ldots, \alpha_q} = \msf{f} (\alpha_1, \ldots, \alpha_q)$. Defining $\msf{u} \in \msf{K}^{p,q}$ by $\msf{u}(\alpha_1, \ldots, \alpha_q) \dfn \pmb{u}_{\alpha_1, \ldots, \alpha_q}$ yields $ \dd_{\widehat{L}}^{p,q} \msf{u} = \msf{f}$ as desired.
\end{proof}

We shall also need the following elementary functional-analytic fact~\cite[p.~13]{kothe_tvs2}:
\begin{Lem}
  \label{lem:seq_invert}
  Let $T:E \to F$ be a continuous linear map between Fr{\'e}chet spaces. Then, $T$ has closed range if and only it is \emph{sequentially invertible}, that is, if for every sequence $\{x_\nu \}_{\nu \in \N} \sset E$ such that $Tx_n \to 0$, there is a sequence $\{y_\nu\}_{\nu \in \N} \sset E$ such that $Tx_\nu=Ty_\nu$ for each $\nu \in \N$ and $y_\nu \to 0$.
\end{Lem}
\begin{Rem}
  \label{rem:seq_invert}
  If $E$ and $F$ are just arbitrary locally convex spaces, there is an analogous result (for nets) when $T$ is a \emph{homomorphism} (i.e.,~$T$ has closed range and the operator induced on the quotient $T: E/ \ker(T) \to \ran(T)$ is an isomorphism). In fact, it is proved in~\cite[p.~18]{kothe_tvs2} that, for Fr{\'e}chet spaces, $T:E \to F$ has closed range if and only if it is an homomorphism.  
\end{Rem}
Define
\begin{equation*}
  E^{p} \dfn \ker \delta^{p,0} = \Lambda_{\widehat{\VV}}^{p}(\R \times \widehat{L})^{\per(\zeta)} \subset \msf{K}^{p,0}
\end{equation*}
and
\begin{equation*}
  Z^p \dfn \ker \dd^{p,0}_{\widehat{L}} = \left\{ \pmb{f} \in \Lambda_{\widehat{\VV}}^p(\R \times \widehat{L}) \st \dd_{\widehat{L}} \pmb{f} = 0 \right\}.
\end{equation*}
Now we can state and prove the main result of this section.
\begin{Thm}
  \label{thm:abstract_eq}
  Let $p \geq 0$. The map
  \begin{equation}
    \label{eq:dprime-complex}
    \dd_{\widehat{L}}: E^p \lra E^{p+1}
  \end{equation}
  has closed range if and only if the map
  \begin{equation}
    \label{eq:delta-complex}
    \delta: Z^p = \msf{C}^{0} \left( \per(\zeta) ; Z^{p} \right) \lra \msf{C}^{1} \left( \per(\zeta) ; Z^{p} \right)
  \end{equation}
  has closed range.
\end{Thm}
\begin{proof}
  Assume that~\eqref{eq:delta-complex} has closed range. Let $\{ \pmb{u}_\nu \}_{\nu \in \N} \sset E^p$ be such that $\dd_{\widehat{L}} \pmb{u}_\nu \to \pmb{f}$ in $E^{p + 1}$. By Proposition~\ref{prop:partial_derham_solvable}, we can find an $\pmb{u} \in \Lambda_{\widehat{\VV}}^{p}(\R \times \widehat{L})$ such that $\dd_{\widehat{L}} \pmb{u} = \pmb{f}$: our goal is to find $\pmb{u}^{\bb} \in E^p$ such that $\dd_{\widehat{L}}\pmb{u}^\bb = \dd_{\widehat{L}}\pmb{u}$.
  
  Since $\dd_{\widehat{L}}(\pmb{u} - \pmb{u}_\nu) \to 0$, and again by Proposition~\ref{prop:partial_derham_solvable}, we apply Lemma~\ref{lem:seq_invert} to find $\{ \pmb{v}_\nu \}_{\nu \in \N} \sset \Lambda_{\widehat{\VV}}^p(\R \times \widehat{L})$ such that $\pmb{v}_\nu \to 0$ and $\dd_{\widehat{L}} \pmb{v}_\nu = \dd_{\widehat{L}}(\pmb{u} - \pmb{u}_\nu)$, that is, $\pmb{w}_\nu \dfn \pmb{u} - \pmb{u}_\nu - \pmb{v}_\nu \in Z^p$ for every $\nu$. By continuity, it follows that
  \begin{equation*}
    \delta^{p,0} \pmb{w}_\nu = \delta^{p,0} \pmb{u} - \delta^{p,0} \pmb{v}_\nu \lra \delta^{p,0} \pmb{u}
    \quad \text{in $\msf{C}^{1} \left( \per(\zeta) ; Z^{p} \right)$}.
  \end{equation*}
  Since the range of~\eqref{eq:delta-complex} is closed, there is $\pmb{v} \in Z^p$ such that $\delta^{p,0} \pmb{v} = \delta^{p,0}\pmb{u}$: setting $\pmb{u}^ \bb \dfn \pmb{u} - \pmb{v}$, we obtain the result.
  
  Now assume that~\eqref{eq:dprime-complex} has closed range, and let $\{ \pmb{u}_\nu \}_{\nu \in \N} \sset Z^p$ be such that $\delta^{p,0} \pmb{u}_\nu \to \msf{f}$ in $\msf{C}^{1} \left( \per(\zeta) ; Z^{p} \right) \sset \msf{K}^{p,1}$. By Lemma~\ref{lem:delta_acyclic}, there is $ \pmb{u} \in \Lambda_{\widehat{\VV}}^p(\R \times \widehat{L})$ such that $\delta^{p,0} \pmb{u} = \msf{f}$: our goal is to find $\pmb{u}^\bb \in Z^{p}$ such that $\delta^{p,0} \pmb{u}^\bb = \delta^{p,0} \pmb{u}$.
  
  It also follows from Lemma~\ref{lem:delta_acyclic} that the range of $\delta^{p,0}: \msf{K}^{p,0} \to \msf{K}^{p,1}$ is closed, hence using that $\delta^{p,0} (\pmb{u} - \pmb{u}_\nu) \to 0$ we can apply Lemma~\ref{lem:seq_invert} to find $\{\pmb{v}_\nu \}_{\nu \in \N} \sset \Lambda_{\widehat{\VV}}^{p}(\R \times \widehat{L})$ such that $\pmb{v}_\nu \to 0$ and $\delta^{p,0} \pmb{v}_\nu = \delta^{p,0} (\pmb{u} - \pmb{u}_\nu)$, that is, $\pmb{w}_\nu \dfn \pmb{u} - \pmb{u}_\nu - \pmb{v}_\nu \in E^p$ for every $\nu$. By continuity, it follows that
  \begin{equation*}
    \dd_{\widehat{L}} \pmb{w}_\nu =  \dd_{\widehat{L}} \pmb{u} - \dd_{\widehat{L}} \pmb{v}_\nu \lra \dd_{\widehat{L}} \pmb{u}
    \quad \text{in $E^{p + 1}$}.
  \end{equation*}
  Since the range of~\eqref{eq:dprime-complex} is closed, there is $\pmb{v} \in E^p$ such that $\dd_{\widehat{L}} \pmb{v} = \dd_{\widehat{L}} \pmb{u}$: setting $\pmb{u}^ \bb \dfn \pmb{u} - \pmb{v}$, we obtain the result.
\end{proof}

\subsection{Proof of Theorem~\ref{thm:first_step_solv}: reduction}

Let us specialize to degree $p=0$. Observe that, since the leaf $\widehat{L}$ is connected, it follows that 
\begin{equation*}
  Z^{0} = \left\{f\in \cinfty(\R \times \widehat{L}) \st \dd_{\widehat{L}}f = 0 \right\} = \cinfty(\R).
\end{equation*}
The action of $\per(\zeta)$ on this space is simply
\begin{equation*}
  T_{\alpha}^* f = f(\cdot + \alpha),
  \quad \alpha \in \per(\zeta), \ f \in \cinfty(\R).
\end{equation*}
Combining Proposition~\ref{prop:pullback_map} and Theorem~\ref{thm:abstract_eq}, we conclude the following:
\begin{Cor}
  \label{cor:abstract_eq}
  The following are equivalent:
  \begin{enumerate}
  \item $\dd':\cinfty(\Omega)\to \Lambda^1_{\VV}(\Omega)$ is globally solvable.
  \item $\dd_{\widehat{L}}: \cinfty(\R \times \widehat{L})^{\per(\zeta)} \to \Lambda_{\widehat{\VV}}^1(\R \times \widehat{L})^{\per(\zeta)}$ has closed range.
  \item The mapping
    \begin{equation*}
      \TR{\delta}{f}{\cinfty(\R)}{\left( T_\alpha^* f - f \right)_{\alpha \in \per(\zeta)}}{\prod_{\alpha \in \per(\zeta)}\cinfty(\R)}
    \end{equation*}
    has closed range.
  \end{enumerate}
\end{Cor}
We can reduce to verifying the latter only for the generators of $\per(\zeta)$:
\begin{Cor}
  \label{cor:reduction_to_generators}
  $\dd':\cinfty(\Omega) \to \Lambda^1_{\VV}(\Omega)$ is globally solvable if and only if
  \begin{equation*}
    \TR{D}{f}{\cinfty(\R)}{\left( T_{a_j}^* f - f \right)_{j = 0, \ldots, r}}{\prod_{j = 0}^r \cinfty(\R)}
  \end{equation*}
  has closed range.
\end{Cor}
\begin{proof}
  From Corollary~\ref{cor:abstract_eq}, we have to show that $\delta$ has closed range if and only if $D$ has closed range. First, we prove the following claim: for each $\alpha \in \per(\zeta)$ there are $N \in \Z_+$, $c_{jl} = \pm 1$, and $d_{jl} \in \per(\zeta)$, for $1 \leq l \leq N$, $0 \leq j \leq r$, such that the identity
  \begin{equation}
    \label{eq:recurr_form}
    \left( T_\alpha^* f \right)(x) - f(x) = \sum_{j=0}^r \sum_{l=1}^N c_{jl} \left( T_{a_j}^*f - f \right)(x+d_{jl})
  \end{equation}
  holds for every $f\in \cinfty(\R)$ and for every $x\in \R$. We prove it by induction on the \emph{length} of $\alpha$, that is\footnote{Recall that $a_0 = 1, a_1, \ldots, a_r$ is a basis of $\per(\zeta)$ as a $\Z$-module.}, 
  \begin{equation*}
    |\alpha| \dfn \sum_{j=0}^r |n_j|
    \quad \text{where} \quad
    \alpha = \sum_{j=0}^r n_j a_j.
  \end{equation*}
  If $|\alpha| =1$, then $\alpha = \pm a_{j_0}$, for some $j_0 \in \{0, \ldots, r \}$. If $\alpha = a_{j_0}$, the result follows trivially; otherwise,
  \begin{align*}
    \left( T_{-a_{j_0}}^* f \right)(x) - f(x)
    &= f\left(x - a_{j_0}\right) - f(x) \\
    &= - \left\{ f \left( [x - a_{j_0}] + a_{j_0} \right) - f\left(x  - a_{j_0} \right) \right\} \\
    &= - \left( T_{a_{j_0}}^* f - f \right) \left(x - a_{j_0} \right). 
  \end{align*}
  Next assume the claim holds for every element in $\per(\zeta)$ of length at most $k$ and take $\alpha = \sum_{j=0}^r n_ja_j$, with $\sum_{j = 0}^{r} |n_{j}| = k + 1$. Note that there exists an $n_{j_1}$ which is either strictly positive or strictly negative. Assume first it is positive and let
  \begin{equation*}
    (n'_0,\ldots,n'_r) \dfn (n_0,\ldots,n_{j_1 - 1},n_{j_1} -1,n_{j_1+1},\ldots,n_r).
  \end{equation*}
  Applying the induction hypothesis to $\alpha' \dfn \sum_{j=0}^r n'_j a_j$, we obtain
  \begin{align*}
    \left( T_\alpha^* f \right)(x) - f(x)
    &= f(x + \alpha' + a_{j_1}) - f(x) \\
    &= f(x + \alpha' + a_{j_1}) - f(x + \alpha') + f(x + \alpha') - f(x) \\
    &= \left(T_{a_{j_1}}^* f \right)(x + \alpha') - f(x + \alpha') + \left( T_{\alpha'}^* f \right)(x) - f(x) \\
    &= \left(T_{a_{j_1}}^* f - f \right)(x + \alpha') + \sum_{j=0}^r \sum_{l=1}^{N'} c'_{jl} \left( T_{a_j}^*f - f \right)(x + d'_{jl})
  \end{align*}
  for certain $N', c_{jl}', d_{jl}'$, which is clearly of the form~\eqref{eq:recurr_form} for an appropriate choice of parameters. Similarly, if $n_{j_1} < 0$ we apply the induction hypothesis now to $\alpha'' \dfn \alpha + a_{j_1}$, obtaining
  \begin{align*}
    \left( T_\alpha^* f \right)(x) - f(x)
    &= f(x + \alpha'' - a_{j_1}) - f(x) \\
    &= f(x + \alpha'' - a_{j_1}) - f(x + \alpha'') + f(x + \alpha'') - f(x) \\
    &= \left(T_{-a_{j_1}}^* f \right)(x + \alpha'') - f(x + \alpha'') + \left( T_{\alpha''}^* f \right)(x) - f(x) \\
    &= - \left(T_{a_{j_1}}^* f - f \right) \left(x + \alpha'' - a_{j_1} \right) + \left( T_{\alpha''}^* f \right)(x) - f(x) \\
    &= - \left(T_{a_{j_1}}^* f - f \right) \left(x + \alpha \right) + \sum_{j=0}^r \sum_{l=1}^{N''} c''_{jl} \left( T_{a_j}^*f - f \right)(x + d''_{jl})
  \end{align*}
  once more of the form~\eqref{eq:recurr_form}; hence, by induction, proving the claim.
  
  We are now ready to prove the main statement. Assume $\delta$ has closed range and let $\{ f_\nu \}_{\nu \in \N}$ be a sequence of smooth functions such that 
  \begin{equation*}
    D(f_\nu) \to \msf{g} \dfn (g_j)_{j=0,\ldots,r}
    \Longrightarrow
    T^*_{a_j} f_\nu - f_\nu \to g_j
    \quad \text{in $\cinfty(\R)$ for all $j$}.
  \end{equation*}
  Using~\eqref{eq:recurr_form}, we deduce that for each $\alpha \in \per(\zeta)$ the sequence $\left\{ T_{\alpha}^* f_\nu - f_\nu \right\}_{\nu \in \N}$ has a limit in $\cinfty(\R)$, which implies that $\{ \delta(f_\nu) \}_{\nu \in \N}$ is convergent. Since the range of $\delta$ is (by hypothesis) closed, there exists $f \in \cinfty(\R)$ such that $\lim_{\nu \to \infty} \delta(f_\nu) = \delta(f)$, that is:
  \begin{equation*}
    T_{\alpha}^* f_\nu - f_\nu \lra T_{\alpha}^* f - f,
    \quad \forall \alpha \in \per(\zeta).
  \end{equation*}
  In particular:
  \begin{equation*}
    D(f) = \left(T_{a_j}^* f - f \right)_{j=0,\ldots,r} = \lim_{\nu \to \infty} D(f_{\nu}) = \msf{g}
  \end{equation*} 
  which implies that the range of $D$ is also closed. 
  
  Conversely, assume that the range of $D$ is closed and let $\{ f_\nu \}_{\nu \in \N} \sset \cinfty(\R)$ be such that $\delta(f_\nu) \to (g_\alpha)_{\alpha \in \per}$. In particular, $D(f_\nu) \to (g_{a_j})_{j = 0, \ldots, r}$, therefore there exists $f \in \cinfty(\R)$ such that  $g_{a_j} = T_{a_j}^* f - f$ for every $j = 0,\ldots,r$. Given $\alpha \in \per(\zeta)$, we apply~\eqref{eq:recurr_form} and obtain
  \begin{equation*}
    \left( T_\alpha^* f_\nu \right)(x) - f_\nu(x) = \sum_{j=0}^r \sum_{l=1}^N c_{jl} \left( T_{a_j}^*f_\nu - f_\nu \right)(x+d_{jl})
  \end{equation*}
  where $N, c_{jl},d_{jl}$ are independent of $\nu$ and $x$. In the limit as $\nu \to \infty$:
  \begin{equation*}
    g_\alpha(x)
    = \sum_{j=0}^r \sum_{l=1}^N c_{jl} \left( T_{a_j}^*f - f \right)(x+d_{jl})
    = \left( T_\alpha^* f \right)(x) - f(x);
  \end{equation*}
  hence, $\delta(f) = (g_{\alpha})_{\alpha \in \per(\zeta)}$, which completes the proof.  
\end{proof}

Recall~\cite[Lemma~2.1]{meziani02} that $\zeta$ is irrational non-Liouville (cf.~Definition~\ref{def:liouville}) if and only if $(a_1,\ldots,a_r)$ is a \emph{non-Liouville vector}, according to the following:
\begin{Def}\label{def:liouville_vector} A vector $(a_1,\ldots,a_r)\in \R^r \setminus \Q^r$ is \emph{Liouville} if there exist $C>0$ and sequences $\{ p_\nu\}_{\nu \in \N} \sset \Z^r$ and $\{ q_\nu \}_{\nu \in \N} \sset \N$, where $q_\nu \geq 2$ is strictly increasing, such that 
  \begin{equation*}
    \sum_{j=1}^r \left| a_j - \frac{p_{\nu,j}}{q_\nu} \right| < \frac{C}{q_\nu^{\nu}},
    \quad \forall \nu \in \N,
  \end{equation*}
  where $p_\nu = (p_{\nu,1}, \ldots, p_{\nu,r})$.
\end{Def}
So, in order to prove Theorem \ref{thm:first_step_solv}, it is enough to prove the following:
\begin{Thm}
  \label{thm:first_step_solv_redux}
  Let $1, a_1, \ldots, a_r$ be positive real numbers, linearly independent over $\Q$. Then, the following are equivalent:
  \begin{enumerate}
  \item \label{thm:first_step_solv_redux-it1} $D$ has closed range.
  \item \label{thm:first_step_solv_redux-it2} $(a_1,\ldots,a_r)$ is a non-Liouville vector.
  \end{enumerate}
\end{Thm}

\subsection{Proof of Theorem~\ref{thm:first_step_solv_redux}}

\paragraph{$\eqref{thm:first_step_solv_redux-it1} \Longrightarrow \eqref{thm:first_step_solv_redux-it2}$} We shall work in the contrapositive: assuming that $(a_1,\ldots,a_r)$ is a Liouville vector, let us show that $D$ does not have closed range. It can be easily seen (for example, \cite[Lemma 3.4]{Bergamasco2015} or \cite[Lemma 8.1]{dm16}) that our assumption entails the existence of a sequence $\{ \xi_\nu \}_{\nu \in \N} \sset \Z \setminus \{0\}$, with $|\xi_\nu| < |\xi_{\nu + 1}|$ for every $\nu \in \N$, and a constant $C>0$ such that
\begin{equation}
  \label{eq:lem_liouville}
  \max_{1 \leq j \leq r} \left|e^{2\pi i a_j \xi_\nu} - 1 \right| < \frac{1}{|\xi_\nu|^\nu},
  \quad \forall \nu \in \N.
\end{equation}
Given $N\in \N$, let
\begin{equation*}
  f_N(x) \dfn \sum_{\nu=1}^N e^{2\pi i x \xi_\nu},
  \quad x \in \R.
\end{equation*}
Clearly, $T_1^* f_N - f_N = 0$. If $j=1,\ldots,r$, then
\begin{equation*}
  \left( T_{a_j}^* f_N - f_N \right)(x) = \sum_{\nu = 1}^N \left( e^{2\pi i a_j \xi_\nu} - 1 \right) e^{2\pi i x \xi_\nu}.
\end{equation*}
Consider next 
\begin{equation*}
  g_j(x) \dfn \sum_{\nu = 1}^\infty \left( e^{2\pi i a_j \xi_\nu} - 1 \right) e^{2\pi i x \xi_\nu},
  \quad x \in \R.
\end{equation*}
Using inequality~\eqref{eq:lem_liouville}, it is not difficult to prove that $g_{j} \in \cinfty(\R)$ and 
\begin{equation*}
  T_{a_j}^* f_N - f_N \lra  g_{j}
  \quad \text{in $\cinfty(\R)$ as $N \to \infty$}
\end{equation*}
for each $j = 1, \ldots, r$. We claim there is no $f \in \cinfty(\R)$ such that $D(f) = (0,g_1,\ldots,g_r)$. Indeed, if there were such a function, it would have to be $1$-periodic, since $T_1 f - f = 0$. Writing its Fourier series
\begin{equation*}
  f(x) = \sum_{\xi\in \Z} \hat{f}(\xi) \ e^{2\pi i x \xi},
  \quad x \in \R,
\end{equation*}
we would have, for any $j=1,\ldots,r$ and $x \in \R$, 
\begin{equation*}
  \sum_{\xi\in \Z} \left( e^{2\pi i a_j \xi} - 1 \right) \hat{f}(\xi) \ e^{2\pi i x \xi}
  =  \sum_{\nu = 1}^\infty \left( e^{2\pi i a_j \xi_\nu} - 1 \right) e^{2\pi i x \xi_\nu}.
\end{equation*}
In particular, by uniqueness of Fourier coefficients, we would have $\hat{f}(\xi_\nu) = 1$ for all $\nu \in \N$, which contradicts the smoothness of $f$.

\paragraph{$\eqref{thm:first_step_solv_redux-it2} \Longrightarrow \eqref{thm:first_step_solv_redux-it1}$} Again, we assume that $a_1,\ldots,a_r$ are positive and such that $\{1,a_1,\ldots,a_r\}$ is linearly independent over $\Q$. Suppose $(a_1,\ldots,a_r)$ is non-Liouville.
We shall show $D$ has closed range by applying the \emph{Homomorphism Theorem for Fr{\'e}chet spaces}~\cite[p.~18]{{kothe_tvs2}}:
\begin{quote}
  \emph{$D$ has closed range if and only if its transpose $\transp{D}$ has weakly closed range.}
\end{quote}

Let us determine $\transp{D}$: if $(u_0,\ldots,u_r) \in \bigoplus_{j = 0}^r \E'(\R)$, for any given $\phi \in \cinfty(\R)$, one has
\begin{equation*}
  \left \langle (u_0,\ldots,u_r), D (\phi) \right \rangle
  = \sum_{j=0}^r \left \langle u_{j}, \phi(\cdot + a_{j}) - \phi \right \rangle
  =  \sum_{j=0}^r \left \langle u_{j} * ( \delta_{a_j} - \delta_0), \phi \right\rangle,
\end{equation*}
where $\delta_{x}$ represents the Dirac mass centered at $x \in \R$. Therefore:
\begin{equation*}
  \TR{\transp{D}}{(u_0,\ldots,u_r)}{\bigoplus_{j = 0}^r \E'(\R)}{\sum_{j=0}^r u_{j} * ( \delta_{a_j} - \delta_0)}{\E'(\R)}
\end{equation*}
We claim next that its range is
\begin{equation*}
  \ran (\transp{D}) = \left\{ v \in \E'(\R) \st \hat{v}(0)=0 \right\},
\end{equation*}
where $\hat{v}$ is the Fourier-Laplace transform of $v$, which we normalize as 
\begin{equation*}
  \hat{v}(z) \dfn \langle v, e^{-ix z} \rangle,
  \quad z \in \C.
\end{equation*}
If we verify this we are done, since that subspace is obviously weakly closed. Note first that if $v = \transp{D}(u_0, \ldots, u_r)$ then
\begin{equation*}
  \hat{v}(z)
  = \sum_{j=0}^r \hat{u}_{j}(z) \left( \hat{\delta}_{a_j}(z) - \hat{\delta}_0(z) \right)
  = \sum_{j=0}^r \hat{u}_{j}(z) (e^{-ia_jz} - 1),
  \quad \forall z \in \C,
\end{equation*}
which implies that $\ran (\transp{D}) \sset \{ v \in \E'(\R) \st \hat{v}(0)=0 \}$.

In order to prove the reverse inclusion, let $v \in \E'(\R)$ be such that $\hat{v}(0)=0$. From the Paley-Wiener-Schwartz Theorem, $\hat{v}\in \Ol(\C)$ is an entire function of exponential type which is bounded by a polynomial on $\R$. We factor it
\begin{equation*}
  \hat{v}(z) = z h(z),
  \quad z \in \C,
\end{equation*}
where $h\in \Ol(\C)$ is also of exponential type and bounded by a polynomial on $\R$. Therefore, in order to solve 
\begin{equation*}
  \sum_{j=0}^r u_{j} * ( \delta_{a_j} - \delta_0) = v,
  \quad u_0, \ldots, u_r \in \E'(\R),
\end{equation*}
we must solve 
\begin{equation*}
  \sum_{j=0}^r \left( \frac{e^{-ia_j z} - 1}{z} \right) g_j(z) = h(z),
\end{equation*}
where $g_0, \ldots, g_r \in \Ol(\C)$ are of exponential type, each bounded by a polynomial on $\R$; or, equivalently,
\begin{equation}
  \label{eq:ideal_eq}
  \sum_{j=0}^r \left( \frac{e^{ia_j z} - 1}{z} \right) f_j(z) = h(-z),
\end{equation}
for $f_0, \ldots, f_r$ of the same type. This kind of ``ideal equation'' was studied by L.~Ehrenpreis (see~\cite[Theorem~11.2]{ehrenpreis} and the discussion that follows it):
\begin{Thm}
  \label{thm:ehrenpreis}
  Let $P_0, \ldots, P_r \in \Ol(\C)$ be functions of exponential type, bounded by a polynomial on $\R$. Assume that there are constants $C,N,\epsilon>0$ such that
  \begin{equation*}
    \sum_{j=0}^r |P_j(z)| \geq \frac{Ce^{-\epsilon |\Im z|}}{(1+|z|)^{N}},
    \quad \forall z \in \C.
  \end{equation*}
  Then, there exist functions $F_0, \ldots, F_r \in \Ol(\C)$ of exponential type, bounded by a polynomial on $\R$, satisfying
  \begin{equation*}
    \sum_{j=0}^r P_j(z) F_j(z) = 1,
    \quad \forall z \in \C.
  \end{equation*}
\end{Thm}
Therefore, in order solve~\eqref{eq:ideal_eq} it is sufficient to prove:
\begin{Thm}
  \label{thm:main_result_suff}
  Let $1,a_1,\ldots,a_r$ be positive real numbers, linearly independent over $\Q$ and such that $(a_1,\ldots,a_r)$ is not a Liouville vector. Then, there exist constants $C,N,\epsilon>0$ such that 
  \begin{equation}
    \label{eq:main_result_suff}
    \left|\frac{e^{iz}-1}{z}\right|+ \sum_{j=1}^r \left|\frac{e^{ ia_j z} - 1}{z} \right| \geq \frac{Ce^{-\epsilon |\Im z|}}{(1+|z|)^N},
    \quad \forall z \in \C.
  \end{equation}
\end{Thm}
Indeed, in that case Theorem~\ref{thm:ehrenpreis} furnishes $F_0, \ldots, F_r \in \Ol(\C)$ of exponential type, bounded by a polynomial on $\R$, satisfying
\begin{equation}
  \label{eq:ideal2}
  \sum_{j=0}^r \left( \frac{e^{ia_j z} - 1}{z} \right) F_j(z) = 1.
\end{equation}
Then
\begin{equation*}
  f_{j}(z) \dfn F_j(z) h(-z),
  \quad z \in \C, \ j=0, \ldots, r,
\end{equation*}
are again entire functions of exponential type, bounded by a polynomial on $\R$, but now solve~\eqref{eq:ideal_eq}. We need a preliminary result:
\begin{Lem}
  \label{lem:sin_condition}
  There are constants $C,N>0$ such that
  \begin{equation}
    \label{eq:sin_condition}
    \left|\frac{\sin(x)}{x} \right|+\sum_{j=1}^r \left|\frac{\sin(a_j x)}{x} \right| \geq \frac{C}{(1+|x|)^N},
    \quad \forall x \in \R.
  \end{equation}
\end{Lem}
\begin{Rem} The $r=1$ version of this result appeared in \cite[Theorem 3.18]{christensen2024}.
\end{Rem}
\begin{proof}
  Note that the left-hand side of~\eqref{eq:sin_condition} is never vanishing (since $a_j$ is irrational for $j=1,\ldots,r$), so it attains a positive minimum on every compact interval of the line. Moreover, the functions $\sin(a_jx)/x$ are even. Therefore, it is enough to prove an estimate of the kind~\eqref{eq:sin_condition} for $x>2\pi$. 

  Since $(a_1,\ldots,a_r)$ is a non-Liouville vector there exists $m\in \N$ such that
  \begin{equation}
    \label{eq:dioph_def}
    \max_{1 \leq j \leq r} \left| a_j - \frac{p_j}{q} \right| \geq \frac{1}{q^m},
    \quad \forall (p_1,\ldots,p_r) \in \Z^r, \ q \geq 2.
  \end{equation}
  Fix $x>2\pi$ and let $q,p_1,\ldots,p_r$ be integers which minimize the distance to $x/\pi,a_1 x/\pi,\ldots,a_r x/\pi$, respectively. Since $x>2\pi$ we have $q\geq 2$, and also
  \begin{equation}
    \label{eq:dioph_lem}
    \left| \frac{x}{\pi} - q \right| \leq \frac{1}{2},
    \quad \left|\frac{x}{\pi} - \frac{p_j}{a_j} \right| \leq \frac{1}{2a_j},
    \quad 1 \leq j \leq r.
  \end{equation}
  From~\eqref{eq:dioph_def}, there exists $j_{0}\in \{1,\ldots,r\}$ such that
  \begin{equation*}
    \frac{1}{q^m} \leq \left|a_{j_{0}}-\frac{p_{j_{0}}}{q} \right|
    \Longrightarrow
    \frac{1}{a_{j_{0}}q^{m-1}} \leq \left|q-\frac{p_{j_{0}}}{a_{j_{0}}} \right| \leq \left|\frac{x}{\pi}-q \right| + \left|\frac{x}{\pi}-\frac{p_{j_{0}}}{a_{j_{0}}} \right|
  \end{equation*}
  and, consequently,
  \begin{equation*}
    \text{either}
    \quad
    \left|\frac{x}{\pi}-q \right| \geq \frac{1}{2a_{j_{0}}q^{m-1}}
    \quad
    \text{or}
    \quad
    \left|\frac{x}{\pi}-\frac{p_{j_{0}}}{a_{j_{0}}} \right| \geq \frac{1}{2a_{j_{0}}q^{m-1}}.
  \end{equation*}
  Writing $D \dfn \min\left\{\frac{\pi}{2}, \frac{\pi}{2 a_{1}}, \ldots, \frac{\pi}{2 a_{r}}\right\} > 0$, we conclude that
  \begin{equation}
    \label{alternatives}
     \text{either}
     \quad
     \left| x - q\pi \right| \geq \frac{D}{q^{m-1}}
     \quad
     \text{or}
     \quad
     \left| a_{j_{0}} x - p_{j_{0}}\pi \right| \geq \frac{D}{q^{m-1}}.
  \end{equation}
  Note that, even though $j_{0}$ depends on $x$, the constant $D$ \emph{does not}.

  Moreover, from~\eqref{eq:dioph_lem} it follows that 
  \begin{equation}
    \label{small distance}
    | x - q \pi | \leq \frac{\pi}{2},
    \quad | a_j x - p_j \pi | \leq \frac{\pi}{2},
    \quad 1 \leq j \leq r,
  \end{equation}
  and since $|\sin t| \geq |t|/2$ if $|t| \leq \pi/2$, it follows from~\eqref{alternatives}-\eqref{small distance} that
  \begin{align*}
    \left| \frac{\sin x}{x} \right| + \sum_{j=1}^r \left| \frac{\sin(a_jx)}{x} \right|
    &= \left| \frac{\sin (x - q\pi)}{x} \right| + \sum_{j = 1}^r \left| \frac{\sin(a_j x - p_j\pi)}{x} \right| \\
    &\geq \left| \frac{\sin(x - q\pi)}{x} \right| + \left| \frac{\sin(a_{j_{0}} x - p_{j_{0}}\pi)}{x} \right| \\
    &\geq \frac{1}{x} \max \left\{ | \sin(x - q\pi)|, |\sin(a_{j_{0}} x - p_{j_{0}}\pi)| \right\} \\
    &\geq \frac{1}{2x}\max \left\{ |x - q\pi|, |a_{j_{0}}x - p_{j_{0}}\pi|  \right\} \\
    &\geq \frac{D/2}{(1 + x) q^{m-1}}.
  \end{align*}
  Finally, note that~\eqref{eq:dioph_lem} further implies that
  \begin{equation*}
    q \leq \frac{x}{\pi} + \frac{1}{2} \leq \frac{1}{2} (1+x) \Longrightarrow \frac{1}{q^{m-1}} \geq \frac{2^{m-1}}{(1+x)^{m-1}}
  \end{equation*}
  hence
  \begin{equation*}
    \left| \frac{\sin x}{x} \right| + \sum_{j=1}^r \left| \frac{\sin(a_jx)}{x} \right| \geq \frac{D 2^{m - 2}}{(1+x)^{m}},
    \quad \forall x > 2\pi.
  \end{equation*}
\end{proof}

\begin{proof}[Proof of Theorem~\ref{thm:main_result_suff}]
  Observe that the left-hand side of \eqref{eq:main_result_suff} is never vanishing. Indeed, note that $z = 0$ is a removable singularity of $\frac{e^{iz}-1}{z}$ and its extension equals $i$ at the origin, which proves the assertion for this particular case.  If it vanished at any other point, one would have
  \begin{equation*}
    e^{i a_j z} = 1,
    \quad \forall j=0,\ldots,r.
  \end{equation*}
  When $j = 0$, the identity above would imply that $z = 2 \pi k$ for some $k \in \Z \setminus \{0\}$. On the other hand, the case $j = 1$ already would allow us to deduce a contradiction:
  \begin{equation*}
    \text{$a_{1} z = 2\pi \ell$ for some $\ell \in \Z$} \Longrightarrow  a_{1} = \frac{\ell}{k} \in \Q.
  \end{equation*}
  Therefore, by compactness it suffices to prove the result for $|z| \geq R$, where $R>0$ is yet to be chosen. Let $z = x + iy \in \C$ and $j \in \{0,\ldots,r\}$. Then
  \begin{equation*}
    \left|e^{i a_j z} - 1 \right| \geq \left| \Im \left(e^{i a_j z} \right)\right| = e^{- a_{j} y} |\sin (a_{j} x ) | \geq e^{- a_{j} |y|} |\sin (a_{j} x )|. 
  \end{equation*}
  Let $\alpha \dfn \max \{a_0, a_{1}, \ldots, a_{r} \} \geq 1$. By summing the latter with respect to $j$ and applying Lemma~\ref{lem:sin_condition}, we deduce that
  \begin{equation}
    \label{first estimate}
    \sum_{j = 0}^{r} \left|e^{i a_j z} - 1 \right|
    \geq e^{- \alpha |y|} \sum_{j = 0}^{r} |\sin (a_{j} x) | 
    \geq C e^{- \alpha |y|} \frac{|x|}{(1+|x|)^N}
  \end{equation} 
  Note that the estimate does not completely provide~\eqref{eq:main_result_suff} yet, since $x$ may vanish. Nevertheless, suppose the existence $c > 0$ for which $|x| \geq c$. Using the inequality $e^t \geq t$ for $t \geq 0$, we conclude that
  \begin{equation*}
    e^{- \alpha |y|} |x|
    \geq c e^{- \alpha |y|} 
    \geq c \alpha e^{- 2\alpha |y|}  |y|
    \geq e^{- 2\alpha |y|}  |y|,
  \end{equation*} 
  assuming that $c\alpha \geq 1$. By splitting both expressions and putting the inequality into~\eqref{first estimate}, we deduce that for $|x| \geq 1/\alpha$:
  \begin{align}
    \label{estimate real part big}
    \sum_{j = 0}^{r} \left| e^{i a_j z} - 1 \right|
    &\geq \frac{C}{(1+|x|)^N} \left[ \frac{1}{2} e^{- \alpha |y|} |x| + \frac{1}{2} e^{- \alpha |y|} |x| \right] \nonumber \\
    &\geq \frac{C}{(1+|x|)^N} \left[ \frac{1}{2} e^{-\alpha |y|} |x| + \frac{1}{2} e^{- 2\alpha |y|} |y| \right] \nonumber \\
    &\geq \frac{C}{(1+|x|)^N} \left[ \frac{1}{2} e^{-2\alpha |y|} |x| + \frac{1}{2} e^{- 2\alpha |y|} |y| \right] \nonumber \\
    &\geq \frac{C}{2(1+|x|)^N} e^{- 2\alpha |y|} |z| \nonumber \\
    &\geq \frac{C e^{- 2\alpha |y|}}{2(1+|z|)^N} |z|. 
  \end{align}

  Next, in order to deal with the case $|x| \leq 1/\alpha$, we must estimate with respect to the real part:
  \begin{equation}
    \label{real part estimate}
    \left|e^{i a_j z} - 1 \right| \geq \left| \Re \left(e^{i a_j z} \right) - 1 \right| = \left| e^{- a_{j} y}  \cos (a_{j}x) - 1  \right|.
  \end{equation} 
  Since $|x| \leq \frac{1}{\alpha}$, we have:
  \begin{equation*} 
    a_{j} |x| \leq 1 \leq \frac{\pi}{3} \Longrightarrow \cos (a_{j} x) \geq \frac{1}{2},
    \quad \forall j \in \{0, \ldots, r\}.
  \end{equation*}

  Let $\beta \dfn \min \{a_0, a_{1}, \ldots, a_{r} \} > 0$. If $y < 0$, it follows that
  \begin{equation}
    \label{estimate real part exponential}
    e^{- a_{j} y}  \cos (a_{j} x) = e^{a_{j} |y|} \cos (a_{j} x) \geq \frac{e^{\beta |y|}}{2}.  
  \end{equation} 
  Now note that if $|z| \geq  R$ then $\max\{|x|, |y| \} \geq R / \sqrt{2}$; in particular, if $R > \sqrt{2}$ and $|x| \leq \frac{1}{\alpha} \leq 1$, it follows that $|y| \geq \frac{|z|}{\sqrt{2}} \geq \frac{R}{\sqrt{2}}$. Therefore, 
  \begin{equation*}
    |z| \geq \frac{2\sqrt{2}}{\beta} \Longrightarrow |y| \geq \frac{|z|}{\sqrt{2}} \geq \frac{2}{\beta};
  \end{equation*}
  and, in that case, we infer from~\eqref{estimate real part exponential} that
  \begin{equation*}
    e^{- a_{j} y}  \cos (a_{j} x) - 1  \geq \frac{e^{2}}{2} - 1 > 1. 
  \end{equation*}
  By associating the inequality above to~\eqref{real part estimate} and~\eqref{estimate real part exponential}, we deduce that
  \begin{equation*} 
    \left|e^{i a_j z} - 1 \right| 
    \geq \frac{e^{\beta |y|}}{2} - 1
    \geq \frac{1}{2}(e^{\beta |y|} - 1)
    = e^{\frac{\beta |y|}{2}} \sinh \left(\frac{\beta |y|}{2}\right).
  \end{equation*}
  Since $\sinh (t) \geq t$ for $t \geq 0$, we deduce that
  \begin{equation*} 
    \left|e^{i a_j z} - 1 \right| 
    \geq \frac{\beta |y|}{2}   e^{\frac{\beta |y|}{2}}
    \geq \frac{\beta}{2 \sqrt{2}} |z| e^{\frac{\beta |y|}{2}}
    \geq  \frac{\beta}{2 \sqrt{2}} |z| e^{\frac{- \beta |y|}{2}}.
  \end{equation*}
  By summing with respect to $j$ we get, when $y < 0$ and $|z| \geq \frac{2\sqrt{2}}{\beta}$:
  \begin{equation}
    \label{estimate imaginary part negative}     
    \sum_{j = 0}^{r} \left|e^{i a_j z} - 1 \right| \geq  \frac{(r+1)\beta |z|}{2 \sqrt{2}} e^{\frac{- \beta |y|}{2}}.
  \end{equation}

  Finally, we consider the case where $y > 0$. In this context, 
  \begin{equation*}
    e^{- a_{j} y}  \cos (a_{j} x) \leq e^{- a_{j} y} < 1, 
  \end{equation*}
  which in addition to~\eqref{real part estimate} implies that
  \begin{multline*}
    \left|e^{i a_j z} - 1 \right|
    \geq  1 - e^{- a_{j} y}  \cos (a_{j} x)
    \geq 1 - e^{- \alpha y} \\
    = 2 e^{- \frac{\alpha y}{2}} \sinh \left( \frac{\alpha y}{2} \right)
    \geq \alpha y e^{- \frac{\alpha y}{2}}
    \geq \frac{\alpha |z|}{\sqrt{2}} e^{- \frac{\alpha y}{2}}.
  \end{multline*}
  By summing with respect to $j$, we deduce for $y > 0$ and $|z| \geq \frac{2\sqrt{2}}{\beta}$:
  \begin{equation}
    \label{estimate imaginary part positive}     
    \sum_{j = 0}^{r} \left|e^{i a_j z} - 1 \right|
    \geq  \frac{(r+1)\alpha |z|}{\sqrt{2}} e^{\frac{- \alpha |y|}{2}}.
  \end{equation}
  By putting together~\eqref{estimate real part big}, \eqref{estimate imaginary part negative} and~\eqref{estimate imaginary part positive}, we obtain~\eqref{eq:main_result_suff}.
\end{proof}

\section{The irrational case in top degree}
\label{sec:irr_degn}

In this section, we work on the degree $n$ part of Theorem~\ref{thm:degs1n_mainthm}, which we restate here for convenience:
\begin{Thm}
  \label{thm:last_step_solv}
  Let $\Omega$ be an oriented $(n+1)$-dimensional compact manifold and $\zeta \in F^1(\Omega)$ a closed, non-singular real $1$-form. Then, the following are equivalent:
  \begin{enumerate}
  \item $\zeta$ is globally solvable in degree $n$.
  \item $\zeta$ is either rational or an irrational non-Liouville form.
  \end{enumerate}
\end{Thm}
\begin{Rem}
  This proves the equivalence $\eqref{thm:degs1n_mainthm-it2} \Longleftrightarrow \eqref{thm:degs1n_mainthm-it3}$ in Theorem~\ref{thm:degs1n_mainthm}, and consequently completes its proof.
\end{Rem}

We proceed by duality, which requires us to study the $\dd'$ complex (in the first level) acting on distributions, or currents. The strategy is the same as the one used to prove the result in the first degree, so we shall start with some preliminaries regarding the differential complexes induced by closed, non-singular $1$-forms acting on these generalized objects. Our main reference regarding the theory of currents is~\cite[Section~2, Chapter~I]{demailly}.

\subsection{The differential complex acting on currents}

For $\Omega$ a smooth, connected and oriented $(n+1)$-dimensional manifold (not necessarily compact\footnote{We shall apply these results in the non-compact case only for $\Omega=\R \times \widehat{L}$ and $\zeta=\dd t$, where $(t,x)\in \R \times \widehat{L}$.}), let $F^q_c(\Omega)$ be the space of compactly supported smooth $q$-forms on $\Omega$, endowed with the usual strict LF topology. We write
\begin{equation*}
  \D'(\Omega;\Lambda^q) \dfn F^{n+1-q}_c(\Omega)'
\end{equation*}
for the space of $q$-currents on $\Omega$. We shall consider mostly the \emph{weak dual topology} on $\D'(\Omega;\Lambda^q)$, with a basis of seminorms given by
\begin{equation*}
  \rho^{\sigma}_S(u) \dfn \sup_{\omega \in S} |\langle u, \omega\rangle|,
  \quad u \in \D'(\Omega;\Lambda^q),
\end{equation*}
where $S \subset F^{n+1-q}_c(\Omega)$ is finite. It will also be important to consider on $\D'(\Omega;\Lambda^q)$ the \emph{strong dual topology}, with a basis of seminorms given by
\begin{equation*}
  \rho^{b}_S(u) \dfn \sup_{\omega \in S} | \langle u,\omega \rangle |,
  \quad u \in \D'(\Omega;\Lambda^q),
\end{equation*}
where $S \subset F^{n+1-q}_c(\Omega)$ is a bounded set. It is well-known that elements of $\D'(\Omega;\Lambda^q)$ can be represented locally as $q$-forms with distributional coefficients~\cite[Proposition~(2.9), Chapter~I]{demailly}.

Let $\zeta \in F^1(\Omega)$ be a real, closed, non-singular 1-form defined on $\Omega$. Take
\begin{equation*}
  \NN^q_\VV(\Omega;\D') \dfn
  \left\{ u\in \D'(\Omega; \Lambda^q) \st \langle u, \zeta \wedge \omega \rangle = 0, \ \forall \omega \in F^{n-q}_c(\Omega) \right\}
\end{equation*}
for $0 \leq q \leq n$, whereas $\NN^{n+1}_\VV(\Omega;\D') \dfn \{0\}$. For each $q$, this is a weakly (and therefore strongly) closed subspace of $\D'(\Omega;\Lambda^q)$: we define
\begin{equation*}
  \Lambda^q_{\VV}(\Omega;\D') \dfn
  \frac{\D'(\Omega;\Lambda^q)}{\NN^q_\VV(\Omega;\D')},
  \quad 0 \leq q \leq n + 1.
\end{equation*}
Clearly, this construction can be localized to open subsets $U \subset \Omega$; this procedure produces the space of distribution sections of the vector bundle $\Lambda^q_{\VV} \to \Omega$ over $U$.

Note also that, since $\zeta$ is closed, the exterior derivative satisfies
\begin{equation*}
  \langle \dd u,\zeta \wedge \omega \rangle
  = (-1)^{q+1}\langle u,\dd (\zeta \wedge \omega) \rangle
  = (-1)^{q}\langle u,\zeta \wedge \dd \omega \rangle
\end{equation*}
for all $\omega \in F^{n-q - 1}_c(\Omega)$, which equals zero provided $u \in \NN^q_\VV(\Omega;\D')$; therefore, it descends to the quotient and we obtain a differential complex
\begin{equation*}
  \dd'_q:\Lambda^q_\VV(\Omega;\D') \lra \Lambda^{q+1}_\VV(\Omega;\D'),
  \quad 0 \leq q \leq n.
\end{equation*}
\begin{Prop}
  \label{prop:duality_basic}
  The following are equivalent:
  \begin{enumerate}
  \item \label{it:1noprime} $\dd'_{n-1}:\Lambda^{n-1}_{c,\VV}(\Omega) \to \Lambda^n_{c,\VV}(\Omega)$ has closed range.
  \item \label{it:2noprime} $\dd'_0:\Lambda^0_\VV(\Omega;\D') \to \Lambda^1_\VV(\Omega;\D')$ is a homomorphism for the weak dual topology.
  \end{enumerate}
  Moreover, if $\Omega$ is compact, then the following are equivalent:
  \begin{enumerate}[label=(\arabic*')]
  \item \label{it:1prime} $\dd'_{n-1}:\Lambda^{n-1}_\VV(\Omega)\to \Lambda^n_{\VV}(\Omega)$ has closed range.
  \item \label{it:2prime} $\dd'_0:\Lambda^0_\VV(\Omega;\D')\to \Lambda^1_\VV(\Omega;\D')$ has weakly closed range.
  \item \label{it:3prime} $\dd'_0:\Lambda^0_\VV(\Omega;\D')\to \Lambda^1_\VV(\Omega;\D')$ has strongly closed range.
  \end{enumerate}
\end{Prop}
\begin{Rem}
  Note that $\NN^0_\VV(\Omega;\D') = \{0\}$ (since every top form is locally of the form $\zeta \wedge \omega$), so $\Lambda^0_\VV(\Omega; \D') \simeq \D'(\Omega; \Lambda^0) \simeq \cinfty_c(\Omega)' \simeq \D'(\Omega)$. We also let
  \begin{equation*}
    \Lambda^{q}_{c,\VV}(\Omega) \dfn
    \frac{F^q_c(\Omega)}{\NN^q_{\VV}(\Omega) \cap F^q_c(\Omega)}
  \end{equation*}
  endowed with the quotient topology; note that $\NN^q_\VV(\Omega) \cap F^q_c(\Omega)$ is a closed subspace of $F^q_c(\Omega)$.
\end{Rem}
\begin{proof}
  In this proof, it will be necessary to be more explicit about elements of $\Lambda^q_{\VV}(\Omega;\D'),\Lambda^q_\VV(\Omega)$ and their representatives. We shall use the notation $[u]\in \Lambda^{q}_\VV(\Omega;\D')$, where $u\in \D'(\Omega;\Lambda^q)$ is a representative, and $[\omega]\in \Lambda^q_{\VV}(\Omega)$ (respectively, $[\omega]\in \Lambda^q_{c,\VV}(\Omega)$), where $\omega \in F^{q}(\Omega)$ (respectively, $\omega \in F^q_c(\Omega)$) is a representative.

  Let $0 \leq q \leq n$ and consider the bilinear pairing
  \begin{equation*}
    \TR{\mathfrak{B}_q}
    {\left( [\omega], [u] \right)}{\Lambda^{q}_{c,\VV}(\Omega) \times \Lambda^{n-q}_{\VV}(\Omega; \D')}
    {\langle u,\zeta \wedge \omega \rangle}{\C}
  \end{equation*}
  which is well-defined. Indeed, let $\omega_1,\omega_2 \in F^{q}_c(\Omega)$ be such that $\omega_1-\omega_2 \in \NN^{q}_\VV(\Omega)$ and consider $u_1,u_2 \in \D'(\Omega;\Lambda^{n-q})$ such that $u_1-u_2 \in \NN^{n-q}_\VV(\Omega;\D')$. Observe that $\zeta \wedge (\omega_1-\omega_2) \equiv 0$, since $\omega_1 - \omega_2$ is, locally, a multiple of $\zeta$. Then,
    \begin{equation*}
      \langle u_1,\zeta\wedge \omega_1 \rangle = \langle u_1,\zeta \wedge \omega_2 \rangle = \langle u_2,\zeta \wedge \omega_2 \rangle,
    \end{equation*}
    since $\langle u_1-u_2,\zeta \wedge \omega_2 \rangle = 0$. Given $[u]\in \Lambda^{n-q}_\VV(\Omega;\D')$, let $\mathfrak{I}_q[u]: \Lambda^q_{c,\VV}(\Omega) \to \C$ be defined by
  \begin{equation*}
    \langle \mathfrak{I}_q[u], [\omega] \rangle \dfn \mathfrak{B}_q\left( [\omega], [u] \right) = \langle u,\zeta \wedge \omega\rangle.
  \end{equation*}
  It is clear that $\mathfrak{I}_q[u]$ is continuous, so it defines an element of $\Lambda^q_{c,\VV}(\Omega)'$. We obtain, then, a linear map
  \begin{equation*}
    \TR{\mathfrak{I}_q}{[u]}{\Lambda^{n-q}_\VV(\Omega;\D')}
    {\mathfrak{I}_q[u]}{\Lambda^q_{c,\VV}(\Omega)'}
  \end{equation*}
  which is clearly injective. A simple verification shows that $\mathfrak{I}_q$ is continuous when both spaces are endowed with the weak dual topology. We shall construct an inverse for the map $\mathfrak{I}_q$.

  Fix $X$ a smooth, transversal vector field for $\zeta$ over $\Omega$. For $\phi \in \Lambda^q_{c,\VV}(\Omega)'$ let $\mathfrak{Q}_q^\bb(\phi) \in \D'(\Omega;\Lambda^{n-q})$ be defined by
  \begin{equation*}
    \langle \mathfrak{Q}_q^\bb(\phi), \omega \rangle \dfn \langle \phi, [X \iprod \omega] \rangle,
    \quad \omega \in F^{q+1}_c(\Omega),
  \end{equation*}
  and then take $\mathfrak{Q}_q(\phi) \dfn [\mathfrak{Q}^{\bb}_q(\phi)] \in \Lambda^{n-q}_\VV(\Omega;\D')$. This defines a linear map
  \begin{equation*}
    \TR{\mathfrak{Q}_q}{\phi}{\Lambda^q_{c,\VV}(\Omega)'}
    {\mathfrak{Q}_q(\phi)}{\Lambda^{n-q}_\VV(\Omega;\D')}
  \end{equation*}
  which is also continuous when both spaces carry the weak dual topology. Simple computations show that $\mathfrak{I}_q$ and $\mathfrak{Q}_q$ are inverses of one another.

Indeed, given $[u] \in \Lambda^{n-q}_\VV(\Omega;\D')$ we have, for $\omega \in F^{q+1}_c(\Omega)$:
    \begin{align*}
      \langle \mathfrak{Q}^{\bb}_q( \mathfrak{I}_q[u]) , \omega \rangle
      &= \langle \mathfrak{I}_q[u], [X \iprod \omega] \rangle \\
      &= \langle u, \zeta \wedge (X \iprod \omega) \rangle \\
      &= \langle u, (X \iprod \zeta) \wedge \omega - X \iprod(\zeta \wedge \omega) \rangle \\
      &= \langle u, \omega \rangle - \langle u, X \iprod(\zeta \wedge \omega) \rangle.
    \end{align*}
    However, the distribution $F^{q+1}_c(\Omega) \ni \omega \mapsto \langle u,X \iprod (\zeta \wedge \omega)\rangle$ clearly belongs to $\NN^{n-q}_{\VV}(\Omega;\D')$, so we conclude that $\mathfrak{Q}_q( \mathfrak{I}_q[u]) = [u]$. For the other direction, given $\phi \in \Lambda^q_{c,\VV}(\Omega)'$ and $[\omega]\in \Lambda^q_{c,\VV}(\Omega)$:
    \begin{align*}
      \langle \mathfrak{I}_q (\mathfrak{Q}_q (\phi)), [\omega] \rangle
      &= \langle \mathfrak{I}_q [\mathfrak{Q}_q^\bb (\phi)], [\omega] \rangle \\
      &= \langle \mathfrak{Q}_q^\bb (\phi), \zeta \wedge \omega \rangle \\
      &= \langle \phi, [X \iprod (\zeta \wedge \omega)] \rangle \\
      &= \langle \phi, [\omega - \zeta \wedge (X\iprod \omega)] \rangle \\
      &= \langle \phi, [\omega] \rangle,
    \end{align*}
    since $\zeta \wedge (X\iprod \omega) \in \NN^{q}_\VV(\Omega)$.

  We conclude that $\mathfrak{I}_q$ is a weak topological isomorphism. From the general Homomorphism Theorem~\cite[Proposition~35.7]{treves_tvs}, the map
  \begin{equation*}
    \dd'_{n-1}: \Lambda^{n-1}_{c,\VV}(\Omega) \lra \Lambda^n_{c,\VV}(\Omega)
  \end{equation*}
  has closed range if and only if its transpose
  \begin{equation*}
    \transp{\dd}'_{n-1}: \Lambda^n_{c,\VV}(\Omega)' \lra \Lambda^{n-1}_{c,\VV}(\Omega)'
  \end{equation*}
  is a homomorphism onto its range for the weak dual topology. Using that $\zeta$ is closed, it is easy to check that the diagram
  \begin{equation*}
    \begin{tikzcd}
      \Lambda^0_{\VV}(\Omega;\D')  & \Lambda^1_{\VV}(\Omega;\D') \\
      \Lambda^{n}_{c,\VV}(\Omega)'  & \Lambda^{n-1}_{c,\VV}(\Omega)'
      \arrow["{-\dd'_0}", from=1-1, to=1-2]
      \arrow["{\mathfrak{I}_0}", from=1-1, to=2-1]
      \arrow["{\mathfrak{I}_1}", from=1-2, to=2-2]
      \arrow["{\transp{\dd}'_{n-1}}", from=2-1, to=2-2]
    \end{tikzcd}
  \end{equation*}
  is commutative. Indeed, let $[u]\in \Lambda^0_{\VV}(\Omega;\D')$ and $[\omega]\in \Lambda^{n-1}_{c,\VV}(\Omega)$. Then,
    \begin{align*}
      - \langle \mathfrak{I}_1 (\dd'_0[u]), [\omega] \rangle
      &= - \langle \mathfrak{I}_1 [\dd u], [\omega] \rangle \\
      &= - \langle \dd u, \zeta \wedge \omega \rangle \\
      &= - \langle u, \dd (\zeta \wedge \omega) \rangle \\
      &= \langle u, \zeta \wedge \dd \omega \rangle \\
      &= \langle \mathfrak{I}_0[u], [\dd \omega] \rangle \\
      &= \langle \mathfrak{I}_0[u], \dd'_{n - 1} [\omega] \rangle \\
      &= \langle \transp{\dd}'_{n - 1} (\mathfrak{I}_0[u]), [\omega] \rangle.
    \end{align*}

  Since $\mathfrak{I}_0$ and $\mathfrak{I}_1$ are isomorphisms for the weak dual topology, we conclude that $\transp{\dd}'_{n-1}$ is a weak homomorphism if and only if $\dd'_0$ is a weak homomorphism, thus proving~$\eqref{it:1noprime} \Longleftrightarrow \eqref{it:2noprime}$.

  The remaining items are well-known and can be treated in the following way: if $\Omega$ is compact, then the $\Lambda^{q}_{c,\VV}(\Omega)=\Lambda^q_{\VV}(\Omega)$ are Fr{\'e}chet-Schwartz (FS) spaces, and the $\Lambda^{q}_{\VV}(\Omega;\D')$ are dual Fr{\'e}chet-Schwartz (DFS) spaces, respectively. From $\eqref{it:1noprime} \Longrightarrow \eqref{it:2noprime}$, we already know that $\ref{it:1prime} \Longrightarrow \ref{it:2prime}$, since $\dd'_0$, as a weak homomorphism, must have weakly closed range. Clearly, $\ref{it:2prime} \Longrightarrow \ref{it:3prime}$, since every weakly closed subspace must be strongly closed~\cite[Proposition~35.2]{treves_tvs}.

  For the final implication $\ref{it:3prime} \Longrightarrow \ref{it:1prime}$, we observe that, in the compact case, the maps $\mathfrak{I}_q$ are also isomorphisms for the strong dual topologies. Indeed, one just has to show strong continuity of $\mathfrak{I}_q$ and $\mathfrak{Q}_q$. Regarding $\mathfrak{I}_q$, let $\{[u_\kappa]\}_\kappa$ be a net in $\Lambda^{n-q}_{\VV}(\Omega;\D')$ that converges strongly to zero. The meaning of this is the following: for each $\kappa$ there is $v_\kappa \in \NN^q_{\VV}(\Omega;\D')$ such that the net $\{u_\kappa-v_\kappa\}_\kappa$ in $\D'(\Omega; \Lambda^{n - q}) = F^{q+1}(\Omega)'$ converges strongly to zero; the latter means that for every bounded set $B \subset F^{q+1}(\Omega)$ we have
  \begin{equation*}
    \sup_{\eta \in B} | \langle u_\kappa - v_{\kappa}, \eta \rangle | \lra 0
    \quad \text{as $\kappa \to \infty$ (in the sense of nets)}.
  \end{equation*}
  We want to show that the net $\{ \mathfrak{I}_q[u_\kappa] \}_\kappa$ converges strongly to zero in $\Lambda^{q}_{\VV}(\Omega)'$, that is: given $\pmb{B} \subset \Lambda^{q}_{\VV}(\Omega)$ bounded we have
  \begin{equation*}
    \sup_{\pmb{f} \in \pmb{B}} \left| \langle \mathfrak{I}_q[u_\kappa], \pmb{f} \rangle \right| \lra 0
    \quad \text{as $\kappa \to \infty$}.
  \end{equation*}
  Let $\pi: F^q(\Omega) \to \Lambda^{q}_{\VV}(\Omega)$ denote the standard projection. By definition, for every $\pmb{f} \in \pmb{B}$ there is a corresponding $\omega_{\pmb{f}} \in \pi^{-1}(\pmb{f})$ such that the set $B \dfn \{\omega_{\pmb{f}} \st \pmb{f} \in \pmb{B} \} \subset F^{q}(\Omega)$ is bounded. Therefore, $\zeta \wedge B = \{\zeta \wedge \omega_{\pmb{f}} \st \pmb{f} \in \pmb{B} \}$ is bounded in $F^{q+1}(\Omega)$, hence
  \begin{align*}
    \sup_{\pmb{f} \in \pmb{B}} \left| \langle \mathfrak{I}_q[u_\kappa], \pmb{f} \rangle \right|
    &= \sup_{\pmb{f} \in \pmb{B}} \left| \langle \mathfrak{I}_q[u_\kappa], [\omega_{\pmb{f}}] \rangle \right| \\
    &= \sup_{\pmb{f} \in \pmb{B}} \left| \langle u_\kappa, \zeta \wedge \omega_{\pmb{f}} \rangle \right| \\
    &=\sup_{\eta \in \zeta \wedge B} \left| \langle u_\kappa - v_\kappa, \eta \rangle \right| \lra 0
  \end{align*}
  as $\kappa \to \infty$, as desired. In a very similar way, one proves that $\mathfrak{Q}_q$ is also strongly continuous (by using the fact that $X\iprod B$ is bounded whenever $B \subset F^{q}(\Omega)$ is bounded).

Indeed, let $\{\phi_\kappa\}_\kappa$ be a net that converges strongly to zero in $\Lambda^{q}_{\VV}(\Omega)'$. Hence, for any bounded $\pmb{B} \subset \Lambda^{q}_{\VV}(\Omega)$ we have
    \begin{equation*}
      \sup_{\pmb{f} \in \pmb{B}} \left| \langle \phi_\kappa, \pmb{f} \rangle \right| \lra 0,
      \quad \kappa \to \infty.
    \end{equation*}
    If $B \sset F^{q+1}(\Omega)$ is bounded then $X \iprod B$ is bounded in $F^{q}(\Omega)$, which implies that $\pi (X \iprod B)$ is bounded in $\Lambda^{q}_{\VV}(\Omega)$. Hence:
    \begin{align*}
      \sup_{\omega \in B} \left| \langle \mathfrak{Q}^{\bb}_q(\phi_{k}), \omega \rangle \right|
      &= \sup_{\omega \in B} \left| \langle \phi_{k}, [X \iprod \omega] \rangle \right| \\
      &=\sup_{\pmb{f} \in \pi (X \iprod B)} \left| \langle \phi_{k}, \pmb{f} \rangle \right| \lra 0,
        \quad \kappa \to \infty.
    \end{align*}
    Therefore, $\mathfrak{Q}^{\bb}_q$ is strongly continuous, which allows us to deduce that $\mathfrak{Q}_q$ is strongly continuous as well (it is just its projection to a quotient space).
  
  Returning to the proof, the Homomorphism Theorem for Fr{\'e}chet-Montel implies that $\dd'_{n-1}: \Lambda^{n-1}_\VV(\Omega) \to \Lambda^{n}_{\VV}(\Omega)$ has closed range if and only if its transpose $\transp{\dd'_{n-1}}$ has \emph{strongly} closed range. Since $\mathfrak{I}_{q}$ is also an isomorphism for the strong dual topology, this is equivalent with $\dd'_0:\Lambda^0_{\VV}(\Omega;\D') \to \Lambda^1_{\VV}(\Omega;\D')$ having strongly closed range, which is condition $\ref{it:3prime}$. The proof is complete.

\end{proof}

\subsection{Applications to specific structures}

Consider now $\Omega$ compact and $\zeta \in F^1(\Omega)$ irrational, and let $g:\R \times \widehat{L} \to \Omega$ be the covering constructed in Section~\ref{sec:irr_deg1}. We recall the definition of the \emph{integration along the fibers map} (see \cite[Section~(2.15), Chapter~I]{demailly}):
\begin{equation*}
  g_*: F^{q}_c(\R \times \widehat{L}) \lra F^q(\Omega).
\end{equation*}
Given $\omega \in F^{q}_c(\R \times \widehat{L})$, let $g_* \omega \in F^q(\Omega)$ be defined by the rule\footnote{The sum is finite by compactness of $\supp (\omega)$.}
\begin{equation}
  \label{eq:int_fibers}
  (g_*\omega)|_x (v_1, \ldots, v_q) \dfn \sum_{p \in g^{-1}(x)} \omega|_p \left( v_{1,p}^\sharp, \ldots, v_{q,p}^\sharp \right),
\end{equation}
where $v_{j,p}^\sharp \in T_p (\R \times \widehat{L})$ in the unique lift of $v_j \in T_x \Omega$ by $g_*: T_p (\R \times \widehat{L}) \to T_x \Omega$, for each $p \in g^{-1}(x)$ and $j = 1, \ldots, q$. We will show that this map is surjective by exhibiting a right inverse for it.

Fix a finite covering $\{U_j\}$ of $\Omega$ by distinguished open sets: we have $g^{-1}(U_j) = \bigsqcup_{\alpha \in \mathcal{P}(\zeta)} \til{U}_{j, \alpha}$, where $g|_{\til{U}_{j, \alpha}}: \til{U}_{j, \alpha} \to U_j$ is a diffeomorphism and $T_\alpha(\til{U}_{j, \beta}) = \til{U}_{j, \alpha + \beta}$ for all $\alpha, \beta, j$. We assume, by refining the cover, that the pairwise intersections $U_j \cap U_k$ are connected for all $j,k$. Notice that (this will be important later on), with such choices, one can prove that $\{  \til{U}_{j, \alpha} \}$ is a locally finite open covering of $\R \times \widehat{L}$.

For each $j$, let $\til{V}_j\dfn \til{U}_{j, 0} \subset \R \times \widehat{L}$. Fix also a partition of unity $\{\rho_j\}$ subordinated to the covering $\{U_j\}$, and define the map
\begin{equation*}
  g^{!}: F^q(\Omega) \lra F^q_c(\R \times \widehat{L})
\end{equation*}
by
\begin{equation}
  \label{eq:pullback_cpct_supp}
  g^{!} \omega \dfn \sum_j \left( g |_{\til{V}_j} \right)^* (\rho_j \omega),
  \quad \omega \in F^q(\Omega);
\end{equation}
notice that $\left( g |_{\til{V}_j} \right)^* (\rho_j \omega)$ is compactly supported in $\til{V}_j$ for each $j$.
\begin{Lem}
  \label{lem:astexc}
  The maps $g_*:F^q_{c}(\R \times \widehat{L})\to F^q(\Omega)$ and $g^{!}:F^q(\Omega)\to F^q_c(\R \times \widehat{L})$ are both continuous and $g_* \circ g^! = \text{identity}$.
\end{Lem}
\begin{proof} Continuity of $g_*$, $g^!$ follows from their definitions~\eqref{eq:int_fibers},~\eqref{eq:pullback_cpct_supp}. Given $j$, in the context of~\eqref{eq:int_fibers}, note that
  \begin{multline*}
    \left. \left( g_* \left[ \left( g |_{\til{V}_j} \right)^* (\rho_j \omega) \right] \right) \right|_x (v_1, \ldots, v_q) \\
    = \sum_{p \in g^{-1}(x) \cap \til{V}_j} \left. \left[ \left( g |_{\til{V}_j} \right)^* (\rho_j \omega) \right] \right|_p \left( v_{1,p}^\sharp, \ldots, v_{q,p}^\sharp \right).    
  \end{multline*}
  If $x \notin U_j$ then $g^{-1}(x) \cap \til{V}_j = \emptyset$, so the form must be supported in $U_j$. But, over $U_j$, there is exactly one point in each fiber of $g$ that belongs to $\til{V}_j$. This implies that
  \begin{equation*}
    g_* \left[ \left( g |_{\til{V}_j} \right)^* (\rho_j \omega) \right] = \rho_j \omega
  \end{equation*}
  and the final claim follows from linearity.
\end{proof}

Thanks to continuity of $g_*$, we can, by duality, define the pullback of currents $g^*: \D'(\Omega; \Lambda^q) \to \D'(\R \times \widehat{L};\Lambda^q)$:
\begin{equation*}
  \langle g^*u,\omega \rangle \dfn \langle u, g_* \omega \rangle,
  \quad \omega \in F^{n+1-q}_c(\R\times \widehat{L}).
\end{equation*}
If $u \in \NN^q_{\widehat{\VV}}(\Omega; \D')$, then~\cite[Theorem~(2.14), Chapter~I]{demailly}
\begin{equation*}
  \langle g^*u, \dd t \wedge \omega \rangle
  = \langle u, g_* (g^* \zeta \wedge \omega) \rangle
  = \langle u,\zeta \wedge g_* \omega \rangle = 0.
\end{equation*}
hence $g^*$ descends to a well-defined and weakly continuous pullback map
\begin{equation*}
  g^*:\Lambda^q_{\VV}(\Omega;\D') \lra \Lambda^q_{\widehat{\VV}}(\R \times \widehat{L};\D').
\end{equation*}
By a similar argument, for each $\alpha \in \per(\zeta)$ one checks that $T_\alpha^*$ induces a map
\begin{equation*}
T_\alpha^*:\Lambda^q_{\widehat{\VV}}(\R \times \widehat{L}; \D' ) \lra \Lambda^q_{\widehat{\VV}}(\R \times \widehat{L}; \D'),
\end{equation*}
which allows us to define
\begin{equation*}
  \Lambda^q_{\widehat{\VV}}(\R \times \widehat{L}; \D' )^{\per(\zeta)} \dfn
  \left\{ \pmb{f} \in \Lambda^q_{\widehat{\VV}}(\R \times \widehat{L}; \D' ) \st T_\alpha^* \pmb{f} = \pmb{f}, \quad \forall \alpha \in \per(\zeta) \right\}
\end{equation*}
endowed with the subspace topology.

\begin{Prop}
  \label{prop:pullback_iso_dist}
  The map
  \begin{equation}
    \label{eq:gstardistributions}
    g^*:\Lambda^q_\VV(\Omega;\D') \lra \Lambda^q_{\widehat{\VV}}(\R \times \widehat{L}; \D' )^{\per(\zeta)}
  \end{equation}
  is a topological isomorphism for the weak dual topologies.
\end{Prop}
\begin{proof}
  If $[u] \in \Lambda^q_{\VV}(\Omega;\D')$, then $g^*[u] \in \Lambda^q_{\widehat{\VV}}(\R \times \widehat{L}; \D' )^{\per(\zeta)}$: for every $\alpha \in \per(\zeta)$ and $\omega \in F^{n+1-q}_c(\R \times \widehat{L})$ we have~\cite[Theorem~(2.14), Chapter~I]{demailly}
  \begin{equation*}
    \langle T_\alpha^*(g^*u), \omega \rangle
    = \langle u, g_*(T_\alpha)_* \omega \rangle
    = \langle u, (g \circ T_\alpha)_* \omega \rangle
    =\langle u, g_* \omega \rangle
    =\langle g^* u,\omega \rangle,
  \end{equation*}
  so~\eqref{eq:gstardistributions} is well-defined and continuous: we shall construct its inverse explicitly. Let $h: \D'(\R \times \widehat{L}; \Lambda^q) \to \D'(\Omega; \Lambda^q)$ be the transpose of $g^!$; that is, given for $u \in \D'(\R \times \widehat{L}; \Lambda^q)$:
  \begin{equation*}
    \langle h(u),\omega \rangle \dfn \langle u,g^{!}\omega \rangle,
    \quad \omega \in F^{n+1-q}(\Omega).
  \end{equation*}
  As such, it is weakly continuous, and moreover descends to a continuous linear map $h: \Lambda_{\widehat{\VV}}^q (\R \times \widehat{L}; \D') \to \Lambda_{\VV}^q (\Omega; \D')$.

Indeed, if $v \in \NN^q_{\widehat{\VV}}(\R \times \widehat{L};\D')$ we will show that
    \begin{equation*}
      \langle h(v), \zeta \wedge \eta \rangle = \langle v, g^! (\zeta \wedge \eta) \rangle = 0,
      \quad \forall \eta \in F^{n-q}(\Omega).
    \end{equation*}
    For that matter, first observe that for each $j$ one has
    \begin{equation*}
      \left( g |_{\til{V}_j} \right)^* (\rho_j \zeta \wedge \eta)
      = \left( g |_{\til{V}_j} \right)^* \zeta \wedge \left( g |_{\til{V}_j} \right)^* (\rho_j \eta)
      = \dd t \wedge \left( g |_{\til{V}_j} \right)^* (\rho_j \eta);
    \end{equation*}
    therefore,
    \begin{equation*}
      \langle v, g^! (\zeta \wedge \eta) \rangle
      = \sum_j \left \langle v, \left( g |_{\til{V}_j} \right)^* (\rho_j \zeta \wedge \eta) \right \rangle
      = \sum_j \left \langle v,  \dd t \wedge \left( g |_{\til{V}_j} \right)^* (\rho_j \eta) \right \rangle
      = 0
    \end{equation*}
    since $v \in \NN^q_{\widehat{\VV}}(\R \times \widehat{L}; \D')$.

  When restricted to $\per(\zeta)$-invariant currents, $h$ and $g^*$ are mutual inverses. Let us prove this. Given $u \in \D'(\Omega; \Lambda^q)$ we have, by Lemma~\ref{lem:astexc},
  \begin{equation*}
    \langle h(g^*u),\omega \rangle = \langle u, g_*(g^{!}\omega) \rangle = \langle u, \omega \rangle,
    \quad \forall \omega \in F^{n+1-q}(\Omega),
  \end{equation*}
  hence $h \circ g^*$ is the identity already on $\D'(\Omega; \Lambda^q)$. For the other composition, let $[u]\in \widehat{\Lambda}^q(\R \times \widehat{L};\D')^{\mathcal{P}(\zeta)}$: we want to prove that $[g^*h(u)] = [u]$, or in other words, that
  \begin{equation}
    \label{eq:maineq_ast}
    \langle g^*h(u) - u, \dd t \wedge \eta \rangle = 0
  \end{equation}
  for all $\eta \in F^{n-q}_c(\R \times \widehat{L})$. Since the relation~\eqref{eq:maineq_ast} is linear, it suffices to prove it for $\eta$ supported in some $\til{U}_{j, \alpha} \subset \R \times \widehat{L}$, since such open sets cover $\R \times \widehat{L}$. We obtain
  \begin{align*}
    \langle g^*h(u) - u, \dd t \wedge \eta \rangle
    &= \langle u, g^{!} \left( g_*(\dd t \wedge \eta) \right) - \dd t \wedge \eta \rangle \\
    &= \langle u, g^{!} \left( g_* (g^*\zeta \wedge \eta) \right) - \dd t \wedge \eta \rangle \\
    &= \langle u, g^{!} \left( \zeta \wedge g_* \eta \right) - \dd t \wedge \eta \rangle.
  \end{align*}
  Since $\supp (\eta) \subset \til{U}_{j, \alpha}$, it follows that $g_* \eta = \left( g |_{\til{U}_{j, \alpha}} \right)_* \eta$ is supported in $U_j$, hence~\cite[Theorem~(2.14), Chapter~I]{demailly}
  \begin{align*}
    g^{!} \left( \zeta \wedge g_* \eta \right)
    &= \sum_k \left( g |_{\til{V}_k} \right)^* ( \rho_k \zeta \wedge g_* \eta ) \\
    &= \sum_{U_k \cap U_j \neq \emptyset} \dd t \wedge \left( g |_{\til{V}_k} \right)^* (\rho_k g_* \eta) \\
    &= \sum_{U_k \cap U_j \neq \emptyset} \dd t \wedge \left( g |_{\til{V}_k} \right)^* \left( \rho_k \left( g |_{\til{U}_{j, \alpha}} \right)_* \eta \right) \\
    &= \sum_{U_k \cap U_j \neq \emptyset} \dd t \wedge \left( g |_{\til{V}_k} \right)^* \left( g |_{\til{U}_{j, \alpha}} \right)_* \left( \left( \rho_k \circ  g |_{\til{U}_{j, \alpha}} \right) \eta \right).
  \end{align*}
  For each $k$ such that $U_k \cap U_j \neq \emptyset$ (recall this intersection is connected) there is exactly one $\beta_k \in \mathcal{P}(\zeta)$ such that $\til{U}_{j,\alpha} \cap g^{-1}(U_k) = \til{U}_{j,\alpha} \cap \til{U}_{k, \beta_k}$. Since, moreover, $\supp \left( \left( \rho_k \circ  g |_{\til{U}_{j, \alpha}} \right) \eta \right) \subset \til{U}_{j,\alpha} \cap \til{U}_{k, \beta_k}$, we obtain
  \begin{equation*}
    \left( g |_{\til{U}_{j, \alpha}} \right)_* \left( \left( \rho_k \circ  g |_{\til{U}_{j, \alpha}} \right) \eta \right)
    =
    \left( g |_{\til{U}_{k, \beta_k}} \right)_* \left( \left( \rho_k \circ  g |_{\til{U}_{j, \alpha}} \right) \eta \right).
  \end{equation*}
  We conclude that 
  \begin{align*}
    g^{!} \left( \zeta \wedge g_* \eta \right)
    &= \sum_{U_k \cap U_j \neq \emptyset} \dd t \wedge \left( g |_{\til{V}_k} \right)^* \left( g |_{\til{U}_{j, \alpha}} \right)_* \left( \left( \rho_k \circ  g |_{\til{U}_{j, \alpha}} \right) \eta \right) \\
    &= \sum_{U_k \cap U_j \neq \emptyset} \dd t \wedge \left( g |_{\til{U}_{k,0}} \right)^* \left( g |_{\til{U}_{k, \beta_k}} \right)_* \left( \left( \rho_k \circ  g |_{\til{U}_{j, \alpha}} \right) \eta \right),
  \end{align*}
  and since $T_{-\beta_k}(\til{U}_{k, \beta_k}) = \til{U}_{k, 0}$ and $g = g \circ T_{-\beta_k}$, one has $g|_{\til{U}_{k, \beta_k}} = g|_{\til{U}_{k, 0}} \circ T_{-\beta_k}$, which further implies that
  \begin{align*}
    g^{!} \left( \zeta \wedge g_* \eta \right)
    &= \sum_{U_k \cap U_j \neq \emptyset} \dd t \wedge \left( T_{-\beta_k} \right)_* \left( \left( \rho_k \circ  g |_{\til{U}_{j, \alpha}} \right) \eta \right) \\
    &= \sum_{U_k \cap U_j \neq \emptyset} \dd t \wedge T_{\beta_k}^* \left( \left( \rho_k \circ  g |_{\til{U}_{j, \alpha}} \right) \eta \right).
  \end{align*}
  The identity above tells us that $g^{!}(\zeta \wedge g_* \eta) - \dd t \wedge \eta$ equals
  \begin{equation}
    \label{eq:finalform}
    \sum_{U_k \cap U_j \neq \emptyset} \dd t \wedge \left\{ T_{\beta_k}^* \left( \left( \rho_k \circ  g |_{\til{U}_{j, \alpha}} \right) \eta \right) - \left( \rho_k \circ  g |_{\til{U}_{j, \alpha}} \right) \eta \right\}.
  \end{equation}
  However, since $[u]$ is $\per(\zeta)$-invariant, we have
  \begin{equation*}
    \langle u, \dd t \wedge (T_\beta^* \omega - \omega) \rangle = 0,
    \quad \forall \beta \in \per(\zeta), \ \omega \in F^{n-q}_c(\R \times \widehat{L}),
  \end{equation*}
  hence evaluating $u$ on~\eqref{eq:finalform} yields the result.
\end{proof}

Next, we touch upon a mild form of the distributional version of Proposition~\ref{prop:partial_derham_solvable}. We keep the notation and context (and refer the reader to that statement), yet for simplicity we deal only with the case $M = \R$ and care specifically with the first degree of that complex. Explicitly:
\begin{Prop}
  \label{prop:partialdr_dist}
  The partial exterior derivative
  \begin{equation*}
    \dd_{N}: \Lambda_{\WW}^0 (\R \times N; \D') \lra \Lambda_{\WW}^1 (\R \times N; \D')
  \end{equation*}
  has weakly closed range and, moreover, is a homomorphism for the weak dual topologies.
\end{Prop}
This result will be a consequence of the following lemma:
\begin{Lem}
  \label{lem:partialdr_char}
  Let $N$ be an $n$-dimensional connected, smooth, oriented and non-compact manifold. Then an $\omega \in \Lambda^n_{c, \WW} (\R \times N)$ is of the form $\dd_{N} \eta$ for some $\eta \in \Lambda^{n - 1}_{c, \WW} (\R \times N)$ if and only if
  \begin{equation}
    \label{eq:topdegree_zeroint}
    \int_N \omega(t,\cdot) = 0,
    \quad \forall t \in \R.
  \end{equation}
\end{Lem}
\begin{Rem}
  This is simply a version with parameters of the classical result concerning compactly supported differential forms of top degree (see, for example, \cite[Theorem~8.6.4]{conlon_manifolds}); their proofs go hand in hand.
\end{Rem}

  \begin{proof}
    The direct implication is an immediate consequence of Stokes Theorem, so we just have to prove the converse. It is clear that~\eqref{eq:topdegree_zeroint} is preserved if we apply a diffeomorphism to $N$.
    
    When $N = \R^n$, this is Poincar{\'e} Lemma with parameters~\cite[Section~VII.3]{bch_iis}: given $\omega \in \Lambda^n_{c, \WW} (\R \times \R^n)$, represented as $\omega = f(t,x) \dd x_1 \wedge \cdots \wedge \dd x_n$  with $f \in \cinfty_c(\R \times \R^n)$, choose $(t_0,x_0) \in \R \times \R^{n} \setminus \supp (f)$ and define $\mathrm{G}(f) \in \Lambda^{n - 1}_{\WW} (\R \times \R^n)$ by~\cite[eqn.~(VII.10)]{bch_iis}. It follows that $\dd_{\R^n}\mathrm{G}(f) = \omega$, and~\eqref{eq:topdegree_zeroint} implies that $\mathrm{G}(f)$ has compact support.
    
    Now, for the case of general $N$, choose $\{U_j\}_{j\in \N}$ a countable cover of $N$ by open subsets that are diffeomorphic to $\R^n$ and let $N_k \dfn U_1 \cup \cdots \cup U_k$ for each $k \geq 1$. Since $N$ is connected, we can assume that $N_k \cap U_{k+1} \neq \emptyset$ for every $k$. If $\omega \in \Lambda^n_{c, \WW} (\R \times N)$ then $\supp(\omega) \sset \R \times N_k$ for some $k$: it suffices to show that, if~\eqref{eq:topdegree_zeroint} holds, then there exists $\eta \in \Lambda^{n - 1}_{c, \WW} (\R \times N_k)$ such that $\dd_N \eta = \omega$. We prove this claim by induction in $k$.
    
    For $k = 1$, since $N_1 = U_1$ is diffeomorphic to $\R^n$, the claim reduces to the first paragraph. Assume the claim is true for some $k \geq 1$, and suppose $\omega \in \Lambda^n_{c, \WW} (\R \times N_{k + 1})$ satisfies~\eqref{eq:topdegree_zeroint}, where $N_{k+1} = N_k \cup U_{k+1}$. Take $\theta \in F^n_{c}(N_{k+1})$ supported in $N_k \cap U_{k+1}$ and satisfying $\int_{N_{k+1}} \theta = 1$, let $\{\phi,\psi\}$ be a smooth partition of unity for $N_{k+1}$ subordinated to the cover $\{N_k, U_{k+1}\}$, and let $c = c(t) \dfn \int_{N_{k+1}} \phi \omega(t,\cdot) \in \cinfty_c(\R)$. Hence $\phi \omega - c \theta \in \Lambda^n_{c, \WW} (\R \times N_{k})$ and satisfies~\eqref{eq:topdegree_zeroint}, so by the induction hypothesis there exists $\eta \in \Lambda^{n - 1}_{c, \WW} (\R \times N_{k})$ such that $\dd_{N} \eta = \phi \omega - c \theta$. Similarly, $\psi \omega + c\theta$ is supported in $\R \times U_{k+1}$ and
    \begin{align*}
      \int_{U_{k+1}} \psi \omega(t, \cdot) + c(t) \theta
      &= \int_{N_{k+1}} (1 - \phi) \omega(t, \cdot) + c(t) \\
      &= \int_{N_{k+1}} \omega(t,\cdot) - \int_{N_{k+1}}\phi \omega(t, \cdot) + c(t) \\
      &=0
    \end{align*}
    for every $t \in \R$: since $U_{k + 1}$ is diffeomorphic to $\R^n$, there is $\alpha \in \Lambda^{n - 1}_{c, \WW} (\R \times U_{k + 1})$ such that $\dd_N \alpha = \psi \omega + c\theta$. Clearly $\eta + \alpha$ is supported in $\R \times N_{k + 1}$ and solves $\dd_N (\eta + \alpha) = \omega$, as desired.
  \end{proof}

\begin{proof}[Proof of Proposition~\ref{prop:partialdr_dist}]
  By Proposition~\ref{prop:duality_basic}, we must show that the range of $\dd_{N}: \Lambda^{n - 1}_{c, \WW} (\R \times N) \to \Lambda^{n}_{c, \WW} (\R \times N)$ is closed, but this follows immediately from Lemma~\ref{lem:partialdr_char}.
\end{proof}

The final result we require is the following:
\begin{Prop}
  \label{prop:delta_dist}
  The map 
  \begin{equation*}
    \TR{\delta}{\pmb{u}}{\Lambda^0_{\widehat{\VV}}(\R \times \widehat{L}; \D')}
    {\left( T^*_\alpha \pmb{u} - \pmb{u} \right)_{\alpha \in \per(\zeta)}}{\prod_{\alpha \in \per(\zeta)} \Lambda^0_{\widehat{\VV}}(\R \times \widehat{L}; \D')}
  \end{equation*}
  has weakly closed range and, moreover, is a homomorphism for the weak dual topologies (and the product topology).
\end{Prop}
\begin{proof}
  Consider the subspace
  \begin{equation*}
    E \dfn \left \{ (\pmb{u}_\alpha) \in \prod_{\alpha \in \per(\zeta)} \Lambda^0_{\widehat{\VV}}(\R \times \widehat{L}; \D') \st
      \pmb{u}_{\alpha + \beta} = T^*_\alpha \pmb{u}_\beta + \pmb{u}_\alpha, \quad \alpha, \beta \in \per(\zeta) \right\},
  \end{equation*}
  which is weakly closed in that product. Note that, if $\pmb{u} \in \Lambda^0_{\widehat{\VV}}(\R \times \widehat{L}; \D')$, then setting $\pmb{u}_\alpha \dfn T^*_\alpha \pmb{u} - \pmb{u}$ we have
  \begin{equation*}
    T^*_\alpha \pmb{u}_\beta + \pmb{u}_\alpha = T^*_{\alpha + \beta} \pmb{u} - T^*_\alpha \pmb{u} + T^*_\alpha \pmb{u} - \pmb{u} = T^*_{\alpha + \beta} \pmb{u} - \pmb{u} = \pmb{u}_{\alpha+\beta},
  \end{equation*}
  so $\ran (\delta) \subset E$. We will construct a weakly continuous linear map $S: E \to \Lambda^0_{\widehat{\VV}}(\R \times \widehat{L}; \D')$ such that $\delta \circ S = \Id_E$. This implies that $\delta$ is a homomorphism onto its range, which equals $E$.

  We work with the covering $\{U_j\}$ of $\Omega$ fixed on the occasion of the definition of $g^!$, and all the related choices and notation. Let $\widehat{\rho}_j \dfn \rho_j \circ g \in \cinfty_c(g^{-1}(U_j))$ and $\widehat{\rho}_{j, \alpha} \dfn \rho_j \circ g|_{\til{U}_{j,\alpha}} \in \cinfty_c(\til{U}_{j,\alpha})$, which we can extend by zero and regard as an element of $\cinfty_c(\R \times \widehat{L})$. Let $\msf{u} \dfn (\pmb{u}_\alpha) \in E$ and define
  \begin{equation*}
    S(\msf{u}) \dfn -\sum_j \sum_{\alpha \in \per(\zeta)} \pmb{u}^j_{-\alpha,\alpha} \in \Lambda^0_{\widehat{\VV}}(\R \times \widehat{L}; \D')
  \end{equation*}
  where
  \begin{equation*}
    \pmb{u}^j_{\alpha,\beta} \dfn \widehat{\rho}_{j, \beta} \ \pmb{u}_\alpha \in \Lambda^0_{c, \widehat{\VV}}( \til{U}_{j,\beta}; \D'),
    \quad \beta\in \per(\zeta).
  \end{equation*}
  Since $\{ \til{U}_{j,\alpha} \}$ is locally finite, $S(\msf{u})$ well-defined. It is easily seen that $S$ is weakly continuous: let $\msf{u}^{\kappa} = (u^{\kappa}_\alpha)$ be a net in $E$ which converges to $\msf{u} = (\pmb{u}_\alpha) \in E$, which is the same as saying $\pmb{u}^{\kappa}_\alpha \to \pmb{u}_\alpha$ weakly for every $\alpha$. Given a test form $\pmb{\phi} \in \Lambda_{c, \widehat{\VV}}^n (\R \times \widehat{L})$, it is clear that $\langle S(\msf{u}^\kappa), \pmb{\phi} \rangle \to \langle S(\msf{u}), \pmb{\phi} \rangle$ by considering only those $\alpha$ such that $\supp(\pmb{\phi})$ intersects $\til{U}_{j,\alpha}$.

  An easy calculation shows that $\delta(S(\msf{u})) = \msf{u}$ for all $\msf{u} \in E$. Indeed, given $\msf{u} = (\pmb{u}_\alpha) \in E$ we must to compute
  \begin{equation*}
    T^*_\alpha S(\msf{u}) - S(\msf{u}) = \sum_{j} \sum_{\beta \in \per(\zeta)} \pmb{u}^{j}_{-\beta,\beta} - T^*_\alpha \pmb{u}^{j}_{-\beta,\beta}.
  \end{equation*}
  for each $\alpha \in \per(\zeta)$. Given $\beta \in \per(\zeta)$ we have
  \begin{equation*}
    T^*_\alpha \pmb{u}^{j}_{-\beta,\beta} = (T^*_\alpha \widehat{\rho}_{j, \beta} ) \ T^*_\alpha \pmb{u}_{-\beta},
  \end{equation*}
  and since $T^*_\alpha \widehat{\rho}_{j,\beta} =  (\rho_j \circ g|_{\til{U}_{j,\beta}}) \circ T^\alpha = \rho_j \circ g|_{\til{U}_{j,\beta - \alpha}} = \widehat{\rho}_{j, \beta - \alpha}$, we conclude that 
  \begin{align*}
    T^*_\alpha S(\msf{u}) - S(\msf{u})
    &= \sum_{j} \sum_{\beta \in \per(\zeta)} \widehat{\rho}_{j, \beta} \ \pmb{u}_{-\beta} - \widehat{\rho}_{j, \beta - \alpha} \ T^*_\alpha \pmb{u}_{-\beta} \\
    &= \sum_{j}\sum_{\beta \in \per(\zeta)} \widehat{\rho}_{j, \beta} \left( \pmb{u}_{-\beta} - T^*_\alpha \pmb{u}_{-(\alpha+\beta)} \right)
  \end{align*}
  by relabeling the terms in the second sum. But $\msf{u} \in E$ means that $T^*_\alpha \pmb{u}_{-(\alpha+\beta)} = \pmb{u}_{-\beta}-\pmb{u}_\alpha$, hence
  \begin{equation*}
    T^*_\alpha S(\msf{u}) - S(\msf{u}) = \left( \sum_{j}\sum_{\beta \in \per(\zeta)} \widehat{\rho}_{j, \beta} \right) \pmb{u}_\alpha = \pmb{u}_\alpha
  \end{equation*}
  whatever $\alpha \in \per(\zeta)$, which completes the proof. 
\end{proof}

\subsection{Proof of Theorem~\ref{thm:last_step_solv}}

We recall first that, thanks to Proposition~\ref{prop:duality_basic}, $\zeta$ is globally solvable in degree $n$ if and only if $\dd'_0: \Lambda^0_\VV(\Omega;\D') \to \Lambda^1_\VV(\Omega;\D')$ has weakly closed range; which, in turn, is equivalent (Proposition~\ref{prop:pullback_iso_dist}) to
\begin{equation}
  \label{eq:first-partial-d-distributions}
  \dd_{\widehat{L}}: \Lambda^0_{\widehat{\VV}} (\R \times \widehat{L};\D')^{\per(\zeta)} \lra \Lambda^1_{\widehat{\VV}} (\R \times \widehat{L};\D')^{\per(\zeta)}
\end{equation}
having weakly closed range. Let
\begin{equation*}
  W \dfn \left\{ \pmb{u} \in \Lambda^0_{\widehat{\VV}} (\R \times \widehat{L};\D') \st \dd_{\widehat{L}} \pmb{u} = 0 \right\}.
\end{equation*}
\begin{Prop}
  \label{prop:dist_equiv1}
  The map~\eqref{eq:first-partial-d-distributions} has weakly closed range if and only if so does the map
  \begin{equation}
    \label{eq:first-delta-distributions}
    \delta: W \lra \prod_{\alpha \in \per(\zeta)} W.
  \end{equation}
\end{Prop}
\begin{proof}
  Suppose that~\eqref{eq:first-delta-distributions} has weakly closed range and let $\{\pmb{u}_\kappa\}$ be a net in $\Lambda^0_{\widehat{\VV}} (\R \times \widehat{L};\D')^{\per(\zeta)}$ such that $\dd_{\widehat{L}} \pmb{u}_\kappa \to \pmb{f} \in \Lambda^1_{\widehat{\VV}} (\R \times \widehat{L};\D')^{\per(\zeta)}$ weakly. By Proposition~\ref{prop:partialdr_dist}, we can find $\pmb{u} \in \Lambda^0_{\widehat{\VV}} (\R \times \widehat{L};\D')$ such that $\pmb{f} = \dd_{\widehat{L}} \pmb{u}$: our goal is to find $\pmb{u}^{\bb} \in \Lambda^0_{\widehat{\VV}} (\R \times \widehat{L};\D')$ such that $\delta \pmb{u}^\bb = 0$ and $\dd_{\widehat{L}} \pmb{u}^\bb = \dd_{\widehat{L}} \pmb{u}$.

  Since $\dd_{\widehat{L}}(\pmb{u} - \pmb{u}_\kappa) \to 0$ and $\dd_{\widehat{L}}: \Lambda^0_{\widehat{\VV}} (\R \times \widehat{L};\D') \to \Lambda^1_{\widehat{\VV}} (\R \times \widehat{L};\D')$ is a weak homomorphism (Proposition~\ref{prop:partialdr_dist}), by Remark~\ref{rem:seq_invert} we can find a net $\{\pmb{v}_\kappa\}$ converging to zero in  $\Lambda^0_{\widehat{\VV}} (\R \times \widehat{L};\D')$ such that $\pmb{w}_\kappa \dfn \pmb{u} - \pmb{u}_\kappa - \pmb{v}_\kappa$ satisfies $\dd_{\widehat{L}} \pmb{w}_\kappa = 0$. Therefore, $\{\pmb{w}_\kappa\}$ is a net in $W$ such that $\delta \pmb{w}_\kappa = \delta \pmb{u} - \delta \pmb{v}_\kappa \to \delta \pmb{u}$. By hypothesis, there is $\pmb{w} \in W$ such that $\delta (\pmb{u} - \pmb{w}) = 0$: setting $\pmb{u}^\bb \dfn \pmb{u} - \pmb{w}$ yields the result.

  Conversely, suppose now that~\eqref{eq:first-partial-d-distributions} has weakly closed range and let $\{ \pmb{w}_\kappa\}$ be a net in $W$ such that $\delta \pmb{w}_\kappa \to \msf{w}$ weakly. By Proposition~\ref{prop:delta_dist}, there exists $\pmb{w} \in \Lambda^0_{\widehat{\VV}} (\R \times \widehat{L};\D')$ such that $\delta \pmb{w} = \msf{w}$, so in particular $\delta(\pmb{w} - \pmb{w}_\kappa) \to 0$. Applying once again Proposition~\ref{prop:delta_dist} and Remark~\ref{rem:seq_invert}, we are able to find a net $\{ \pmb{v}_\kappa\}$ converging to zero in $\Lambda^0_{\widehat{\VV}} (\R \times \widehat{L};\D')$ and such that $\delta(\pmb{w} - \pmb{w}_\kappa - \pmb{v}_\kappa ) = 0$, that is, $\pmb{u}_\kappa \dfn \pmb{w} - \pmb{w}_\kappa - \pmb{v}_\kappa \in \Lambda^0_{\widehat{\VV}} (\R \times \widehat{L};\D')^{\per(\zeta)}$. Clearly $\dd_{\widehat{L}} \pmb{u}_\kappa = \dd_{\widehat{L}} \pmb{w} - \dd_{\widehat{L}} \pmb{v}_\kappa \to \dd_{\widehat{L}} \pmb{w}$, hence, by hypothesis, there exists $\pmb{u} \in \Lambda^0_{\widehat{\VV}} (\R \times \widehat{L};\D')^{\per(\zeta)}$ such that $\dd_{\widehat{L}} \pmb{u} = \dd_{\widehat{L}} \pmb{w}$. Thus $\pmb{w}^\bb \dfn \pmb{w} - \pmb{u} \in W$ solves $\delta \pmb{w}^\bb = \delta \pmb{w} - \delta \pmb{u} = \msf{w}$.
\end{proof}

As a consequence of Proposition~\ref{prop:dist_equiv1} and previous observations, we have thus far determined that
\begin{equation*}
  \text{$\zeta$ is globally solvable in degree $n$}
  \Longleftrightarrow
  \text{\eqref{eq:first-delta-distributions} has weakly closed range}.
\end{equation*}

Next we identify $W$ as the space of distributions on $\R \times \widehat{L}$ which are independent of the $\widehat{L}$-variables. That is, we are going to explicitly construct a weak topological isomorphism $S: W \to \D'(\R)$. First, write $W$ as
\begin{equation}
  \label{eq:dist_partial_closed}
  W = \left\{ u \in F^{n+1}_c(\R \times \widehat{L})' \st
    \langle u,\dd t \wedge \dd \phi \rangle = 0, \ \forall \phi \in F^{n-1}_c(\R \times \widehat{L}) \right\}.
\end{equation}
Now, fix a non-singular $\mu \in F^n(\widehat{L})$ (orientability of $\widehat{L} \simeq L$ is granted by that of $\Omega$ and the fact that $L$ admits a transversal vector field). Then, there is a topological isomorphism $T:F^{n+1}_c(\R \times \widehat{L}) \to \cinfty_c(\R \times \widehat{L})$ such that
\begin{equation*}
  \omega = T(\omega) \dd t \wedge \mu,
  \quad \omega \in F^{n+1}_c(\R \times \widehat{L}),
\end{equation*}
and for such $\omega$ define the function in $\cinfty_c(\R)$
\begin{equation*}
  t \longmapsto \int_{\widehat{L}}T(\omega)(t,x) \mu(x)
\end{equation*}
which we denote simply by $\int_{\widehat{L}}T(\omega) \mu$. Choosing an $\omega_0 \in F^n_c(\widehat{L})$ such that $\int_{\widehat{L}}\omega_0 = 1$, Lemma~\ref{lem:partialdr_char} ensures that for each $\omega \in F^{n+1}_c(\R \times \widehat{L})$ we can find $\alpha \in \Lambda^{n-1}_{\widehat{\VV}, c} (\R \times \widehat{L})$ such that
\begin{equation*}
  \dd_{\widehat{L}} \alpha = T(\omega) \mu  - \left( \int_{\widehat{L}} T(\omega) \mu \right) \omega_0.
\end{equation*}
In particular, we obtain
\begin{align*}
  \dd t \wedge \dd \alpha
  &= \dd t \wedge T(\omega)\mu - \dd t \wedge  \left( \int_{\widehat{L}} T(\omega) \mu \right) \omega_0 \\
  &= \omega - \left( \int_{\widehat{L}} T(\omega) \mu \right) \dd t \wedge \omega_0,
\end{align*}
where $\alpha$ can now be seen as a compactly supported $(n-1)$-form in $\R \times \widehat{L}$. From~\eqref{eq:dist_partial_closed}, we have, for every $u \in W$ and $\omega \in F^{n+1}_c(\R \times \widehat{L})$:
\begin{equation}
  \label{identity for u in W}
  \langle u, \omega \rangle = \left \langle u, \left( \int_{\widehat{L}} T(\omega) \mu \right) \dd t \wedge \omega_0 \right\rangle.
\end{equation}
This gives rise to a continuous mapping $S: W \to \D'(\R)$ defined by
\begin{equation*}
  \langle S(u), \phi \rangle \dfn \langle u, \phi \ \dd t \wedge \omega_0 \rangle,
  \quad \phi \in \cinfty_c(\R)
\end{equation*}
(i.e.,~the transpose of the multiplication map $\phi \mapsto \phi \ \dd t \wedge \omega_0$), whose inverse $R:\D'(\R)\to W$ we can explicitly construct:
\begin{equation*}
  \langle R(v), \omega \rangle \dfn \left \langle v,  \int_{\widehat{L}}T(\omega) \mu \right \rangle,
  \quad \omega \in F^{n+1}_c(\R \times \widehat{L}),
\end{equation*}
which is also continuous.

Indeed, notice that for $\phi \in F^{n-1}_c(\R \times \widehat{L})$ we have $\dd t \wedge \dd \phi = T(\dd t \wedge \dd \phi) \dd t \wedge \mu$, which implies $\dd_{\widehat{L}} \phi = T(\dd t \wedge \dd \phi) \mu$, hence $\int_{\widehat{L}}T(\dd t \wedge \dd \phi) \mu = 0$ by Stokes Theorem. It follows that $\langle R(v), \dd t \wedge \dd \phi \rangle = 0$ for all such $\phi$, that is, $R(v) \in W$. Moreover, $R$ is clearly continuous, and for any $u \in W$ and $\omega \in F^{n+1}_c(\R \times \widehat{L})$ we have, by~\eqref{identity for u in W}, that
  \begin{equation*}
    \langle R \circ S (u), \omega \rangle
    = \left \langle S(u),  \int_{\widehat{L}} T(\omega) \mu \right \rangle 
    = \left \langle u, \left( \int_{\widehat{L}} T(\omega) \mu \right) \dd t \wedge \omega_0 \right \rangle 
    = \langle u, \omega \rangle,
  \end{equation*}
  which implies that $R \circ S$ is the identity. On the other hand, notice that for any given $\phi \in C_{c}^{\infty}(\R)$ we have
  \begin{align*}
    \phi \ \dd t \wedge \omega_0 = T(\phi \ \dd t \wedge \omega_0) \dd t \wedge \mu
    &\Longrightarrow
      \phi \omega_0 = T(\phi \ \dd t \wedge \omega_0) \mu \\
    &\Longrightarrow
      \int_{\widehat{L}} T(\phi \ \dd t \wedge \omega_0) \mu =  \int_{\widehat{L}} \phi \omega_0 = \phi,
  \end{align*}
  hence for any $v \in \D'(\R)$ we get
  \begin{equation*}
    \langle S \circ R (v), \phi \rangle
    = \langle R(v), \phi \ \dd t \wedge \omega_0 \rangle 
    = \left \langle v, \int_{\widehat{L}} T(\phi \ \dd t \wedge \omega_0) \mu \right \rangle 
    = \langle v, \phi \rangle, 
  \end{equation*}
  which implies that $S \circ R$ is again the identity.
  
This shows that $S$ is an isomorphism for the weak dual topologies. It is, moreover, compatible with translations, in the sense that $T_\alpha^*S(u) = S(T^*_\alpha u)$ for every $u \in W$ and $\alpha \in \per(\zeta)$. For that matter, notice first that for each $\phi \in \cinfty_c(\R)$ we have
\begin{equation*}
  T_\alpha^* (\phi \ \dd t \wedge \omega_0) = T_\alpha^* (\phi \dd t) \wedge T_\alpha^* \omega_0 = \phi (\cdot + \alpha) \dd t \wedge T_\alpha^* \omega_0.
\end{equation*}
Also, note that $T^*_\alpha \omega_0 - \omega_0$ is exact on $\widehat{L}$ by de Rham's Theorem (its integral is zero); actually, $\phi(\cdot + \alpha) (T^*_\alpha \omega_0 - \omega_0)$ is $\dd_{\widehat{L}}$-exact, hence
\begin{equation*}
  \left \langle u, \phi (\cdot + \alpha) \dd t \wedge T^*_\alpha \omega_0 \right \rangle = \langle u, \phi(\cdot + \alpha) \dd t \wedge \omega_0 \rangle
\end{equation*}
thanks to~\eqref{eq:dist_partial_closed}, and the latter can thus be rewritten as
\begin{align*}
  \left \langle T_\alpha^* S(u), \phi \right \rangle
  &= \left \langle S(u), \phi(\cdot + \alpha) \right \rangle \\
  &= \langle u, \phi(\cdot + \alpha) \dd t \wedge \omega_0 \rangle \\
  &= \left \langle u, \phi (\cdot + \alpha) \dd t \wedge T^*_\alpha \omega_0 \right \rangle \\
  &= \left \langle u, T_\alpha^* (\phi \ \dd t \wedge \omega_0) \right \rangle \\
  &= \left \langle T_\alpha^* u, \phi \ \dd t \wedge \omega_0 \right \rangle \\
  &= \left \langle S(T_\alpha^* u), \phi \right \rangle.
\end{align*}

Therefore, to show that~\eqref{eq:first-delta-distributions} has a weakly closed range if the same is true for the map
\begin{equation}
  \label{eq:delta_map}
  \TR{\delta}{u}{\D'(\R)}
  {\left( T^*_\alpha u - u \right)_{\alpha \in \per(\zeta)}}{\prod_{\alpha \in \per(\zeta)}\D'(\R)},
\end{equation}
which, by the same argument used in the $\cinfty$ case (Corollary~\ref{cor:reduction_to_generators}), is equivalent to
\begin{equation}
  \label{eq:Donditros}
  \TR{D}{u}{\D'(\R)}{\left( T_{a_j}^* u - u \right)_{j = 0, \ldots, r}}{\prod_{j = 0}^r \D'(\R)}
\end{equation}
having weakly closed range. The proof will thus be complete provided we show that:
\begin{Thm}
  The following are equivalent:
  \begin{enumerate}
  \item The map~\eqref{eq:Donditros} has weakly closed range.
  \item $(a_1, \ldots, a_r)$ is a non-Liouville vector.
  \end{enumerate}
\end{Thm}

\begin{proof}
  Let us prove the necessity first. Assuming $(a_1, \ldots, a_r)$ is a Liouville vector, let $\{ \xi_\nu \}_{\nu \in \N}$ be the sequence of integers furnished in the proof of Theorem~\ref{thm:first_step_solv_redux}. Define the sequence of smooth functions
  \begin{equation*}
    f_N(x) \dfn \sum_{\nu = 1}^N | \xi_{\nu}|^{\frac{\nu}{2}} e^{2\pi i x \xi_{\nu}},
    \quad x \in \R.
  \end{equation*}
  For each $j = 0, 1, \ldots, r$, it is not difficult to check that $T^*_{a_j} f_N - f_N$ converges in $\cinfty(\R)$ to
  \begin{equation*}
    g_j(x) \dfn \sum_{\nu = 1}^\infty \left( e^{2\pi i a_j \xi_\nu} - 1 \right) | \xi_{\nu}|^{\frac{\nu}{2}} e^{2\pi i x \xi_{\nu}},
    \quad x \in \R,
  \end{equation*}
  as $N \to \infty$; keep in mind that
  \begin{equation*}
    |\xi_\nu| < |\xi_{\nu + 1}|
    \quad \text{and} \quad
    \max_{1 \leq j \leq r} \left| e^{2\pi i a_j \xi_\nu} - 1 \right| < \frac{1}{|\xi_\nu|^\nu},
    \quad \forall \nu \in \N.
  \end{equation*} 
  Yet, if there were an $f \in \D'(\R)$ such that $T^*_{a_j}f - f = g_j$ for all $j$, it would have to be $1$-periodic since $a_0 = 1$ and $g_0 = 0$. On the other hand,  the $\nu$-th Fourier coefficient of $f$ would have to be $|\xi_\nu|^{\frac{\nu}{2}}$, and these do not form a tempered sequence, proving that~\eqref{eq:Donditros} does not have weakly closed range.  

  Now we verify the converse. The transpose of~\eqref{eq:Donditros} is clearly
  \begin{equation*}
    \TR{\transp{D}}{(u_0,\ldots,u_r)}{\bigoplus_{j = 0}^r \cinfty_c(\R)}{\displaystyle \sum_{j=0}^r u_{j} * ( \delta_{a_j} - \delta_0)}{\cinfty_c(\R)},
  \end{equation*}
  which we will prove to have closed range provided $(a_1, \ldots, a_r)$ is a non-Liouville vector: the desired conclusion will follow once more from the general Homomorphism Theorem~\cite[Proposition~35.7]{treves_tvs}.
  
  From the same calculations done previously in the smooth case we have $\ran (\transp{D}) \sset \{v \in \cinfty_c(\R) \st \hat{v}(0) = 0\}$. In order to prove the reverse inclusion -- which finishes the proof of the theorem --, let $v \in C_{c}^{\infty}(\R)$ be such that $\hat{v}(0) = 0$. By the Paley-Wiener-Schwartz Theorem, $\hat{v}\in \Ol(\C)$ is an entire function of exponential type which has \emph{fast decay} on $\R$, which we factor as $\hat{v}(z) = z h(z)$ with $h \in \Ol(\C)$ of exponential type and fast decay on $\R$: in order to solve 
  \begin{equation*}
    \sum_{j=0}^r u_{j} * ( \delta_{a_j} - \delta_0) = v,
    \quad u_0, \ldots, u_r \in \cinfty_c(\R),
  \end{equation*}
  we now must solve~\eqref{eq:ideal_eq} with $f_0, \ldots, f_r \in \Ol(\C)$ of exponential type and \emph{fast decay} on $\R$.
  
  By associating Theorems~\ref{thm:ehrenpreis} and~\ref{thm:main_result_suff}, there are $F_{0}, \ldots, F_{r} \in \Ol(\C)$ of exponential type and bounded by a polynomial on $\R$ such that~\eqref{eq:ideal2} holds. Define then $f_j(z) \dfn F_j(z)h(-z)$, $j = 0, \ldots, r$: since the $F_j$ are of exponential type and bounded by a polynomial on $\R$, whereas $h(-z)$ is of exponential type and has fast decay on $\R$, it follows that $f_0, \ldots, f_r$ enjoy the latter property, and solve~\eqref{eq:ideal_eq}.
\end{proof}

\section{Global hypoellipticity}
\label{sec:hypo}

The purpose of this section is to prove Theorem~\ref{thm:hypo_mainthm}. The direct implication was proved by Meziani~\cite[Theorem~3.1]{meziani02}: if $\zeta$ is rational or Liouville then $\dd'_0$ is not globally hypoelliptic. Our goal is to supplement his results by proving the following:
\begin{Thm}
  \label{thm:hypo_suff}
  If $\zeta$ is irrational and non-Liouville then $\dd'_0$ is globally hypoelliptic.
\end{Thm}
We shall apply the constructions used to prove solvability. We have the standard partial de Rham complex on $\R \times \widehat{L}$ acting on currents:
\begin{equation*}
  \begin{tikzcd}
    \D'(\R \times \widehat{L}) \arrow[r, "\dd_{\widehat{L}}"] &
    \Lambda^1_{\widehat{\VV}} (\R \times \widehat{L}; \D') \arrow[r, "\dd_{\widehat{L}}"] &
    \cdots                                                     \arrow[r, "\dd_{\widehat{L}}"] &
    \Lambda^n_{\widehat{\VV}} (\R \times \widehat{L}; \D')
  \end{tikzcd} .
\end{equation*}
The group of periods $\per(\zeta)$ acts on these spaces via pullback in a similar way, and we can regard them as $\per(\zeta)$-modules. The first lemma we prove concerns \emph{almost global hypoellipticity} (in the sense of \cite{afr22}) for $\dd_{\widehat{L}}$:
\begin{Lem}
  \label{lem:agh_partialderham}
  If $u \in \D'(\R \times \widehat{L})$ is such that $\dd_{\widehat{L}} u$ is smooth then there exists $v \in \cinfty(\R \times \widehat{L})$ such that $\dd_{\widehat{L}} v = \dd_{\widehat{L}}u$.
\end{Lem}
\begin{proof}
  Let $\{ (U_j ,\chi_j) \}$ be an atlas for $\widehat{L}$ such that $V_j \dfn \chi_j(U_j) \sset \R^n$ is an open convex neighborhood of the origin for each $j$. Therefore, $\Phi_j: \R \times U_j \to \R \times V_j$ given by $\Phi_j(t,x) \dfn (t,\chi_j(x))$ form an atlas for $\R \times \widehat{L}$: the datum of a distribution $u \in \D'(\R \times \widehat{L})$ is encoded in a family of distributions $u_j \in \D'(\R \times V_j)$ satisfying
  \begin{equation}
    \label{eq:compatibility}
    u_k|_{\R \times \chi_k(U_j \cap U_k)} = \left(\Phi_j \circ \Phi_k^{-1} \right)^* u_j|_{\R \times \chi_j(U_j \cap U_k)}.
  \end{equation}
  Assume that $\dd_{\widehat{L}}u$ is smooth. This means that for all $j$ we have
  \begin{equation*}
    \frac{\del u_j}{\del x_l} \in \cinfty(\R \times V_j),
    \quad 1 \leq l \leq n,
  \end{equation*}
  which implies that $u_j$ is smooth with respect to $x$, in the following sense. Since $V_j$ is convex, we can define the smooth function
  \begin{equation*}
    \Psi_j(t,x) \dfn \sum_{l=1}^n x_l\int_0^{1} \frac{\del u_j}{\del x_l} (t,sx) \dd s \in \cinfty(\R \times V_j).
  \end{equation*}
  We claim that $\del \Psi_j/\del x_k = \del u_j/\del x_k$. Indeed, note that
  \begin{equation*}
    \frac{\del}{\del s} \left( \frac{\del u_j}{\del x_k}(t, sx) \right)
    = \sum_{l=1}^n x_l \frac{\del^2 u_j}{\del x_l \del x_k}(t,sx)
    = \sum_{l=1}^n x_l \frac{\del^2 u_j}{\del x_k \del x_l}(t,sx)
\end{equation*}
since the mixed derivatives match: this is true in the distributional sense and both functions are smooth. Therefore,
  \begin{align*}
    \frac{\del \Psi_j}{\del x_k}(t,x)
    &= \int_{0}^1 \frac{\del u_j}{\del x_k}(t,sx) \dd s + \sum_{l=1}^n x_l \int_0^{1} s \frac{\del^2 u}{\del x_k \del x_l}(t,sx) \dd s \\
    &= \int_{0}^1 \frac{\del u_j}{\del x_k}(t,sx) \dd s + \int_{0}^1 s \frac{\del}{\del s} \left( \frac{\del u_j}{\del x_k}(t, sx) \right) \dd s \\
    &= \int_{0}^1 \frac{\del u_j}{\del x_k}(t,sx) \dd s + \frac{\del u_j}{\del x_k}(t, x) - \int_0^1 \frac{\del u_j}{\del x_k}(t, sx) \dd s \\
    &= \frac{\del u_j}{\del x_k}(t, x)
  \end{align*}
  where we integrated by parts. Therefore, $u^\bb \dfn u_j-\Psi_j$ is a distribution in $\R \times V_j$ such that $\del u^\bb/\del x_l = 0$ for all $l=1,\ldots,n$. By~\cite[Theorem~3.1.4']{hormander_alpdo1}, there is a distribution $u^0_j \in \D'(\R)$ such that
  \begin{equation*}
    \langle u^\bb ,\phi \rangle = \int_{V_j}\langle u^0_j, \phi(\cdot,x) \rangle \dd x,
    \quad \forall \phi \in \cinfty_c(\R \times V_j).
  \end{equation*}
  Given $x \in V_j$, define
  \begin{equation*}
    u_j(x) \dfn \Psi_j(\cdot,x) + u^0_j \in \D'(\R).
  \end{equation*}
  The map $V_j \ni x \mapsto u_j(x) \in \D'(\R)$ is smooth, in the sense that $V_j \ni x \mapsto \langle u_j(x),\psi \rangle$ is smooth for all $\psi \in \cinfty_c(\R)$, and clearly
  \begin{equation}
    \label{eq:dist_formula}
    \langle u_j, \phi \rangle = \int_{V_j} \left \langle u_j(x), \phi(\cdot,x) \right \rangle \dd x,
    \quad \forall \phi \in \cinfty(\R \times V_j).
  \end{equation}
  We introduce mollifiers $\rho_\varepsilon(t) \dfn \varepsilon^{-1}\rho(t/\varepsilon)$ where $\rho \in \cinfty_c(\R)$ is non-negative with $\int \rho = 1$. Consider, for fixed $x \in V_j$, the convolution
  \begin{equation*}
    u^\varepsilon_j(x) \dfn u_j(x) * \rho_\varepsilon \in \cinfty(\R),
  \end{equation*}
  which is the following function:
  \begin{equation*}
    \left(u^\varepsilon_j(x)\right) (t) = \left \langle u_j(x), \rho_\varepsilon(t-\cdot) \right \rangle,
    \quad t \in \R.
  \end{equation*}
  Let now
  \begin{equation*}
    u^\varepsilon_j(t,x) \dfn \left( u^\varepsilon_j(x) \right)(t) = \left( \Psi_j(\cdot,x) * \rho_{\varepsilon} \right)(t) + \left( u^0_j * \rho_\varepsilon \right)(t),
  \end{equation*}
  which is a smooth function in $\R \times V_j$. Note that $u^\varepsilon_j \to u_j$ in the weak topology of $\D'(\R \times V_j)$. We shall show that the functions $u^\varepsilon_j$ glue to a global object $u^\varepsilon \in \cinfty(\R \times \widehat{L})$. Indeed, we have to show that $u^\varepsilon_j \circ \Phi_j = u^\varepsilon_k \circ \Phi_k$ on $\R \times U_j \cap U_k$; or, equivalently,
  \begin{equation}
    \label{eq:transition_varepsilon}
    u^\varepsilon_k(t,x) = u^\varepsilon_j \left(\Phi_j\circ \Phi_k^{-1}(t,x) \right),
    \quad \forall (t,x) \in \R \times \chi_k(U_j \cap U_k).
  \end{equation}

  From~\eqref{eq:compatibility}, we deduce that, for all $\phi \in \cinfty_c \left( \R \times \phi_k(U_j\cap U_k) \right)$,
  \begin{equation*}
    \langle u_k,\phi \rangle
    = \left \langle \left(\Phi_j \circ \Phi_k^{-1} \right)^*u_j, \phi \right \rangle 
    = \left \langle u_j, \frac{1}{\left|\det \mathrm{D} (\chi_j \circ \chi_k^{-1}) \right|} \phi \circ \left(\Phi_k \circ \Phi_j^{-1} \right) \right \rangle.
  \end{equation*}
  Applying~\eqref{eq:dist_formula}, we conclude that
  \begin{align*}
    \int_{V_k} \left \langle u_k(x), \phi(\cdot,x) \right \rangle \dd x
    &= \int_{V_j}\left \langle u_j(y), \phi(\cdot, \chi_k\circ \chi_j^{-1}(y)) \right \rangle \frac{\dd y}{\left|\det \mathrm{D} (\chi_j \circ \chi_k^{-1}) \right|} \\
    &=\int_{V_k} \left \langle u_j(\chi_j \circ \chi_k^{-1}(x)), \phi(\cdot, x) \right \rangle \dd x
  \end{align*}
  by change of variables.  Applying the latter to test functions of type $\phi(t,x) = \psi(t) \phi_2(x)$ with $\phi_2 \in \cinfty_c \left( \chi_k(U_j\cap U_k) \right)$, we obtain
  \begin{equation*}
    \int_{V_k} \left(\langle u_k(x), \psi \rangle - \langle u_j(\chi_j\circ \chi_k^{-1}(x)), \psi \rangle \right) \phi_2(x) \dd x = 0,
  \end{equation*}
  which in turn entails by usual arguments that
  \begin{equation*}
    \langle u_k(x), \psi \rangle = \langle u_j \left( \chi_j\circ \chi_k^{-1}(x) \right), \psi \rangle,
    \quad  \forall x \in \chi_k(U_j\cap U_k), \ \forall \psi \in \cinfty_c(\R).
  \end{equation*}
  Choosing $\psi(s) \dfn \rho_\varepsilon(t-s)$ we obtain~\eqref{eq:transition_varepsilon}.

  Therefore, we have $u^\varepsilon \in \cinfty(\R \times \widehat{L})$: if we prove that $\dd_{\widehat{L}}u^\varepsilon \to \dd_{\widehat{L}}u$ in $\Lambda^1_{\widehat{\VV}} (\R \times \widehat{L})$ as $\varepsilon \to 0$, then the closedness of the range of $\dd_{\widehat{L}}$ entails the result. To show this, it is enough to prove that
  \begin{equation*}
    \frac{\del u^{\varepsilon}_j}{\del x_l} \lra \frac{\del u_j}{\del x_l}
    \quad \text{in $\cinfty(\R \times V_j)$ for all $l=1,\ldots,n$}.
  \end{equation*}
  Note that
  \begin{align*}
    \frac{\del u^\varepsilon_j}{\del x_l}
    &= \frac{\del}{\del x_l} \left \langle u_j(x), \rho_{\varepsilon}(t-\cdot) \right \rangle  \\
    &=\frac{\del}{\del x_l} \left( \int_{\R} \Psi_j(s,x) \rho_\varepsilon(t-s)\dd s + \left \langle u^0_j, \rho_{\varepsilon}(t-\cdot) \right \rangle \right) \\
    &=\int_{\R} \frac{\del \Psi_j}{\del x_l}(s,x) \rho_{\varepsilon}(t-s)\dd s \\
    &\lra \frac{\del \Psi_j}{\del x_l} = \frac{\del u_j}{\del x_l}
  \end{align*}
  in the smooth topology, as we desired.
\end{proof}
Now we can proceed with the main result.
\begin{Thm}
  \label{thm:group_coh_hypo}
  Suppose that~\eqref{eq:delta_map} is globally hypoelliptic in the following sense: if $\delta(u) \in \prod_{\alpha \in \per(\zeta)} \cinfty(\R)$ then $u \in \cinfty(\R)$. Then, $\dd'_0$ is~\eqref{eq:GH}.
\end{Thm}
\begin{proof}
  Let $v \in \D'(\Omega)$ be such that $\dd'_0 v \in \Lambda^1_{\VV}(\Omega)$. By Proposition~\ref{prop:pullback_map}, $u \dfn g^*v \in \D'(\R \times \widehat{L})$ is $\per(\zeta)$-invariant, i.e.,~$T_\alpha^*u = u$ for all $\alpha \in \per(\zeta)$, and moreover $\dd_{\widehat{L}}u$ is smooth. From Lemma~\ref{lem:agh_partialderham}, we can find $u^\bb \in \cinfty(\R \times \widehat{L})$ such that $\dd_{\widehat{L}} (u - u^\bb) = 0$, and since $\widehat{L}$ is connected there exists $\phi \in \D'(\R)$ such that $u = u^\bb + \phi$. It follows that $\delta(\phi) = - \delta(u^\bb)$ is smooth, hence so is $\phi$ itself; the smoothness of $u$, and then of $v$, follow by consequence.
\end{proof}
Now, the proof of Theorem~\ref{thm:hypo_suff} will be complete once we show:
\begin{Prop}
  Let $1, a_1, \ldots,a_r$ positive real numbers, linearly independent over $\Q$. If $(a_1, \ldots,a_r)$ is a non-Liouville vector then the operator $\delta$ given by~\eqref{eq:delta_map} is globally hypoelliptic.
\end{Prop}
\begin{proof}
  Take $F_0, \ldots, F_r \in \Ol(\C)$ of exponential type, bounded by a polynomial on $\R$, such that~\eqref{eq:ideal2} holds everywhere: their existence follows from our hypothesis. We rewrite the latter as
  \begin{equation*}
    \sum_{j=0}^r (e^{-ia_j z} - 1) F_j(z) = z;
  \end{equation*}
  hence, taking Fourier transform on both sides yields
  \begin{equation*}
    \sum_{j=0}^r \left( T_{a_j}^* u_j - u_j \right) = \delta_{0}'
  \end{equation*}
  for certain $u_0, \ldots, u_r \in \E'(\R)$, where $\delta_{0}'$ denotes the derivative of the Dirac mass centered at the origin. Let $u\in \D'(\R)$ be such that $\delta(u)$ is smooth. Then its derivative is also smooth, since it equals
  \begin{equation*}
    u' = u' * \delta_0 = u * \delta'_0
    = \sum_{j=0}^r u * \left( T_{a_j}^* u_j - u_j \right)
    = \sum_{j=0}^r \left( T_{a_j}^* u - u \right) * u_j
  \end{equation*}
  for the convolution commutes with translations. Hence, $u$ is also smooth.
\end{proof}

\section{Acknowledgments}
The first-named author was supported by Sao Paulo Research Foundation (FAPESP), grants 2018/14316-3, 2021/03888-9, 2024/08416-6, 2024/12753-8, and by National Council for Scientific and Technological Development (CNPq), grants 163837/2022-8, 404175/2023-6, 170244/2023-7, 303483/2024-5. The second-named author was supported by Sao Paulo Research Foundation (FAPESP), grants 2024/08416-6, 2024/12753-8 and by National Council for Scientific and Technological Development~(CNPq), grant 313581/2021-5. The third-named author was supported by Sao Paulo Research Foundation (FAPESP), grant 2021/03888-9, 2021/03199-9 and by National Council for Scientific and Technological Development (CNPq), grant 404175/2023-6. The fourth-named author was supported by Sao Paulo Research Foundation, grants 2021/03888-9, 2023/17607-7 and by the Austrian Science Fund (FWF), grant 10.55776/PAT3705925.

\bibliographystyle{plain}
\bibliography{global1form}
\end{document}